\pdfoutput=1
\documentclass{amsart}

\usepackage{empheq}
\usepackage{color}
\usepackage[colorlinks,linkcolor=blue,anchorcolor=blue,citecolor=blue]{hyperref}
\usepackage{amssymb,amstext, amsbsy, amscd}
\usepackage[mathscr]{eucal}
\usepackage{times}
\usepackage{amsmath, amsthm, amsfonts,extarrows}
\usepackage{multirow}
\usepackage{booktabs}
\usepackage[all]{xy}

\usepackage{tikz}
\usepackage{amsxtra}

\usepackage{enumitem}

\newtheorem{thm}{Theorem}[section]
\newtheorem{prop}[thm]{Proposition}
\newtheorem{lemma}[thm]{Lemma}
\newtheorem{cor}[thm]{Corollary}

\newtheorem{conjecture}[thm]{Conjecture}
\newtheorem{preconjecture}[thm]{Pre-Conjecture}

\newtheorem{prop-defn}[thm]{Proposition/Definition}

\theoremstyle{definition}
\newtheorem{defn}[thm]{Definition}
\newtheorem{predefn}[thm]{Pre-Definition}
\newtheorem{example}[thm]{Example}
\newtheorem{remark}[thm]{Remark}
\newtheorem{prob}[thm]{Problem}
\newtheorem{question}[thm]{Question}

\newlist{casesp}{enumerate}{3} 
\setlist[casesp]{align=left, 
                 listparindent=\parindent, 
                 parsep=\parskip, 
                 font=\normalfont\bfseries, 
                 leftmargin=0pt, 
                 labelwidth=0pt, 
                 itemindent=.4em,labelsep=.4em, 
                 partopsep=0pt, 
                 }
\setlist[casesp,1]{label=Case~\Roman*:,ref=\Roman*}
\setlist[casesp,2]{label=Case~\thecasespi.\arabic*:,ref=\thecasespi.\arabic*}
\setlist[casesp,3]{label=Case~\thecasespii.\alph*:,ref=\thecasespii.\alph*}

\newcommand{\bC}{{\mathbb C}}

\newcommand{\bP}{{\mathbb P}}

\newcommand{\bR}{{\mathbb R}}
\newcommand{\bZ}{{\mathbb Z}}

\newcommand{\cO}{{\mathcal O}}
\newcommand{\cF}{{\mathcal F}}

  \newcommand{\<}{\langle}
  \renewcommand{\>}{\rangle}

\newcommand{\cM}{\mathcal{M}}
\newcommand{\ovcM}{\overline{\cM}}
\newcommand{\cC}{\mathcal{C}}
\newcommand{\cE}{\mathcal{E}}
\newcommand{\cA}{\mathcal{A}} 

\newcommand{\bN}{\mathbb{N}}
\newcommand{\hGamma}{\widehat{\Gamma}}
\newcommand{\tGamma}{\widetilde{\Gamma}}

\newcommand{\hG}{\widehat{G}} 

\newcommand{\bfd}{{\mathbf{d}}}
\newcommand{\bq}{{\mathbf{q}}}
\newcommand{\bx}{{\mathbf{x}}}

\newcommand{\frakm}{\mathfrak{m}}
\newcommand{\frq}{{\mathfrak{q}}}
\newcommand{\frf}{{\mathfrak{f}}}

\newcommand{\Hom}{\operatorname{Hom}}
\newcommand{\Ext}{\operatorname{Ext}}
\newcommand{\End}{\operatorname{End}}
\newcommand{\NEN}{\operatorname{NE_{\bN}}}
\newcommand{\ovNE}{\operatorname{\overline{NE}}}
\newcommand{\Spec}{\operatorname{Spec}}
\newcommand{\Jac}{\operatorname{Jac}}

\renewcommand{\Re}{\operatorname{Re}}
\renewcommand{\Im}{\operatorname{Im}}

\newcommand{\Bri}{\operatorname{Bri}}
\newcommand{\Mir}{\operatorname{Mir}}
\newcommand{\Lef}{\operatorname{Lef}}
\newcommand{\ch}{\operatorname{ch}}
\newcommand{\Ch}{\operatorname{Ch}}
\newcommand{\rank}{\operatorname{rank}}
\newcommand{\Cone}{\operatorname{Cone}}
\newcommand{\pt}{\operatorname{pt}}
\newcommand{\ev}{\operatorname{ev}}
\newcommand{\Const}{\operatorname{Cst}}
\newcommand{\Hess}{\operatorname{Hess}} 
\newcommand{\Log}{\operatorname{Log}} 
 
\newcommand{\Trop}{\operatorname{Trop}} 

\newcommand{\iu}{{\mathbf{i}}}

\newcommand{\frs}{\mathfrak{s}}

\newcommand{\sfT}{\mathsf{T}}

\newcommand{\TAcon}{T_{\text{A,con}}}

\newcommand{\parfrac}[2]{\frac{\partial #1}{\partial #2}}

\theoremstyle{plain}
\newtheorem*{conjO}{Conjecture $\cO$}
\newtheorem*{GammaI}{Gamma conjecture I}

\begin{document}

{\allowdisplaybreaks[4]

\title{Revisiting Gamma conjecture I: counterexamples and modifications}

\author[Galkin]{Sergey Galkin}
\address{PUC-Rio, Departamento de Matem\'{a}tica, Rua Marqu\^{e}s de S\~{a}o Vicente 225, G\'{a}vea, Rio de Janeiro}
\email{sergey@puc-rio.br}
\address{Hetao Institute of Mathematics and Interdisciplinary Sciences (HIMIS), Shenzhen}
\email{gamma-1-mod@galkin.org.ru}
\address{Zhizhen School of Mathematics and Interdisciplinary Sciences (ZSIMS), CUHK, Hong Kong}
\email{sergeygalkin@cuhk.edu.hk}
\thanks{
 }
\author[Hu]{Jianxun Hu}
\address{School of Mathematics, Sun Yat-sen University, Guangzhou 510275, P.R. China}

\email{stsjxhu@mail.sysu.edu.cn}
\thanks{
 }

\author[Iritani]{Hiroshi Iritani}
\address{Department of Mathematics, Kyoto University, Kitashirakawa-Oiwake-cho, Sakyo-ku, Kyoto, Japan}

\email{iritani@math.kyoto-u.ac.jp}

\thanks{
 }
\author[Ke]{Hua-Zhong Ke}
\address{School of Mathematics, Sun Yat-sen University, Guangzhou 510275, P.R. China}
\email{kehuazh@mail.sysu.edu.cn}
\thanks{ 
 }

\author[Li]{Changzheng Li}
 \address{School of Mathematics, Sun Yat-sen University, Guangzhou 510275, P.R. China}
\email{lichangzh@mail.sysu.edu.cn}

\author[Su]{Zhitong Su}
\address{MOE-LCSM, School of Mathematics and Statistics, Hunan Normal University, Changsha 410081, P. R. China}
\email{suzht@hunnu.edu.cn}
\thanks{ 
 }

\date{
      }




\begin{abstract}
We continue investigation of asymptotics of quantum differential equation for Fano manifolds, with a special regard to Gamma conjecture I and its underlying Conjecture $\mathcal{O}$.
We introduce the A-model conifold value, a symplectic invariant of a Fano manifold, and propose modifications for Gamma conjecture I based on this new definition. In particular, the modified Gamma conjecture I is proved for all toric  Fano manifolds.
We also discuss an interplay of birational transformations with an extension of Gamma conjecture I over the K\"ahler moduli space.
These heuristics are applied to rigorously identify the principal asymptotic class
for the $\mathbb{P}^1$-bundles  $X_n=\mathbb{P}_{\mathbb{P}^{n}}(\mathcal{O}\oplus\mathcal{O}(n))$.
Notably, we observe that for $X_n$ of dimension at least four, Conjecture $\mathcal{O}$ holds only for even values of $n$, and in these cases we falsify the original unmodified Gamma conjecture I.
\end{abstract}

\maketitle

\setcounter{tocdepth}{1}
\tableofcontents

\section{Introduction}


{

\subsection{Conjecture $\cO$ and Gamma conjecture I}
The quantum cohomology ring $QH(X)=(H^*(X)\otimes \mathbb{C}[\mathbf{q}], \star)$ of a Fano manifold $X$ is a  deformation
 of the classical cohomology ring $H^*(X)=H^*(X,\mathbb{C})$ by incorporating genus-zero, three-pointed Gromov-Witten invariants of $X$.
 The quantum multiplication by the first Chern class $c_1(X)$ of $X$ induces a linear operator $\hat c_1=c_1(X)\star_{\mathbf{q}=\mathbf{1}}$ on the even part  $H^\bullet(X):=H^{\rm ev}(X)=QH^{\rm ev}(X)|_{\mathbf{q}=\mathbf{1}}$, which is a vector space of finite dimension.
 In \cite{GGI}, Galkin, Golyshev and Iritani introduced \textit{Property $\mathcal{O}$} (Definition \ref{defnrho}) which says
 \begin{itemize}
\item[(1)] the spectral radius $\rho = \max\{|u| : u\in \Spec(\hat{c}_1)\}$ is itself an eigenvalue of $\hat{c}_1$ whose multiplicity as a root of the characteristic polynomial is one;
\item[(2)] if $u\in \Spec(\hat{c}_1)$ satisfies $|u|=\rho$, then $u=e^{2\pi \iu k/i_X} \rho$ for some $k\in \bZ$,
 \end{itemize}
 where $\Spec(\hat{c}_1)$ is the set of eigenvalues of $\hat{c}_1$ and $i_X$ is the Fano index of $X$.
 They also made the following conjecture with the name $\mathcal{O}$, indicating the relationship between the eigenvalue $\rho$ and the structure sheaf $\cO$ via homological mirror symmetry.
\begin{conjO}[Conjecture 3.1.2 of \cite{GGI}]
   Every Fano manifold satisfies   Property $\mathcal{O}$.
\end{conjO}
\noindent It is an important problem to study the  distribution of eigenvalues of $\hat c_1$. For instance, the monotone Fukaya category of $X$ decomposes into components indexed by eigenvalues of $\hat{c}_1$ (see e.g.~\cite{Auroux,She}). 

 The Gamma class $\widehat \Gamma_X$
    \cite{HKTY, Libg, Iri1, Lu}  of $X$ is a  real characteristic class, and  can  be regarded as a ``square root'' of the Todd class of $X$.
    It is  defined by   the Chern roots   of the tangent bundle  $TX$ of $X$ and
  Euler's $\Gamma$-function, and  has the following expansion:
  $$\widehat \Gamma_X =\exp\big(-C_{\rm eu}c_1(X)+\sum_{k=2}^\infty (-1)^k(k-1)!\zeta(k)\ch_k(TX)\big)  \in H^* (X, \mathbb{R})$$
where $C_{\rm eu}=\lim_{n\to \infty}{(-\log n+\sum_{k=1}^n{1\over k})}$ is the Euler-Mascheroni   constant, $\zeta(k)=\sum_{n=1}^\infty{1\over n^k}$ is the value of the Riemann zeta function at $k$, and  $\ch_k$ denotes the $k$-th Chern character.
 Gamma conjectures I and II were both proposed in \cite{GGI}, which relate the quantum cohomology of $X$ and the Gamma class $\widehat \Gamma_X$ in terms of quantum differential equations. In this paper, we will focus on Gamma conjecture I, of which \textit{Property $\mathcal{O}$} is an underlying assumption.

There is a quantum connection,
  $\nabla_{z\partial_z}=z\partial_z-{1\over z} \hat c_1+\mu$,
defined on $H^\bullet(X)\otimes \mathbb{C}[z, z^{-1}]$.  It has a regular singularity at $z=\infty$ and an irregular
singularity at $z=0$.
Near $z=\infty$, there is a canonical correspondence between $H^*(X)$ and the space of flat sections \cite{Giv96}, and near $z=0$, the space of  flat sections (along a sector) has a natural filtration by  exponential growth order.  It is shown in \cite[Proposition 3.3.1]{GGI} that
the vector space of flat sections $s(z)$ with the smallest asymptotics as $z\to 0$ along $\mathbb{R}_{>0}$ is one-dimensional, provided that Property $\mathcal{O}$ holds. Such one-dimensional space $\mathbb{C}s(z)$ near $z=\infty$ corresponds to a class $A_X\in H^*(X)$ up to a scalar, called \textit{the principal asymptotic class} of $X$.
\begin{GammaI}[\protect{\cite[Conjecture 3.4.3]{GGI}}]
Let $X$ be a Fano manifold satisfying Property $\mathcal{O}$\footnote{In this paper, we relax this property to condition \eqref{eq:condition_star}. See Section \ref{subsubsec:Gamma-I}. }. 
The principal asymptotic class $A_X$ is given by $\widehat \Gamma_X$.
\end{GammaI}

The principal asymptotic class $A_X$ also appears in an asymptotic expansion of Givental's $J$-function $J_X(\tau,z)$ restricted to the anti-canonical line $\tau\in \bR c_1(X)$. As explained in \cite[Proposition 3.8]{GaIr},  $J_X(c_1(X) \log t,1)$ admits an asymptotic expansion at $t={1\over z}\to +\infty$ with leading term $A_X$, provided that Property $\mathcal{O}$ holds. As a consequence in \cite[Corollary 3.6.9]{GGI}, Gamma conjecture I is equivalent to the limit formula:
\begin{equation}
\label{eq:J_limit_formula_projective_space}
\lim_{t\to +\infty} [J_X(c_1(X) \log t, 1)] =[\hGamma_X]
\end{equation}
in the projective space $\bP(H^\bullet(X))$.

 Gamma conjecture I, together with its underlying Conjecture $\mathcal{O}$,  has been proved for complex Grassmannians \cite{GGI}, Fano threefolds of Picard number one \cite{GoZa}, Fano complete intersections in projective spaces  \cite{GaIr, SaSh, Ke},   del Pezzo surfaces \cite{HKLY} and flag varieties $G/P$  by combining \cite{ChLi} and \cite{Chow}.
 Conjecture $\mathcal{O}$ has also been proved for a few more cases    \cite{GaGo, Che, LMS, BFSS, With, HKLS}.

\subsection{Gamma conjecture and mirror symmetry}
 Gamma conjectures and mirror symmetry are closely related. There are recent studies of mirror symmetry version of  Gamma conjectures for Fano manifolds and Calabi-Yau manifolds  \cite{AGIS, Wang, FWZ, Iri23}. The case for Calabi-Yau manifolds originates from  \cite[Conjecture 2.2]{Hoso} by Hosono.
There have been extensive studies of mirror symmetry for toric  Fano manifolds in \cite{Baty,   Gi2,   HoVa, ChOh, FOOO, Iri1, Abou, FLTZ} and by many others\footnote{See also the surveys \cite{Iri20, Chan20} and references therein.}.
Roughly speaking, mirror symmetry asserts equivalences between geometric information on the so-called A-side and B-side. For a toric Fano manifold $X_\Delta$ of dimension $N$  on the A-side,   the mirror object on  the B-side is a Landau-Ginzburg model $(\check X, f_\Delta)$, where $f_\Delta\colon \check X=(\mathbb{C}^\times)^N\to \mathbb{C}$ is a Laurent polynomial called the Landau-Ginzburg superpotential.
Mirror symmetric Gamma conjecture for $X_\Delta$ focuses on the matching between the Gamma-integral structure of quantum cohomology of $X_\Delta$ and the natural integral structure of oscillatory integrals associated with $f_\Delta$, and has been verified in \cite{Iri1}. Building on this result, \cite[Theorem 6.3]{GaIr} showed that the (non-mirror-symmetric) Gamma conjecture I holds for $X_\Delta$, provided that $f_\Delta$ satisfies a \emph{B-analogue of Property $\mathcal{O}$} as defined in \cite[Condition 6.1]{GaIr}.


It is therefore  natural to investigate the B-analogue of Conjecture $\cO$ in the study of Gamma conjecture I for $X_\Delta$. To explain the condition, let us start with some facts in mirror symmetry. It was shown in \cite{Baty, Gi2} (see also \cite{ChLe}) that the quantum cohomology ring $QH(X_\Delta)$ is isomorphic to the Jacobi ring $\Jac(f_\Delta)$. 
Moreover, it was proved in \cite{Auroux} (see also \cite{OsTy}) that the eigenvalues of $\hat c_1$ (with multiplicities counted) coincide with the critical values of the superpotential $f_\Delta$. The B-analogue of Property $\mathcal{O}$ (Definition \ref{BanalogyPO}) says:
\begin{itemize}
\item[(a)] every critical value $u$ of $f_\Delta$ satisfies $|u|\le T_{\text{con}}$, where $T_{\text{con}}$ is the value of $f_\Delta$ at the conifold point $\bx_{\text{con}}$ (a unique critical point of $f_\Delta|_{(\bR_{>0})^N}$ \cite{Gal});
\item[(b)] $\bx_{\text{con}}$ is a unique critical point of $f_\Delta$ contained in $f_\Delta^{-1}(T_{\text{con}})$.
\end{itemize}
When $X_\Delta$ is a blow-up of $\mathbb{P}^n$ along a linear subspace $\mathbb{P}^{r-1}$, the B-analogue of Property $\mathcal{O}$ was verified in \cite{Yang} by using classical complex analysis.

\subsection{Counter-examples to Conjecture $\cO$ and Gamma conjecture I}

As the first highlight of  this paper, we investigate  toric Fano manifolds of Picard number two and provide counterexamples of both conjectures in the following,   which summarizes \textbf{Theorems \ref{thmconjO1}, \ref{thmGC1nothold} and \ref{thmVGC}}.
\begin{thm}\label{mainthm} Let $X_n= \mathbb{P}_{\mathbb{P}^{n}}(\mathcal{O}\oplus\mathcal{O}(n))$, where $n\geq 3$.
\begin{enumerate}
   \item If $n$ is even, then   Conjecture $\mathcal{O}$ holds for $X_n$, while Gamma conjecture I does not hold for $X_n$.
    \item If $n$ is odd, then Conjecture $\mathcal{O}$ does not hold for $X_n$, while Gamma conjecture I holds for $X_n$.
\end{enumerate}
\end{thm}
\noindent We note that both Conjecture $\mathcal{O}$ and Gamma conjecture I hold for $n\in\{1, 2\}$. When $n=1$, $X$ is a blow-up of $\mathbb{P}^2$ at one point, and both conjectures have been verified in \cite{HKLY}. The case for $n=2$ will  be verified in Sections \ref{sec:conjecture_O_Fano_toric} and \ref{sec:Gamma-I}.
For Conjecture $\mathcal{O}$, we use mirror symmetry and  reduce the analysis of the critical values of the mirror superpotential to optimization problems in  nonlinear programming.
Such reduction seems effective: it also works for $\mathbb{P}_{\mathbb{P}^{n}}(\mathcal{O}\oplus\mathcal{O}(n-1))$ and a few more examples as we discuss in the appendix.  This method is also applied in \cite{HKLSu} to find a counterexample of the lower bound conjecture in \cite{Gal}.

While $X_n$ with odd $n\ge 3$ does not satisfy Conjecture $\cO$, it satisfies the weaker condition \eqref{eq:condition_star} below and Gamma conjecture I makes sense for it. 
We say that a linear operator $A$ on a finite-dimensional vector space has a \emph{simple rightmost eigenvalue $u\in \bC$} if (1) $u$ is an eigenvalue of $A$ whose multiplicity as a root of the characteristic polynomial is one and (2) any other eigenvalue $u'$ of $A$ satisfies $\Re(u')<\Re(u)$. We consider the following condition (Definition \ref{defn:star}):
\begin{equation}
\label{eq:condition_star}
\tag{$\mathrm{SR}$} \text{the operator $\hat{c}_1$ has a simple rightmost eigenvalue $\rho'$.}
\end{equation}
Since $\hat{c}_1$ is a real operator, the simple rightmost eigenvalue $\rho'$ (if it exists) is a real number and is given by $\max\{\Re(u):u\in\Spec(\hat{c}_1)\}$. 
As explained in \cite[Remark 3.1.9]{GGI}, when condition \eqref{eq:condition_star} holds, the space of $\nabla$-flat sections with the smallest asymptotics as $z\to +0$ is one-dimensional, and the principal asymptotic class can still be defined. Hence Gamma conjecture I makes sense under condition \eqref{eq:condition_star}. 

\begin{remark}
Hugtenburg \cite{Hug24} recently studied Gamma conjecture I for certain non-K\"ahler monotone symplectic manifolds constructed by Fine-Panov \cite{Fine-Panov}, based on the calculation of quantum cohomology by Evans \cite{Evans:qcoh_twistor}. In those examples, the spectral radius $\rho$ of $\hat{c}_1$ is a non-simple eigenvalue of $\hat{c}_1$ and Conjecture $\cO$ fails. However, Gamma conjecture I still holds in the sense that the limit formula \eqref{eq:J_limit_formula_projective_space} is true.
\end{remark}

\subsection{Modifications of Gamma conjecture I} 
The counterexample $X_n$ shows that the original statement of Gamma conjecture I has to be revised. 
As the second highlight of this paper, we propose the modifications of Gamma conjecture I,  based on  the notion of \textit{A-model conifold value} $T_{\rm A, con}\in \mathbb{R}_{\geq 0}$ in \textbf{Definition 
\ref{defn:A-model_conifold_value}}, which we are about to introduce and emerged from our discussions with Kai Hugtenburg.    

Comparing Theorem \ref{mainthm} with \cite[Theorem 6.3]{GaIr}, we   see that the conifold value $T_{\rm con}$  plays a crucial role in the case of toric Fano manifolds, which is defined by using  the mirror superpotential $f_\Delta$. It is still unknown how to  construct a  mirror Landau-Ginzburg model for a general Fano manifold. Nevertheless,  we can consider the \emph{quantum period} $G_X(t)$  of a Fano manifold $X$,  defined to be the pairing of the $J$-function and the point class $[\pt] \in H^{\text{top}}(X)$ \cite{CCGGK:mirrorsymmetry}:
\[
G_X(t) := \left\langle J_X(c_1(X)\log t,1), [\pt] \right\rangle^X = \sum_{n=0}{a_n\over n!}t^n.   
\]
We then define 
$$\TAcon := \limsup_{n\to \infty} |a_n|^{1/n},$$
which is the inverse of the convergence radius of $\sum_{n=0}^\infty a_nt^n$. 
If $X$ admits a Laurent polynomial mirror $f\in \bC[x_1^\pm,\dots,x_m^\pm]$, it is expected that $G_X(t)$ coincides with the series $\sum_{n\ge 0} \Const(f^n) t^n/n!$, where $\Const(f^n)$ is the constant term of $f^n$. Conversely, if this equality holds, we say that $f$ is a \emph{weak Landau-Ginzburg model} \cite{Przyjalkowski:weak} of $X$. 

We have the following theorem, from the combination of \textbf{Proposition \ref{prop: TAconeigen}}, \textbf{Theorem \ref{thm:nonnegative_Laurent}}  and \textbf{Corollary \ref{cor:TAcon_Tcon}}. See Section \ref{subsec:TAcon} for the relevant notions.
\begin{thm}\label{thm: TAcon_Tcon}
Suppose that a Fano manifold X admits a convenient weak Landau-Ginzburg model
$f$ with nonnegative coefficients. Then the A-model conifold value $\TAcon$ of $X$ is an eigenvalue of $\hat c_1$, and it  coincides
with the (B-model) conifold value $T_{\rm con}=\min_{\bx\in (\bR_{>0})^m} f(\bx)$ of $f$.
\end{thm}
\noindent We notice that the hypothesis in the above theorem is satisfied by   all   toric Fano manifolds and  Fano threefolds \cite{CCGGK:quantum period}. More generally, it is satisfied by any Fano manifold that admits a  $\mathbb{Q}\Gamma$ (e.g. $\mathbb{Q}$-Gorenstein) degeneration to a toric variety with at worst terminal singularities \cite{GaMi, Sanda:toric_degeneration, Ton}, such as partial flag manifolds. This provides strong evidence  for  the following conjecture. 
\begin{conjecture}\label{conj: TAcon}
    For any Fano manifold $X$, the A-model conifold value $\TAcon$ is an eigenvalue of the linear operator $\hat c_1$. 
\end{conjecture}

 \noindent By Remark \ref{rmk:GammaIcomp}, this conjecture also holds  for all the aforementioned cases for which  Gamma conjecture I holds. We will introduce the notion of Property $\cO_A$ in \textbf{Definition \ref{defnPOA}} that says that $\TAcon$ is a simple rightmost eigenvalue of $\hat{c}_1$. This is a combination of condition \eqref{eq:condition_star} and part (b) of the B-analogue of Property $\cO$. 
  Let us formulate a modification of Gamma conjecture I below, which will be restated in \textbf{Conjecture \ref{conj:modified_Gamma-I_weak}}. 
\begin{conjecture}[Modified Gamma conjecture I: weak form] 
\label{conj:modified_Gamma-I_weak} 
Suppose that a Fano manifold $X$ satisfies Property $\cO_A$. Then $X$ satisfies Gamma conjecture I, i.e~the principal asymptotic class $A_X$ is given by the Gamma class $\hGamma_X$.
\end{conjecture} 

We will also propose a stronger form of the modified Gamma conjecture I that does not require Property $\cO_A$. It implies the weak form, and by Remark \ref{rmk: strong_form} it holds for all toric Fano manifolds. See Conjecture \ref{conj:modified_Gamma-I_strong} for details. 


\begin{remark}\label{rmk:GammaIcomp}
 By \cite[Proposition 3.7.6]{GGI}, the hypotheses (a) Property $\cO$ (1) and (b) $\langle A_X, [{\rm pt}]\rangle\neq 0$ together imply that the convergence radius of $\sum_{n=0}^\infty a_nt^n$ equals $1\over \rho$; consequently, $T_{\rm A, con}=\rho$ and Property $\cO_A$ holds. In particular, whenever the original Gamma conjecture I (together with Property $\cO$) holds for a Fano manifold $X$, the above weak form of modified Gamma conjecture I holds for $X$ as well. See also Proposition \ref{prop:sufficient_condition_PropertyOA}.   
\end{remark}

\begin{remark}
If a Fano manifold $X$ admits a mirror Landau-Ginzburg model $f$ that is a convenient Laurent polynomial with nonnegative coefficients, we can formulate a ``Property $\cO_B$'' which can be seen as a mirror analogue of Property $\cO_A$. Property $\cO_B$ says that the conifold value $T_{\text{con}}=f(\bx_{\text{con}})$ of $f$ is a simple rightmost critical value, i.e.~all critical points $\bx\neq \bx_{\text{con}}$ of $f$ satisfy $\Re(f(\bx)) < \Re(f(\bx_{\text{con}}))$. For toric Fano manifolds and their Hori-Vafa superpotentials $f_\Delta$, Property $\cO_B$ is equivalent to Property $\cO_A$ and is a weaker condition than the previously mentioned B-analogue of Property $\cO$. 


Moreover, we have numerically analyzed the eigenvalues of $\hat c_1$ for all the 8635 toric Fano manifolds of dimension less than 7. As  from Table \ref{tabPO} in Section  5.3,  most of these manifolds satisfy the expected properties.  In total, around $18\%$ of these toric Fano manifolds do not satisfy of Property $\cO$ (1), while there are   fewer than $9\%$ of these toric Fano manifolds do not satisfy  Property $\cO_A$. It would be interesting to explore a geometric/combinatorial characterization of Property $\cO_A$ for  general toric Fano manifolds.  

We also observed from the numerical computations that, for all these lower dimensional toric Fano maniolds $X$, $u^{i_X}=T_{\rm con}^{i_X}$, provided that $u$ is an eigenvalue of $\hat c_1$ with $|u|=T_{\rm con}$. Therefore, one reasonable revision of Conjecture $\cO$ would be the combination of Conjecture \ref{conj: TAcon} and Property $\cO$ (2)  with   {$\rho$} replaced with $\TAcon$.
\end{remark}

\subsection{Exploring Gamma conjecture I over the K\"ahler moduli space}
Motivated by the counterexample, we explore Gamma conjecture I for a general quantum parameter $\bq$ and determine the principal asymptotic class $A_X$. The parameter $\bq$ of quantum cohomology lies in the so-called ``K\"ahler moduli space''
\[
\cM_X := \Hom(H_2(X,\bZ)/{\text{torsion}}, \bC^\times) \cong H^2(X,\bZ)\otimes \bC^\times.
\]
The original Gamma conjecture I focused on the quantum multiplication of $c_1(X)$ at $\bq=\mathbf{1}$, but this specialization lacked a strong theoretical justification.
Exploring Gamma conjecture I for general parameters $\bq \in \cM_X$ offers two advantages: (1) we can connect the behaviour of eigenvalues of the quantum multiplication by $c_1(X)$ with the birational geometry of $X$, and (2) we can sometimes determine the principal asymptotic class through mutation from the asymptotic class near a boundary of the K\"{a}hler moduli space, which is typically easier to determine.

First of all, we generalize the condition \eqref{eq:condition_star} for the operator $\hat c_1= (c_1(X)\star_{\bq=\mathbf{1}})$ on $H^\bullet(X)$ to the condition that $(c_1(X)\star_\bq)$ has a simple rightmost eigenvalue (see Definition \ref{defnsimplerm}).
Assuming the existence of a simple rightmost eigenvalue $u$ for $(c_1(X)\star_\bq)$, we can define the \emph{principal asymptotic class} $A_X(\tau)$ at $\tau\in H^2(X,\bC)$ (Proposition/Definition \ref{pdasymp}), where $\tau$ represents a point on the universal covering $H^2(X,\bC)\to \mathcal{M}_X$ lifting $\bq\in \cM_X$ (which we write as $\bq = e^\tau$). The aforementioned counter-example $X_n$ motivates many questions, including the following in Question \ref{quesAX}:
\begin{itemize}
    \item Does there exist a point $\bq   \in \cM_X$ such that $(c_1(X)\star_\bq)$ has a simple rightmost eigenvalue and   $[A_X(\tau)] = [\hGamma_X]$?
    \item Can we find such a $\tau$ within $H^2(X,\bR)$?
    \item Can we characterize the region in $H^2(X)$ where $[A_X(\tau)] = [\hGamma_X]$ holds?
\end{itemize}
If $[A_X(\tau)] \neq [\hGamma_X]$, we can further ask about the possible forms of $A_X(\tau)$.
Dubrovin's conjecture \cite{Du1998} and Gamma conjecture II \cite{GGI} suggest that there should be an exceptional object $V$ in the derived category $D^b(X)$ of coherent sheaves on $X$ such that
\begin{equation*}
[A_X(\tau)] = [\hGamma_X \Ch(V)]
\end{equation*}
where $\Ch(V) = \sum_{p\ge 0} (2\pi \iu)^p \ch_p(V)$ is the modified Chern character, see Question \ref{question:AX_exceptional}.
The result \cite[\S 4.3]{Iri1} on toric mirror symmetry implies the following result:
\begin{thm}[Theorem \ref{thm:Lefschetz_A}]
Let $X$ be a toric Fano manifold. Suppose that $(c_1(X)\star_\bq)$ has a simple rightmost eigenvalue $u$ for some $\bq =e^\tau \in \cM_X$. Then the principal asymptotic class $A_X(\tau)$ is of the form $\hGamma_X \Ch(V)$ for some $V$ in the topological $K$-group $K^0(X)$.
\end{thm}

The spectrum of the quantum multiplication $(c_1(X)\star_\bq)$ reflects the birational geometry of $X$. When $\bq$ approaches a boundary of $\cM_X$ given by an extremal ray (in the sense of Mori theory), the spectrum of the quantum cohomology of $X$ clusters in a specific way predicted by the corresponding extremal contraction. It is therefore easier to study Gamma conjecture I when $\bq$ is close to such a boundary.

We study this connection to birational geometry in detail for the example $X_n = \bP_{\bP^n}(\cO \oplus \cO(n))$. The space $X_n$ has two extremal contractions: the divisorial contraction $\varphi \colon X_n \to \bP(1,\dots,1,n)$ to the weighted projective space (where $1$ appears $n+1$ times) and the $\bP^1$-fibration $\pi \colon X_n \to \bP^n$.
Each contraction corresponds to an extremal ray of the Mori cone and associates a line in the boundary of $\mathcal{M}_X$. The distributions of eigenvalues of $(c_1(X)\star_\bq)$ near these boundaries are shown in Figures \ref{fig:critical_values_even} and \ref{fig:critical_values_odd}.
As the third highlight of this paper, we describe the principal asymptotic classes near these boundaries   as follows.
\begin{thm}[Theorems \ref{XnasympGamma}, \ref{thm:A_for_Twrong}]
Let $n\ge 2$ be even. Let $\bq$ be a positive real quantum parameter and let $\tau\in H^2(X_n,\bR)$ be the corresponding real class. Then the principal asymptotic class $A_{X_n}(\tau)$ near each boundary is   respectively given  by
\[
\begin{cases}
\hGamma_{X_n} &\text{for $\bq$ near the ``$\bP^1$-fibration boundary''} \\
\hGamma_{X_n} \Ch(\cO_E(-n/2)) & \text{for $\bq$ near the ``divisorial contraction boundary''}
\end{cases}
\]
where $E\cong \bP^n$ is the exceptional divisor of $\varphi$.
Furthermore, if $n$ is even and $n\ge 4$, we have $A_{X_n}(0) = \hGamma_{X_n} \Ch(\cO_E(-n/2))$.
\end{thm}

In this theorem, we derive the principal asymptotic class near the ``divisorial contraction boundary'' from the one near the ``$\bP^1$-fibration boundary'' through a sequence of mutations. We finally remark that   Gamma conjecture I over the K\"ahler moduli has further applications beyond the above investigation. It   leads to the verification of   Gamma conjecture II  for del Pezzo surfaces in \cite{HKLS26}.



 \subsection*{History and organization of the paper}

In the early stages of this project, J.~Hu, H.~Ke, C.~Li, and Z.~Su attempted to prove Conjecture $\mathcal{O}$  and Gamma conjecture I for toric Fano manifolds of Picard number two. However, they discovered counterexamples. Later, S.~Galkin and H.~Iritani joined the project to analyze these counterexamples and identify principal asymptotic classes over the K\"ahler moduli space. 

 This paper is organized as follows. In Section \ref{sec:preliminaries}, we describe the precise statements of Conjecture $\mathcal{O}$ and Gamma conjecture I, and review basic facts of mirror symmetry for Fano manifolds. In Section \ref{sec:conjecture_O_Fano_toric}, {we study  Conjecture $\mathcal{O}$ for  $\mathbb{P}_{\mathbb{P}^{n}}(\mathcal{O}\oplus\mathcal{O}(n))$,} by a reduction to optimization problems in  nonlinear programming.
In Section \ref{sec:Gamma-I}, we investigate Gamma conjecture I for $\mathbb{P}_{\mathbb{P}^{n}}(\mathcal{O}\oplus\mathcal{O}(n))$. In Section \ref{sec:towards}, we discuss possible modifications of Gamma conjecture I.
In Sections \ref{sec:Gamma_Kaehler} and \ref{sec:example_Kaehler}, we discuss   Gamma conjecture I over the K\"ahler moduli space, demonstrating a close relationship with the birational geometry of the Fano manifold. In Section \ref{sec:Gamma_Kaehler}, we introduce the principal asymptotic class $A_X(\tau)$ of a Fano manifold $X$ for a general $\tau\in H^2(X,\bC)$; in Section \ref{sec:example_Kaehler} we compute it for the example $\bP_{\bP^n}(\cO\oplus \cO(n))$ in certain regions of the K\"ahler moduli space.
Finally in the appendix, {we study  Conjecture $\mathcal{O}$ for  $\mathbb{P}_{\mathbb{P}^{n}}(\mathcal{O}\oplus\mathcal{O}(n-1))$.}

 \subsection*{Acknowledgements} 
The authors would like to thank Kai Hugtenburg for very helpful discussions on modified Gamma conjecture I. The definition of the A-model conifold value, which emerged from our discussions with him, is key to the formulation of modified Gamma conjecture I.
The authors would also like to thank Kwok Wai Chan, Bohan Fang, ChunYin Hau, King-Leung Lee, Jiayu Song and Heng Xie for helpful discussions, and to thank Zhi Chen for collecting the toric data from \cite{Obro2} for toric Fano manifolds of dimensions less than 7. The authors are also grateful to the anonymous referee for their highly professional and valuable comments. 
Z.~Su  would like to thank Xiaowei Wang for constant encouragement. J.~Hu, C.~Li and H.~Ke are  supported in part by the National Key R \&  D Program of China No.~2023YFA1009801.
S.~Galkin is supported by CNPq grants PQ 315747 and and PQ 308303, and Coordena\c{c}\~{a}o de Aperfei\c{c}oamento de Pessoal de N\'{i}vel Superior-Brasil (CAPES)-Finance Code 001. J. Hu is also supported in part by NSFC Grant 12531003.
H.~Iritani is supported by JSPS grant 21H04994,  23H01073 and 23K25770. 
H.~Ke is also supported in part by NSFC Grant 12271532.

\section{Preliminaries}
\label{sec:preliminaries}
\subsection{Gamma conjecture I for Fano manifolds}
In this section, we briefly review Conjecture $\cO$ and Gamma conjecture I for Fano manifolds following \cite{GGI, GaIr} and \cite[Section 2.1]{HKLY}.
We also note that we can consider Gamma conjecture I under the weaker condition \eqref{eq:condition_star} than Conjecture $\cO$. 

\subsubsection{Quantum cohomology} We refer to \cite{CoKa} for background material on Gromov-Witten invariants and quantum cohomology.

Let $X$ be a Fano manifold, namely a compact complex manifold $X$ whose anticanonical line bundle is ample. Let $\overline{\mathcal{M}}_{0, k}(X, \mathbf{d})$ denote the moduli stack of $k$-pointed genus-zero stable maps $(f: C\to X; p_1, ..., p_k)$ of class $\mathbf{d}\in H_2(X,\mathbb{Z})$. Let $[\overline{\mathcal{M}}_{0, k}(X, \mathbf{d})]^{\rm virt}$ be the virtual fundamental class of $\overline{\mathcal{M}}_{0, k}(X, \mathbf{d})$, which is of  degree $\dim X-3+\int_{\mathbf{d}} c_1(X)+k$ in the Chow group $A_*(\overline{\mathcal{M}}_{0, k}(X, \mathbf{d}))$.
Given classes $\gamma_1, ..., \gamma_k\in H^*(X)=H^*(X, \mathbb{C})$ and nonnegative integers $a_i$ for $1\leq i\leq k$, we have the following  associated gravitational correlator
  $$\langle \psi^{a_1}\gamma_1,..., \psi^{a_k}\gamma_k\rangle_{0,k,\mathbf{d}}:=\int_{[\overline{\mathcal{M}}_{0, k}(X, \mathbf{d})]^{\rm virt}}\prod_{i=1}^k\big(c_1(\mathcal{L}_i)^{a_i}\cup \ev_i^*(\gamma_i)\big).$$
  Here $\mathcal{L}_i$ denotes the line bundle on $\overline{\mathcal{M}}_{0, k}(X, \mathbf{d})$ whose fiber over the stable map  $(f: C\to X; p_1, ..., p_k)$ is the cotangent space $T^*_{p_i}C$, and $\ev_i$ denotes the $i$-th  evaluation map. 

The (small) quantum product of $\alpha,\beta\in H^*(X)$ is defined by
\begin{equation}
\label{eq:small_quantum_product}
\alpha\star_\mathbf{q}\beta:=\sum_{\mathbf{d}\in H_2(X,\mathbb{Z})}\sum_{i} \langle \alpha,\beta,\phi_i\rangle_{0,3,\mathbf{d}}\phi^i\mathbf{q}^{\mathbf{d}}.
\end{equation}
Here $\{\phi_i\}$ is a homogeneous basis of $H^*(X)$, and $\{\phi^i\}$ is its dual basis in $H^*(X)$ with respect to the Poincar\'e pairing; the quantum parameter $\mathbf{q}$ lies in the ``K\"ahler moduli space''
\[
\mathcal{M}_X:=\mathrm{Hom}(H_2(X,\bZ)/\mathrm{torsion},\bC^\times),
\]
and the monomial function $\mathbf{q}^\mathbf{d}\colon \cM_X \to \bC$ is the evaluation at $\mathbf{d}\in H_2(X,\bZ)$. Note that $\mathbf{q}^\mathbf{d}$ has a non-zero coefficient in \eqref{eq:small_quantum_product} only if $\mathbf{d}$ is effective. By the Fano condition, there are only finitely many $\mathbf{d}$ such that $\mathbf{q}^\mathbf{d}$ has a non-zero coefficient. The quantum product $\star_\mathbf{q}$ is a family of products on $H^*(X)$ with parameter $\mathbf{q}$, and it is a deformation of the cup product:
\[
\lim\limits_{\mathbf{q}\to0}\alpha\star_\mathbf{q}\beta=\alpha\cup\beta.
\]


\subsubsection{Conjecture $\cO$ and condition \eqref{eq:condition_star}} Consider the even part of the cohomology $H^\bullet(X):=H^{\rm even}(X)$ and the finite-dimensional $\mathbb{C}$-algebra $QH^\bullet(X)=
   (H^\bullet(X), \bullet)$ with the product defined by $\alpha\bullet \beta:=\alpha\star_{\mathbf{q}=\mathbf{1}}\beta$, namely by specialization of the quantum parameter to the constant function with value $1$. Let $\hat c_1$ denote the linear operator induced by the first Chern class:
      $$\hat c_1: H^\bullet(X)\longrightarrow H^\bullet(X); \, \beta\mapsto c_1(X)\bullet\beta.$$

 \begin{defn}[Property $\cO$]\label{defnrho} For a Fano manifold $X$, we denote by $\rho=\rho(\hat c_1)$ the spectral radius of the linear operator $\hat c_1$, namely
    \[
    \rho:=\max\{|\lambda| : \lambda\in \Spec(\hat c_1)\}
    \]
    where $\Spec(\hat c_1):=\{\lambda : \lambda\in \bC  \mbox{ is an eigenvalue of } \hat c_1\}$.
 We say that $X$ satisfies \textbf{Property $\cO$} if the following two conditions are satisfied.
 \begin{enumerate}
   \item $\rho$ is a simple root of the characteristic polynomial of $\hat{c}_1$ (in particular $\rho \in \Spec(\hat{c}_1)$);
   \item for any $\lambda\in \Spec(\hat c_1)$ with $|\lambda|=\rho$, we have $\lambda^{i_X}=\rho^{i_X}$, where $i_X$ denotes the Fano index of $X$ defined by $i_X:=\max\{k\in \bZ~:~ {c_1(X)\over k}\in H^2(X,\mathbb{Z})\}$.
 \end{enumerate}
 \end{defn}
 \begin{conjO}[Conjecture 3.1.2 of \cite{GGI}]Every Fano manifold satisfies Property $\cO$.
   \end{conjO}

Conjecture $\cO$ has proven true for any flag variety $G/P$  in \cite{ChLi} by using Perron-Frobenius theorem based on a remark due to Kaoru Ono.

We also recall condition \eqref{eq:condition_star} from the Introduction, which is a weaker version of Property $\cO$. As remarked in \cite[Remark 3.1.9]{GGI}, we can study Gamma conjecture I under this weaker condition.
\begin{defn}[Condition ($\mathrm{SR}$)]
\label{defn:star}
We say that $\hat{c}_1$ has a \emph{simple rightmost eigenvalue} $u\in \bC$ if
\begin{itemize}
    \item[(1)] $u$ is a simple root of the characteristic polynomial of $\hat{c}_1$ and
    \item[(2)] any other eigenvalue $u'$ of $\hat{c}_1$ satisfies $\Re(u')<\Re(u)$.
\end{itemize}
We say that $X$ satisfies \emph{condition \eqref{eq:condition_star}} if $\hat{c}_1$ has a simple rightmost eigenvalue. In this case, the simple rightmost eigenvalue is given by
\[
\rho':=\max\{\Re(u) : u\in \Spec(\hat{c}_1)\}
\]
as $\hat{c}_1$ is a real operator. Note that $\rho' = \rho$ when Property $\cO$ holds.
\end{defn}

\begin{remark}
Here we regard $\hat{c}_1$ as an operator on the even part $H^\bullet(X)$. By the arguments in \cite{SaSh, GaIr},   Property $\cO$ or condition \eqref{eq:condition_star} remains equivalent even if we consider $\hat{c}_1$ as an operator on the entire cohomology group $H^*(X)$ or on the smaller part $\bigoplus_{p} H^{p,p}(X)\subset H^\bullet(X)$.
\end{remark}

\subsubsection{Gamma conjecture I}
\label{subsubsec:Gamma-I}
There is a quantum connection on the trivial $H^\bullet(X)$-bundle over $\mathbb{P}^1$, given by
  \[
  \nabla_{z\partial_z}=z\partial_z-{1\over z}(c_1(X)\bullet )+\mu.
  \]
  Here $z$ is an inhomogeneous coordinate on $\mathbb{P}^1$ and $\mu\in \End(H^\bullet(X))$ is the Hodge grading operator defined by $\mu(\phi)=(p-{\dim X\over 2})\phi$ for $\phi\in H^{2p}(X)$.
  The quantum connection is a meromorphic connection, which is logarithmic at $z=\infty$ and irregular singular at $z=0$. The space of flat sections can be identified with the cohomology group $H^\bullet(X)$ via the fundamental solution
\[
S(z)z^{-\mu}z^{c_1}
\]
where $S \colon \bP^1 \setminus \{0\} \to \End(H^\bullet(X))$ is a holomorphic function given by
\[
S(z) \alpha = \alpha + \sum_i \sum_{\bfd \neq 0} \left\<\frac{\alpha}{-z-\psi},\phi^i\right\>_{0,2,\bfd}  \phi_i
\]
and $c_1\in \End(H^\bullet(X))$ is the cup product operator by $c_1(X)$.
Here, $\alpha/(-z-\psi)$ should be expanded in the series $\sum_{n\ge 0} \alpha \psi^n (-z)^{-n-1}$.
See \cite[Proposition 2.3.1]{GGI} for detailed explanations.
\begin{prop-defn}[\protect{\cite[Proposition 3.3.1, Remark 3.1.9]{GGI}}]
\label{prop-defn:AX}
If $X$ satisfies Property $\cO$, or even more weakly, if $X$ satisfies condition \eqref{eq:condition_star}, then the space
\begin{align*}
\{\alpha\in H^\bullet(X): e^{\rho'/z}S(z)z^{-\mu}z^{c_1}(\alpha)\mbox{ has moderate growth as }z\to+0\}
\end{align*}
is a one-dimensional linear subspace of $H^\bullet(X)$, where `moderate growth' means that $\|e^{\rho'/z}S(z)z^{-\mu}z^{c_1}(\alpha)\|=O(z^{-m})$ as $z\to +0$ for some $m\in \bN$.
The \emph{principal asymptotic class} $A_X$ is a generator of this space defined up to scalar.
    \end{prop-defn}

The Gamma class \cite{Libg, Lu, Iri1}  is a real characteristic class defined for an almost complex manifold. It is defined by Chern roots $\delta_1,\ldots, \delta_n$ of the tangent bundle $TX$ of $X$ and
  Euler's $\Gamma$-function  $\Gamma(x)=\int_0^\infty e^{-t}t^{x-1}dt$, and has the following expansion:
  $$\widehat \Gamma_X:=\prod_{i=1}^n\Gamma(1+\delta_i)=\exp\big(-C_{\rm eu}c_1(X)+\sum_{k=2}^\infty (-1)^k(k-1)!\zeta(k)\ch_k(TX)\big)  \in H^\bullet (X, \mathbb{R})$$
where $C_{\rm eu}$ is the  Euler-Mascheroni  constant, $\zeta(k)=\sum_{n=1}^\infty{1\over n^k}$ is the value of the Riemann zeta function at $k$, and  $\ch_k$ denotes the $k$-th Chern character.

While the original Gamma Conjecture I in \cite[Conjecture 3.4.3]{GGI} requires Property $\cO$, we consider it here under the weaker assumption \eqref{eq:condition_star}. 

\begin{GammaI}
Let $X$ be a Fano manifold satisfying condition \eqref{eq:condition_star}. Then the principal asymptotic class $A_X$ is given by the Gamma class $\widehat \Gamma_X$.
\end{GammaI}

The principal asymptotic class also appears in an asymptotic expansion of Givental's $J$-function.
Givental's $J$-function is a cohomology-valued function in $(\tau,z)\in H^\bullet(X)\times \bC^\times$ given by descendant Gromov-Witten invariants
\begin{equation}
\label{eq:J-function}
J_X(\tau,z) =
1+ \frac{\tau}{z} + \sum_{i} \sum_{(n,\bfd)\neq (0,0)} \left\<1,\tau,\cdots,\tau,\frac{\phi^i}{z-\psi}\right\>_{0,n+2,\bfd} \phi_i
\end{equation}
where $\phi^i/(z-\psi)$ should be expanded in the series $\sum_{n\ge 0} \phi^i \psi^n z^{-n-1}$.
The restriction to the anticanonical line $\tau=c_1(X) \log t$ and $z=1$ is given in terms of the fundamental solution $S(z)z^{-\mu}z^{c_1(X)}$ as follows:
\[
J_X(c_1(X) \log t, 1) = z^{\frac{\dim X}{2}}\big(S(z)z^{-\mu}z^{c_1}\big)^{-1}1
\]
with $t=z^{-1}$ (see also \cite[Remark 2.2]{HKLY}).

The following result follows from the same argument in \cite[Proposition 3.8]{GaIr}, which is based on \cite[Propositions 3.6.2, 3.2.1]{GGI}.
\begin{prop}
\label{propJexpansion}
  For any Fano manifold $X$ satisfying condition \eqref{eq:condition_star}, Givental's $J$-function $J_X(c_1(X) \log t,1)$ has an asymptotic expansion of the form
   \begin{equation}\label{expandJ}
      J_X(c_1(X) \log t, 1)\sim C t^{-{\dim X\over 2}} e^{\rho' t}(A_X+\alpha_1 t^{-1} +\alpha_2 t^{-2}+\cdots)
   \end{equation}
  as $t\to +\infty$ on the positive real line, where $C$ is a non-zero constant and $\alpha_i\in H^\bullet(X)$.
\end{prop}
As a consequence, an equivalent description of Gamma conjecture I is given in terms of the $J$-function $J_X(c_1(X)\log t,1)$ as follows (cf.~\cite[Corollary 3.6.9]{GGI}):
\begin{cor} 
Let $X$ be a Fano manifold satisfying condition \eqref{eq:condition_star}. 
Then Gamma conjecture I holds for $X$ if and only if
\begin{align*}
[\hGamma_X] = \lim_{t\to \infty} [J_X(c_1(X) \log t, 1)].
\end{align*}
holds in the projective space $\bP(H^\bullet(X))$ of the cohomology group.
\end{cor}

\subsection{Mirror symmetry for toric Fano manifolds}
\label{subsec:MS_Fano_toric}

We refer readers to \cite{CLS} for detailed account of toric varieties.
A toric variety $X$ of complex dimension $N$ is an algebraic variety containing an algebraic torus as an open dense subset, such that the action of the torus on itself extends to the whole variety. Geometric information of $X=X_\Sigma$ can be  determined by the combinatorics of its associated {\em fan} $\Sigma$ in $\bR^N$. Here we always assume $X_\Sigma$ to be smooth and Fano, namely    $\Sigma=\Sigma(\Delta)$ is the normal fan of a convex polytope $\Delta$ with nice properties. Thus $X_\Sigma$ is also denoted as $X_\Delta$ alternatively.

Let $f=f_\Sigma=f_\Delta$ denote the Landau-Ginzburg superpotential mirror to the toric Fano manifold $X_\Sigma=X_\Delta$. It is a Laurent polynomial, which can be  read off immediately from $\Sigma$. Precisely, we denote by $b_1,...,b_m\in\bZ^N$ the primitive generators of rays in $\Sigma$. Then the Picard number of $X_\Sigma$ is equal to $m-N$, and $f$  is given by
\begin{align}\label{Laurent}
f: (\bC^\times)^N\rightarrow \bC;\, \mathbf{x} \mapsto f(\mathbf{x})=\mathbf{x}^{b_1}+\cdots+\mathbf{x}^{b_m},
\end{align}
where $\mathbf{x}:=(x_1,...,x_N)$ and for $b_i=(b_{i1},...,b_{iN})$, $\mathbf{x}^{b_i}:=x_1^{b_{i1}}\cdots x_N^{b_{iN}}$.

Mirror symmetry between quantum cohomology of $X_\Sigma$ and oscillatory integrals of $f$ has been proved in \cite{Gi1,Gi2,Iri1}. As a remarkable property, there exists an isomorphism $\Psi$ of $\mathbb{C}$-algebras between   quantum cohomology ring  and   Jacobian ring,
\[
\Psi\colon QH(X)|_{\mathbf{q}=1}\xrightarrow{\cong} \Jac(f)=\bC[x_1^{\pm 1},...,x_N^{\pm 1}]/(x_1\partial_{x_1}f,...,x_N\partial_{x_N}f).
\]
Here we set $\mathbf{q}=\mathbf{1}$ to introduce $f$ without introducing any deformations. Moreover, we have
\begin{prop}[\protect{\cite[Corollary G]{OsTy}}]\label{eigncrit}
We have $\Psi(c_1(X))=[f]$. Specifically, the eigenvalues of $\hat c_1$, including their multiplicities, coincide with the critical values of  $f$.
\end{prop}

Here we notice that the eigenvalues of the linear operator on $\Jac(f)$ induced by the function multiplication by $f$  are precisely the critical values of $f$.

We can use oscillatory integrals of $f$ to get information of quantum cohomology of $X_\Sigma$. For example, it follows from \cite[Theorem 4.14]{Iri1} that
\begin{align}
\<{1},S(z)z^{-\mu}z^{c_1}\widehat \Gamma_X\>^X=z^{-\frac{N}{2}}\int_{(\bR_{>0})^N}e^{-\frac{f(\mathbf{x})}{z}}\frac{dx_1\cdots dx_N}{x_1\cdots x_N}.
\end{align}
The restriction $f|_{(\bR_{>0})^N}$ is a real function on $(\bR_{>0})^N$ that admits a global minimum at a unique point $\mathbf{x}_{\rm con}\in(\bR_{>0})^N$ \cite{Gal,GGI,GaIr}. Such point $\mathbf{x}_{\rm con}$ is called the {\it conifold point} of $f$. Let $T_{\rm con}:=f(\mathbf{x}_{\rm con})$.

\begin{defn}[\protect{\cite[Condition 6.1]{GaIr}}]\label{BanalogyPO}
    We say that $X_\Sigma$ satisfies the \emph{B-analogue of Property $\cO$}, if its mirror superpotential $f=f_\Sigma$ satisfies the following two conditions:
  \begin{enumerate}
        \item[  {\upshape(a)}] every critical value $u$ of $f$ satisfies $|u|\leq T_{\rm con}$;
        \item[  {\upshape(b)}] the conifold point $\mathbf{x}_{\rm con}$ is the unique critical point of $f$ contained in $f^{-1}(T_{\rm con})$.
    \end{enumerate}
\end{defn}

\begin{prop}[\protect{\cite[Theorem 6.3]{GaIr}}]
Suppose that  a toric Fano manifold $X_\Sigma$ satisfies the B-analogue of Property $\cO$. Then $X_\Sigma$ satisfies Gamma conjecture I.
\end{prop}

\section{Conjecture $\mathcal{O}$ for $X_n$}
\label{sec:conjecture_O_Fano_toric}
In the rest of this paper, we will mainly investigate the example
$$X_n:=\mathbb{P}_{\mathbb{P}^n}(\mathcal{O}\oplus \mathcal{O}(n)).$$
Noting $\mathbb{P}_{\mathbb{P}^n}(\mathcal{O}\oplus \mathcal{O}(n))\cong \mathbb{P}_{\mathbb{P}^n}(\mathcal{O}\oplus \mathcal{O}(-n))$, we  always assume $n\geq 1$.
In this section,
 we  analyze the eigenvalues of $\hat c_1$ for $X_n$ via mirror symmetry. As we will show in    Theorem \ref{thmconjO1}, Conjecture $\mathcal{O}$ does not hold for $X_n$ when $n>1$ is odd.

\subsection{Notations}
The variety $X_n$ is a  toric Fano manifold of Picard number two and of Fano index one. Let $e_i:=(0, \ldots, 0, 1, 0, \ldots, 0)$ (with $1$ in the $i$th position), $1\le i\le n+1$ denote the standard basis vectors of $\mathbb{R}^{n+1}$. There are exactly $(n+3)$ primitive ray generators $b_i$ of the associated fan of the toric variety $X_n$, given by
\begin{equation}\label{torbb}
    \begin{aligned}
          b_i&=e_i,\quad \mbox{for } 1\leq i\leq n+1;\quad  b_{n+2}=-\sum_{i=1}^ne_i+ne_{n+1}, \quad
   b_{n+3}=-e_{n+1}.
    \end{aligned}
\end{equation}

\noindent Therefore the mirror superpotential $f$ defined on $(\mathbb{C}^\times)^{n+1}$ in \eqref{Laurent} becomes
\begin{equation}\label{eqnd} f(\textbf{x})=x_1+x_2+\cdots+x_n+x_{n+1}+\frac{x_{n+1}^n}{x_1x_2\cdots x_{n}}+\frac{1}{x_{n+1}}.
\end{equation}

We adopt the following notation to analyze the critical values of $f$.
\begin{align*}
    D_R&:=\{t\in \mathbb{C}:|t|\leq R\},\quad  \partial D_R:=\{t\in \mathbb{C}:|t|= R\},\quad \mathring{D}_R:= D_R\setminus\partial D_R;\\
    N_{[a,b]}&:=\{t\in \mathbb{C}: a\le|t|\le b\},\quad \mbox{where }  b>a>0.
\end{align*}
\noindent Namely  $D_R$ is the closed disc in $\mathbb{C}$ of radius $R>0$ centred at the origin, with  $\partial D_R$ being its boundary; $N_{[a,b]}$ is  an annulus in $\mathbb{C}$.

\subsection{Critical values of the superpotential}
\label{subsec:Example_Xn}
 We  assume $n\ge 2$, since  $X_1$ is the blow-up of $\mathbb{P}^2$ at a point, and has been studied in \cite{HKLY}.

Critical points $\mathbf{x}=(x_1, \ldots, x_{n+1})$ of $f$ are solutions to the system of equations  $\partial_{x_i}f=0, \, 1\le i\le n+1$. The number of critical points (with multiplicities counted) is equal to   $\dim H^*(X_\Delta)=2n+2$.
One can show that
\begin{equation*}
    x_1=x_2=\cdots=x_n=:x \in \mathbb{C}^\times.
\end{equation*}
Denote  $y:=x_{n+1}\in\bC^\times$.  Then the system of equations $\frac{\partial f}{\partial x_i}=0$ can be reduced to
\begin{equation}\label{eqnnn}
    \begin{cases}
        1-\dfrac{y^n}{x^{n+1}}=0,\\
        1+n\dfrac{y^{n-1}}{x^n}-\dfrac{1}{y^2}=0,
    \end{cases}\quad (x,y)\in(\bC^\times)^2.
\end{equation}
Clearly, each $t\in \mathbb{C}^\times$ gives a solution $(t^{n}, t^{n+1})$ to the first equation in \eqref{eqnnn}, and conversely, every solution $(x, y)$ to this equation admits such parametrization by simply taking $t={y\over x}$. Therefore the system
\eqref{eqnnn} is equivalent to
\begin{equation}\label{contr}
    t^{2n+2}+nt^{2n+1}-1=0.
\end{equation}
Consequently, the  critical values of $f$ (with multiplicity counted) are precisely given by
\begin{equation*}
    g(t):=n t^n+t^{n+1}+t^n+\frac{1}{t^{n+1}}
\end{equation*}
with respect to the  $2n+2$ roots of \eqref{contr} .

Now let us move to   the following optimization problem in nonlinear programming.
\begin{prob}\label{NLP1}
  \begin{equation*}
  \begin{array}{ll@{}ll}
\textnormal{maximize}  & \big|&t^{n+1}+( n+1)t^n+\frac{1}{t^{n+1}}\big| &\\
\\
\textnormal{subject to}&     &t^{2n+2}+nt^{2n+1}=1,\quad t\in \mathbb{C}^\times.
\end{array}
  \end{equation*}
  \end{prob}
\noindent The above problem aims to find one $t_0$
such that $|g(t_0)|=\max|g(t)|$ among those $t$ with the constraint. Clearly, this is quite closely related with Property $\mathcal{O}$. Indeed, $X_\Delta$ satisfies part (1) of  Property $\mathcal{O}$ if and only if the aforementioned  $t_0$ is unique and satisfies $g(t_0)\in \mathbb{R}_{\geq 0}$.
Denote
\begin{equation}\label{defhg}
h(t):=t^{2n+2}+nt^{2n+1}, \quad \Tilde{g}(t):=t^n+\frac{2}{t^{n+1}}, \quad \check g (t):=2 t^{n+1}+(2n+1)t^n.
\end{equation}
Notice  $h(t)=1$ if and only if $t^{n+1}+nt^n=\frac{1}{t^{n+1}}$. Whenever this holds,  we have
  $$g(t)=\Tilde{g}(t)=\check g (t).$$
  To maximize  $|g(t)|$, we analyze  the distribution of the roots of $h(t)-1$ as follows.

 One can check that $h(t)-1$ and $h'(t)$ have no common roots, implying that they are coprime. So every root of $h(t)-1$ has multiplicity one.
 Viewed as a real function, $h(t)$
 is decreasing on $(-\infty,-\frac{n(2n+1)}{2n+2})$ and increasing on $(-\frac{n(2n+1)}{2n+2},\infty)$, possessing exactly two critical points at $t=0$ and $t=-\frac{n(2n+1)}{2n+2}$.
 Moreover, $-n<-\frac{n(2n+1)}{2n+2}<0,\, h(-n)=h(0)=0<1$, $h(1)=1+n>1$ and $h(-n-1)=(n+1)^{2n+1}>1$.
 Therefore
$h(t)-1$ has exactly two real roots $a_+$ and $a_-$ with
 $$a_+\in (0,1),\qquad  a_-\in (-n-1,-n).$$
  \noindent Moreover, by examining the signs of $h'(a_+)$ and $h'(a_-)$, we see that they are both simple roots. Noting that  the critical points of $f$ are of the form $(t^n, \ldots, t^n, t^{n+1})$, we have
\begin{prop}\label{prop-conifoldvalue}
The conifold point of $f$ is $\mathbf{x}_{\rm con}:=(a_+^n, ..., a_+^n, a_+^{n+1})$, and
  $g(a_+)$ is the critical value $T_{\rm con}$ of $f$ at $\mathbf{x}_{\rm con}$.
\end{prop}

  Recall the following Rouch\'e's theorem in classical complex analysis (see, e.g., \cite[Theorem 6.24]{GKR}):
\begin{prop}[Rouch\'e's theorem]
    For two functions $f_1$ and $f_2$ holomorphic in a disk  $D_{R}$, if $|f_2|<|f_1|$ on $\partial D_R$, then $f_1$ and $f_1+f_2$ have the same number of zeros (counted with multiplicity) in $D_R$.
\end{prop}
\begin{prop}\label{diskofroots}
      In $D_1$, $h(t)-1$ has exactly $2n+1$ roots with  multiplicity counted.
\end{prop}
\begin{proof}
We write $h(t)-1=f_1(t)+f_2(t)$ with $f_1(t)=nt^{2n+1}$ and $f_2(t)=t^{2n+2}-1$.

When $n>2$, we have
\[|f_2(t)|=|t^{2n+2}-1|\le |t|^{2n+2}+1=2<n=|nt^{2n+1}|=|f_1(t)|,\quad \forall t\in \partial D_1.\]
By Rouch\'e's theorem, $h(t)-1$ has $2n+1$ roots in $D_1$.

\par
When $n=2$,
we have
\[|f_2(t)|\le (1+\epsilon)^{6}+1<2(1+\epsilon)^{5}=|f_1(t)|,\quad \forall t\in \partial D_{1+\epsilon},\]
provided that $0<\epsilon<{2\over 3}$ (where we notice $ (1+\epsilon)^{6}+1-2(1+\epsilon)^{5}$ is strictly decreasing in $[0, {2\over 3})$). Therefore $h(t)-1$ has 5 roots in $D_1$ by taking
$\epsilon\to +0$.
\end{proof}

\begin{lemma}
\label{easy} For any Laurent polynomial $p(t)$ in $\mathbb{R}_{\geq 0}[t, t^{-1}]$, we have
    \begin{equation*}
        \max_{t\in\partial D_R}|p(t)|=p(R).
    \end{equation*}
\end{lemma}
\begin{proof}
    Write $p(t)=\sum_i u_it^i$, where $u_i\in \mathbb{R}_{\geq 0}$ for all $i$. We have
    \[p(t)|_{t=R}\le \max\limits_{t\in\partial D_R}|p(t)|\le  \sum\limits_{i} u_i|t|^i\Big|_{t\in \partial D_R}=\sum\limits_{i}u_iR^i=p(R).\]
 Therefore the statement follows.
\end{proof}
\begin{cor}\label{cordisroot}
    All the roots of $h(t)-1$ but $a_-$ are in the annulus $N_{[a_+,1]}$.
\end{cor}
\begin{proof}
    By Lemma \ref{easy}, we have  $1=h(a_+)=\max_{t\in\partial D_{a_+}}|h(t)|$.
    By the maximum modulus principle, $|h(t)|<1$ for any $t\in \mathring D_{a_+}$. Hence the interior  $\mathring D_{a_+}$ contains no roots of  $h(t)-1$. By Proposition \ref{diskofroots} and noting $a_-\in(-n-1,-n)$,   the statement follows.
\end{proof}

Recall that the functions $g$ and  $\tilde g$ in \eqref{defhg} take the same value at the roots of $h-1$.

\begin{lemma} \label{maxonann1} $\Tilde{g}(a_+)=\max_{t\in N_{[a_+,1]}}|\Tilde{g}(t)|>\tilde{g}(1).$
\end{lemma}
 \begin{proof}
  Notice that $\Tilde{g}(t)$ is holomorphic on $N_{[a_+, 1]}$ and $\Tilde{g}(t)\in\mathbb{R}_{\geq 0}[t, t^{-1}]$. By the maximum modulus principle and Lemma \ref{easy}, we have
\begin{equation*}
    \max_{t\in N_{[a_+,1]}}|\Tilde{g}(t)|=\max_{t\in \partial D_{a_+}\cup \partial D_1}|\Tilde{g}(t)|=\max \{\Tilde{g}(a_+),\Tilde{g}(1)\}=\Tilde{g}(a_+).
\end{equation*}
Here we have $\Tilde{g}(1)<\Tilde{g}(a_+)$ by noting  $\Tilde{g}'(x)=(x^n+\frac{2}{x^{n+1}})'<0$ for any
  $x\in (0, 1)$.
 \end{proof}

The nonlinear programming Problem \ref{NLP1} can be answered   by the following proposition.

\begin{prop} \label{estimate1} Let $n\geq 2$. Take any $\alpha\in \mathbb{C}\setminus\{a_-, a_+\}$ with $h(\alpha)=1$. We have
 \begin{enumerate}
 \item $|g(\alpha)|<g(a_+)<2n+3$.
    \item $n^n-1<(-1)^ng(a_-)<n^n$.
    \item If $n\ge 3$,  then $|g(a_-)|>g(a_+)$; if  $n=2$,   $g(a_+)>|g(a_-)|$.
\end{enumerate}
\end{prop}



\begin{proof}
Recall that the functions $g, \tilde g, \check g$ in \eqref{defhg} take the same value at the roots of $h-1$.
It follows from Corollary \ref{cordisroot} and Lemma \ref{maxonann1} that
 $|g(\alpha)|=|\Tilde{g}(\alpha)|\leq \Tilde{g}(a_+)$, with equality only if  $\alpha\in  \partial D_{a_+}$.
 Tracking the inequalities in the proof of Lemma \ref{easy} with respect to $\check g(t)$, we see that $|\check g(\alpha)|=|\Tilde{g}(\alpha)|=\Tilde{g}(a_+)=\check{g}(a_+)$ only if
 $|2\alpha+2n+1|=2|\alpha|+ 2n+1$, implying $\alpha\in \mathbb{R}_{>0}$. However, this contradicts to the fact that $a_+$ is the only positive root of $h(t)-1$. Hence $|g(\alpha)|<g(a_+)$.
      Since $a_+\in (0,1)$, we have   $g(a_+)=\check g(a_+)<\check g(1)=2n+3$ (by noting that $\check g$ is strictly increasing in $\mathbb{R}_{>0}$).
     Hence, statement (1) holds.
      \par
   Assume  $n$ to be odd. Viewing $h, g, \check{g}$ as   real functions on  $[-n-1,-n]\subset \mathbb{R}$, we have  \[0<-\check g'(x)=-((2n+2)x+2n^2+n)x^{n-1}<-((2n+2)x+2n^2+n)x^{2n}=-h'(x).\]
  Hence,  $g(a_-)=\check g(a_-)>\check g(-n)= -n^n$ and  $g(a_-)<\check g(-n)+1=-n^n+1$, by noting
  \[\check g(a_-)-\check g(-n)=\int_{a_-}^{-n}-\check g'(x) dx<\int_{a_-}^{-n}-h'(x) dx=h(a_-)-h(-n)=1-0=1.\]
  The argument for even $n$ is the same. Thus statement (2) holds.

   For $n\geq 3$,  $|g(a_-)|>n^n-1>2n+3>g(a_+)$.  For $n=2$, we have
   $$g(a_+)=\check g(a_+)>\check g(\tfrac{4}{5})=2(\tfrac{4}{5})^3+5(\tfrac{4}{5})^2
   >2^2>g(a_-)>2^2-1=3.$$
   Therefore statement (3) holds.
   \end{proof}

The following proposition serves as the conclusion of this section. Recall $n\geq 1$.
\begin{thm}\label{thmconjO1}
  Conjecture $\mathcal{O}$ holds for  $X_n$ if and only if either $n$ is even or $n=1$.
\end{thm}
\begin{proof}
  If $n=1$, then  Conjecture $\mathcal{O}$ holds by \cite{HKLY}.
 Assume $n\geq 2$ now. By Proposition \ref{eigncrit}, Property $\mathcal{O}$ (1) holds if and only if $n$ is even.
 Notice that $X_n$ has Fano index one, and that the set $\{|f(\mathbf{x})|: \mathbf{x} \mbox{ is a critical point of } f \}$ achieves the maximum at $g(t)$ for a unique root $t$ of $h(t)-1$ by Proposition \ref{estimate1}. Thus Property $\mathcal{O}$ (2) holds as well, provided that Property $\mathcal{O}$ (1) holds.
\end{proof}





\section{Gamma conjecture I for $X_n$}
\label{sec:Gamma-I}

In this section, we will show in  {Theorem  \ref{thmGC1nothold}} that Gamma conjecture I does not hold for $X_n=\mathbb{P}_{\mathbb{P}^n}(\cO\oplus\cO(n))$ when $n\geq 4$ is even.

\subsection{Gamma class and conifold point}

In this section, we follow the notation in Section \ref{subsec:MS_Fano_toric}, and single out a result implicitly  contained
 in the proof of \cite[Theorem 6.3]{GaIr} in terms of Theorem \ref{thm-gammaclassandconifoldpoint}.

Recall that for an $N$-dimensional toric Fano manifold $X$, its mirror superpotential $f$ is given by \eqref{Laurent}.
The restriction $f|_{(\bR_{>0})^N}\colon (\bR_{>0})^N \to \bR$ to the positive-real locus is proper, bounded below and strictly convex in the sense that the matrix $(\partial_{\log x_i} \partial_{\log x_j} f(\bx))_{i,j}$ is positive definite for every $\bx \in (\bR_{>0})^N$. Therefore it attains a global minimum at a unique point $\mathbf{x}_{\text{con}} \in(\bR_{>0})^N$ called the \emph{conifold point} \cite{Gal}; $\bx_{\text{con}}$ is clearly a non-degenerate critical point of $f$.


\begin{thm}\label{thm-gammaclassandconifoldpoint}
    There are classes $\alpha_0,\alpha_1,\alpha_2,\dots \in H^\bullet(X)$ such that $\alpha_0\neq 0$ and that we have the asymptotic expansion
    \begin{align*}
        e^{T_{\rm con}/z}S(z)z^{-\mu}z^{c_1}\hGamma_X \sim \alpha_0 + \alpha_1 z + \alpha_2 z^2 + \cdots
    \end{align*}
    as $z \to 0$ along the angular sector $|\arg z|<\pi/2+\epsilon$ for some $\epsilon>0$. 
\end{thm}
\begin{proof}
It follows from the argument in \cite[Section 4.3.1]{Iri1} that
    \begin{align*}
        \<\phi,S(z)z^{-\mu}z^{c_1}\hGamma_X\>^X=z^{-\frac{N}{2}}\int_{(\bR_{>0})^N} e^{-f(x)/z}\varphi(x,-z)\frac{dx_1\cdots dx_N}{x_1\cdots x_N},
    \end{align*}
    where $z>0$, $\phi\in H^\bullet (X)$, and $\varphi(x,z)\in\bC[x_1^\pm,...,x_N^\pm,z]$ is such that the class $[e^{f(x)/z}\varphi(x,z)\frac{dx_1\cdots dx_N}{x_1\cdots x_N}]$ corresponds to $\phi$ under the mirror isomorphism in \cite[Proposition 4.8]{Iri1}.  In particular, when $\phi=1$, we have $\varphi=1$.

    Note that $f|_{(\bR_{>0})^N}$ is a real-valued function. So for $z\to 0$ along the sector $|\arg z|<\pi/2$, one can apply the Laplace method (see e.g. \cite{Fed}) to get the asymptotic expansion
    \begin{align*}
        \<\phi,S(z)z^{-\mu}z^{c_1}\hGamma_X\>^X \sim \frac{(2\pi)^{\frac{N}{2}}}{\sqrt{\det\Big(\partial_{\log x_i}\partial_{\log x_j}f(\bx_{\text{con}})\Big)}} e^{-T_{\rm con}/z}\Big[\varphi(\bx_{\text{con}},z)+ O(z) \Big].
    \end{align*}
This implies that there exist classes $\alpha_0, \alpha_1,\alpha_2, \dots \in H^\bullet(X)$ such that
\begin{align*}
e^{T_{\rm con}/z}S(z)z^{-\mu}z^{c_1}\hGamma_X \sim \alpha_0 + \alpha_1 z + \alpha_2 z^2 + \cdots
\end{align*}
as $z\to 0$ along the angular sector $|\arg z|<\pi/2$.
Now we use $\phi=1$ and $\varphi=1$ to get
\begin{align*}
\<1,\alpha_0\>^X=\lim\limits_{z\to +0}\<1,e^{T_{\rm con}/z}S(z)z^{-\mu}z^{c_1}\hGamma_X\>^X=\frac{(2\pi)^{\frac{N}{2}}}{\sqrt{\det\Big(\partial_{\log x_i}\partial_{\log x_j}f(\bx_{\text{con}})\Big)}} \neq0,
\end{align*}
implying that $\alpha_0$ is non-zero.

In fact, the above asymptotic expansion holds on an angular sector of the form $|\arg z|<\pi/2+\epsilon$ for some small $\epsilon>0$. We can see this by bending the Lefschetz thimble $(\bR_{>0})^N$ of $f(z)$ to the one associated with the vanishing path $f(\bx_{\text{con}}) + \bR_{\ge 0} e^{\iu \theta}$ with $\theta \in [-\epsilon,\epsilon]$. We refer to Section \ref{subsec:Gamma-flat-sections_thimbles} for an exposition on the Lefschetz thimbles associated with vanishing paths. Since the Lefschetz thimble $(\bR_{>0})^N$ contains no critical points other than $\bx_{\text{con}}$ in its closure (it is already closed), we can bend it for a small angle $\epsilon$ without changing its relative homology class (see also Example \ref{exa:positive_real_thimble}). 
\end{proof}

\begin{remark}
 The arguments here are essentially those in \cite[proof of Theorem 6.3]{GaIr}. However, the assumption here is weaker, since we do not impose any condition on $T_{\rm con}$.
\end{remark}

\begin{thm}\label{thm-gammaclassandeigenvector}
    The non-zero class $\alpha_0 \in H^\bullet(X)$ in Theorem \ref{thm-gammaclassandconifoldpoint} is an eigenvector of $\hat{c}_1$ with eigenvalue $T_{\rm con}$.
\end{thm}
\begin{proof}
    Recall that $s(z):=S(z)z^{-\mu}z^{c_1}\hGamma_X$ is $\nabla$-flat, i.e.
    \begin{align*}
        z\partial_zs(z)= z^{-1} \hat{c}_1s(z)-\mu s(z).
    \end{align*}
    Theorem \ref{thm-gammaclassandconifoldpoint} implies there exist $\alpha_0,\alpha_1,\alpha_2,...\in H^\bullet(X)$ such that
    \begin{align*}
        e^{T_{\text{con}}/z} s(z)\sim \alpha_0 +\alpha_1z+\alpha_2z^2+\cdots
    \end{align*}
    as $z\to 0$ in the angular sector $|\arg z|<\pi/2$.
    Observe that $\tilde{s}(z) = e^{T_{\text{con}}/z} s(z)$ satisfies the differential equation
    \[
z\partial_z \tilde{s}(z) = z^{-1} (\hat{c}_1-T_{\text{con}}) \tilde{s}(z) - \mu \tilde{s}(z)
    \]
    Using the fact that the differentiation commutes with asymptotic expansion on a complex region (see \cite[Theorem 8.8]{Wasow}), we obtain
    \[
0 = (\hat{c}_1 - T_{\text{con}}) \alpha_0
    \]
    from the leading term. The conclusion follows.
\end{proof}

\subsection{Gamma conjecture I for $X_n$}

Following the notation in Section \ref{sec:conjecture_O_Fano_toric}, we recall that $T_{\rm con}=g(a_+)$ for $X_n$.

\begin{thm}
    When $n\in\{1,2\}$, Gamma conjecture I holds for $X_n$.
\end{thm}
\begin{proof}
    Note that  $X_1$ is the blow-up of $\bP^2$ at one point, which is a del Pezzo surface. This case   was proved in \cite{HKLY}.

    When $n=2$, from Proposition \ref{estimate1} (3), we see that $T_{\rm con}$ is the spectral radius of $\hat{c}_1$, and it is of multiplicity one. So the statement follows from \cite[Theorem 6.3]{GaIr}.
\end{proof}

\begin{thm}\label{thmGC1nothold}
    When $n\geq3$ is even, Gamma conjecture I does not hold for $X_n$.
\end{thm}
\begin{proof}
   From Proposition \ref{estimate1}, we see that the spectral radius $\rho$ of $\hat{c}_1$, given by $-g(a_-)$,  is strictly larger than $T_{\rm con}$. Note that
   \begin{align*}
       e^{\frac{\rho}{z}}S(z)z^{-\mu}z^{c_1}\hGamma_{X_n}=e^{\frac{\rho-T_{\rm con}}{z}}\cdot e^{\frac{T_{\rm con}}{z}}S(z)z^{-\mu}z^{c_1}\hGamma_{X_n}.
   \end{align*}
   It follows from Theorem \ref{thm-gammaclassandconifoldpoint} that $e^{\frac{\rho}{z}}S(z)z^{-\mu}z^{c_1}\hGamma_{X_n}$ does not have moderate growth as $z\to+0$. Therefore the principal asymptotic class $A_{X_n}$ of $X_n$ is not given by its Gamma class.
\end{proof}

\begin{remark}
We shall determine $A_{X_n}$  in Section \ref{subsec:identification_A}.
\end{remark}

When $n\geq3$ is odd, Conjecture $\cO$ does not hold for $X_n$ (Proposition \ref{thmconjO1}),  Gamma conjecture I is not defined in its original form for $X_n$. However, Proposition \ref{estimate1} shows that $X_n$ satisfies condition \eqref{eq:condition_star}, i.e.~$\rho' = \max \{\Re(\lambda) : \lambda \in \Spec(\hat{c}_1)\}$ is a simple eigenvalue of $\hat{c}_1$ and that any other eigenvalue has strictly smaller real part. Therefore   Gamma conjecture I can still be formulated in Section \ref{subsubsec:Gamma-I}.

\begin{thm}\label{thmVGC}
    When $n\geq3$ is odd, Gamma conjecture I holds for $X_n$.
\end{thm}
\begin{proof}
    From Proposition \ref{estimate1}, we see that $\rho'=T_{\rm con}$. Now the statement follows from Theorem \ref{thm-gammaclassandconifoldpoint}.
\end{proof}

\section{Modifications of Gamma conjecture I}
\label{sec:towards}
Our Theorems \ref{thmGC1nothold} and \ref{thmVGC} indicate that the original statement of Gamma conjecture I has to be revised.
In this section, we propose   modifications of Gamma conjecture I in strong and weak forms in Conjectures \ref{conj:modified_Gamma-I_strong} and \ref{conj:modified_Gamma-I_weak} respectively. Here we keep the specialization
$\mathbf{q}=\mathbf{1}$.

\subsection{Regularized quantum periods and A-model conifold value}
\label{subsec:TAcon} 
We introduce the \emph{A-model conifold value} $\TAcon\in \bR_{\ge 0}$ that plays an important role in a modification of the Gamma conjecture I. We will prove in the next section that $\TAcon$ equals the conifold value associated with a weak Landau-Ginzburg model provided that the potential has nonnegative coefficients. The definition of $\TAcon$ emerged from our discussion with Kai Hugtenburg about a modified Gamma conjecture I. We thank him for allowing us to include this discovery in the present paper. 

The \emph{quantum period} $G_X(t)$ of a Fano manifold $X$ is defined to be the pairing of the $J$-function and the point class $[\pt] \in H^{\text{top}}(X)$ \cite{CCGGK:mirrorsymmetry}:   
\[
G_X(t) = \left\langle J_X(c_1(X)\log t,1), [\pt] \right\rangle^X = 1 + \sum_{c_1\cdot \bfd\ge 2} \left\langle [\pt] \psi^{c_1\cdot \bfd -2} \right \rangle_{0,1,\bfd} t^{c_1 \cdot \bfd}.   
\]
The \emph{regularized quantum period} $\hG_X(t)$ is defined to be its Borel transform: 
\begin{equation} 
\label{eq:Ghat} 
\hG_X(t) = \sum_{n=0}^\infty a_n t^n \quad \text{if we write $G_X(t) = \sum_{n=0}^\infty \frac{a_n}{n!} t^n$}. 
\end{equation} 
The argument in \cite[Lemma 3.7.4]{GGI} shows the following. 
\begin{prop} 
\label{prop:Ghat} 
The function $\hG_X(t)$ has a positive radius of convergence. Moreover, $\hG_X(t)$ can be analytically continued along any path  in  
\[
\bC^\times \setminus \{\lambda^{-1} :\lambda \neq 0, \lambda \in \Spec(\hat{c}_1)\}. 
\]
\end{prop} 
\begin{proof} 
As shown in the proof of \cite[Lemma 3.7.4]{GGI}, the function 
\[
\left(t \frac{d}{dt} t\right)^m\hG_X(t) = t^m \sum_{n=0}^\infty \Gamma(1+n+m) \frac{a_n}{n!} t^{n} 
\]
for sufficiently large $m$ can be represented as the pairing $\langle \varphi(-t^{-1}), 1\rangle$, where $\varphi(\lambda)$ is a certain \emph{homology}-valued function satisfying the differential equation 
\[
\left( \parfrac{}{\lambda} + (\lambda + (\hat{c}_1)^{\text{t}})^{-1} \left(1-\frac{\dim X}{2}+m - \mu^{\text{t}}\right) \right) \varphi(\lambda) =0 
\]
where the superscript ``$\text{t}$'' stands for the transpose (we set $\nu = 1-\frac{\dim X}{2} + m$ in the notation of \emph{loc.~cit}; the assumption $\nu\notin \bZ$ there is unnecessary). The differential equation has singularities only at points of the form $-\lambda$ with $\lambda \in \Spec(\hat{c}_1)$. This implies the desired conclusion. 
\end{proof} 

\begin{remark}
The proof of \cite[Lemma 3.7.4]{GGI} cites \cite[Proposition 3.5.2]{GGI}, which was proved under the assumption of Property $\cO$. However, this lemma holds unconditionally. The proof of \cite[Lemma 3.7.4]{GGI} requires the fact that a $\nabla^\vee_{z\partial_z}$-flat section $f(z)$ (i.e.~a homology-valued function $f(z)$ satisfying $(z\partial_z + z^{-1} (\hat{c}_1)^{\text{t}} - \mu^{\text{t}}) f(z)=0$) obeys the estimate $\|f(t^{-1})\|\le C_0 e^{C_1 t}$ for all $t\in [1,\infty)$ and some constants $C_0,C_1>0$. This is a standard result in the theory of differential equations and follows from Gronwall's inequality; see \cite[Appendix B.k]{LCEv10}. This was also described just above (3.4.6) in version 1 of the arXiv preprint \cite{GGI-arXiv1}. 
\end{remark} 

\begin{defn} 
\label{defn:A-model_conifold_value} 
The \emph{A-model conifold value} $\TAcon$ is defined to be the inverse of the convergence radius of $\hG_X(t)$, i.e.~$\TAcon := \limsup_{n\to \infty} |a_n|^{1/n}$ with notation as in \eqref{eq:Ghat}. It is a non-negative real number. 
\end{defn} 

\begin{remark} 
We define $\TAcon$ to be zero if the radius of convergence of $\hG_X(t)$ is infinite. However, we anticipate that $\TAcon$ is always positive (unless $X$ is a point). 
\end{remark} 

As extensive numerical calculations in the Fanosearch programme suggest (see \cite{Coates-Kasprzyk:database}), it is plausible that the coefficients $a_n$ of $\hG_X(t)$ are nonnegative integers for all (but finitely many) $n$. Furthermore, we anticipate that 
\begin{equation} 
\label{eq:suplim_replacedby_lim}
\TAcon = \lim_{n\to \infty} (a_{rn})^{1/(rn)}
\end{equation} 
where $r$ is (expected to be) the Fano index of $X$ and the limsup has been replaced with the limit. For example, they hold for toric Fano manifolds (see Example \ref{exa:toric_weak_LG} below). Under these assumptions, we can derive several conclusions about $\TAcon$.  

\begin{prop}\label{prop: TAconeigen}
Suppose that $a_n \ge 0$ for all but finitely many $n$ and that $\TAcon\neq 0$. Then $\TAcon \in \Spec(\hat{c}_1)$. 
\end{prop} 
\begin{proof} By the Vivanti-Pringsheim theorem, $\hG_X(t)$ cannot be extended to a holomorphic function in a neighbourhood of $t=\TAcon^{-1}$. The conclusion follows by Proposition \ref{prop:Ghat}.  
\end{proof}

\begin{prop} 
Suppose that $a_n\ge 0$ for all but finitely many $n$ and that $\TAcon\neq 0$. Then $\TAcon$ gives the exponential growth order of the quantum period $G_X(t)$ as $t\to \infty$ along the positive real line, i.e.~
\[
\TAcon= \limsup_{t\to +\infty} \frac{\log G_X(t)}{t}. 
\]
If moreover $\TAcon = \lim_{n\to \infty} (a_{rn})^{1/(rn)}$ for some $r\in \bZ_{\ge 1}$, then the limsup in the above formula can be replaced with the limit. 
\end{prop}
\begin{proof} 
We write $T:= \TAcon$ in the proof for simplicity. Since $T>0$, there exists some $n_0\in \bZ_{\ge 0}$ such that $a_{n_0}>0$ and $a_k \ge 0$ for all $k\ge n_0$. When $t\ge 0$, we have that 
\[
G_X(t) \ge \sum_{n=0}^{n_0} \frac{a_n}{n!} t^n = \text{(polynomial of degree $< n_0$)} + \frac{a_{n_0}}{n_0!} t^{n_0} 
\]
and the right-hand side diverges to infinity as $t\to +\infty$. Therefore $\log G_X(t)$ is a well-defined real number for $t\gg 0$. Moreover, if $P(t)$ is a polynomial of $t$, we have 
\[
\frac{\log(G_X(t) + P(t))}{t} = \frac{\log G_X(t)}{t}  + \frac{\log (1+ \frac{P(t)}{G_X(t)})}{t} 
\]
and $\lim_{t\to +\infty} (\log (1+ \frac{P(t)}{G_X(t)}))/t =0$. Hence $\limsup_{t\to\infty} (\log G_X(t))/t$ does not change by adding a polynomial to $G_X(t)$. Thus we may assume that $a_n\ge 0$ for all $n$. 

Since $T =\limsup_{n\to \infty} a_n^{1/n}$, for any $\epsilon>0$ with $\epsilon<T$, there exists $n_0\in \bZ_{\ge 0}$ such that $a_n \le (T+\epsilon)^n$ for all $n\ge n_0$ and $a_n \ge (T-\epsilon)^n$ for infinitely many $n$. Then we have 
\begin{align}
\label{eq:estimate_from_above} 
\begin{split} 
\limsup_{t\to +\infty} \frac{\log G_X(t)}{t} 
& = \limsup_{t\to +\infty} \frac{\log \left( \sum_{n=n_0}^\infty \frac{a_n}{n!} t^n\right) }{t} \\
& \le \limsup_{t\to +\infty} \frac{\log \left(\sum_{n=n_0}^\infty \frac{(T+\epsilon)^n}{n!} t^n\right)}{t} \\
& = \limsup_{t\to +\infty} \frac{\log (e^{(T+\epsilon)t})}{t} = T+\epsilon. 
\end{split} 
\end{align} 
Also, for each $n$ such that $a_n\ge (T-\epsilon)^n$, we have by setting $t= n/(T-\epsilon)$, 
\[
\frac{\log G_X(t)}{t} \ge \frac{ \log (\frac{a_n}{n!} t^n)}{t} \ge \frac{1}{t} \log \left(\frac{((T-\epsilon) t)^n}{n!}\right) 
= (T-\epsilon) \frac{1}{n} \log \frac{n^n}{n!}  
\]
Using the fact that this holds for infinitely many $n$ and the Stirling approximation $\log\frac{n^n}{n!} = n+O(1)$ as $n\to \infty$, we obtain 
\begin{equation} 
\label{eq:estimate_from_below} 
\limsup_{t\to +\infty} \frac{\log G_X(t)}{t} \ge T-\epsilon.  
\end{equation} 
Since \eqref{eq:estimate_from_above} and \eqref{eq:estimate_from_below} hold for all $\epsilon>0$, the first statement follows. 

The latter statement follows by a slight modification of the above argument. By the assumption $T= \lim_{n\to \infty} (a_{rn})^{1/(rn)}$, for any $\epsilon>0$ with $\epsilon<T$, there exists $n_0\in \bZ_{\ge 0}$ such that $a_{rn}\ge (T-\epsilon)^{rn}$ for all $n\ge n_0$. Let $t_0$ be such that $t_0\ge \frac{r n_0}{T-\epsilon}$ and that $G_X(t)$ is positive and monotonically increasing on $[t_0,\infty)$. Such $t_0$ exists because we know that $G_X'(t)>0$ for all sufficiently large $t$ by the same argument as in the first paragraph of the proof. Let $t\ge t_0+{r\over T-\epsilon}$ and choose $n\ge n_0$ such that $\frac{rn}{T-\epsilon} \le t \le \frac{r(n+1)}{T-\epsilon}$. Then 
\[
\frac{\log G_X(t)}{t} \ge \frac{\log G_X(\frac{rn}{T-\epsilon})}{\frac{r(n+1)}{T-\epsilon}} \ge 
\frac{\log\left(\frac{a_{rn}}{(rn)!} \left(\frac{rn}{T-\epsilon}\right)^{rn} \right)}{\frac{r(n+1)}{T-\epsilon}}  
\ge \frac{T-\epsilon}{r(n+1)} \log \left( \frac{(rn)^{rn}}{(rn)!} \right) 
\]
Since $n\to \infty$ as $t\to +\infty$, it follows (again from the Stirling approximation) that 
\[
\liminf_{t\to +\infty} \frac{\log G_X(t)}{t} \ge T-\epsilon.   
\]
This holds for all $\epsilon>0$ and the conclusion follows. 
\end{proof} 

\begin{remark} 
If $\TAcon=0$, we claim that $\lim_{t\to +\infty} (\log G_X(t))/t = 0$, even without the nonnegativity assumption for $a_n$. Note that the positivity of $G_X(t)$ for sufficiently large $t$ is not ensured in this case. In fact, for each $\epsilon>0$, there exists $n_0$ such that $|a_n|\le \epsilon^n$ for all $n\ge n_0$. Then we have
\[
\frac{|\log G_X(t)|}{t} \le \frac{\pi + \log |G_X(t)|}{t} \le \frac{\pi+ \log \left(\sum_{n<n_0} |a_n| t^n + e^{\epsilon t} \right)}{t} 
\]
where we choose the branch of $\log G_X(t)$ so that $\Im(\log G_X(t)) \in \{0,\pi\}$. 
Taking the limit $t\to +\infty$, we find that $\limsup_{t\to +\infty} |\log G_X(t)|/t \le \epsilon$. This implies the claim. 
\end{remark} 

\subsection{A-model conifold value and weak Landau-Ginzburg model}
\label{subsec:TAcon_weakLG}
We study the relationship between the A-model conifold value $\TAcon$ and the original conifold value $T_{\text{con}}$ associated with a Landau-Ginzburg mirror (as appeared in Section \ref{subsec:MS_Fano_toric} for toric Fano manifolds). 

\begin{defn}[\cite{Przyjalkowski:weak}]  
A Laurent polynomial $f(\bx)\in \bC[x_1^\pm,\dots,x_m^\pm]$ (where $m$ is not necessarily the dimension of $X$) is called a \emph{weak Landau-Ginzburg model} of $X$ if the regularized quantum period of $X$ is given by the constant term series of $f(\bx)$: 
\begin{equation} 
\label{eq:hG_constanttermseries}
\hG_X(t) = \sum_{n=0}^\infty \Const(f^n) t^n = \frac{1}{(2\pi \iu)^m}\int_{(S^1)^m} \frac{1}{1-tf(\bx)} \frac{dx_1\cdots dx_m}{x_1\cdots x_m} 
\end{equation} 
where $\Const(f^n)$ means the constant term of the Laurent polynomial $f(\bx)^n$. 
\end{defn} 

If $f(\bx) = \sum_{i=1}^N c_i \bx^{b_i}$ is a weak Landau-Ginzburg model and if $\varphi \colon \bR^m \to \bR$ is a linear function such that $\varphi(b_i) \ge 0$ for all $i$, then the truncated function $\tilde{f}(\bx) = \sum_{i:\varphi(b_i) = 0} c_i \bx^{b_i}$ is also a weak Landau-Ginzburg model, since $\Const(f^n) = \Const(\tilde{f}^n)$. Thus, by reducing the number of variables if necessary, we may assume that a weak Landau-Ginzburg model $f$ is either zero\footnote{A zero weak Landau-Ginzburg model, should it exist, is not convenient since the Newton polytope is empty.} or \emph{convenient}, i.e.~the Newton polytope of $f(\bx)$ contains the origin in its interior. Henceforth, we shall always assume convenience for non-zero weak Landau-Ginzburg models. 

Let $f(\bx)\in \bR_{\ge 0}[x_1^\pm,\dots,x_m^\pm]$ be a convenient Laurent polynomial with \emph{non-negative} coefficients. The \emph{conifold point} $\bx_{\text{con}}\in (\bR_{>0})^m$ of $f$ and its value $T_{\text{con}} = f(\bx_{\text{con}})>0$ are defined similarly to Section \ref{subsec:MS_Fano_toric}, i.e.~$\bx_{\text{con}}\in (\bR_{>0})^m$ is a unique critical point of $f|_{(\bR_{>0})^m}$ and $f|_{(\bR_{>0})^m}$ takes the minimum value $T_{\text{con}} = \min_{\bx \in (\bR_{>0})^m} f(\bx)$ at $\bx_{\text{con}}$. 

We note the following properties for nonnegative Laurent polynomials. 
\begin{lemma}[{\cite[Lemma 3.13]{GaIr}}]  
\label{lem:limit_constantterms}
Let $f(\bx)$ be a convenient Laurent polynomial with nonnegative coefficients. Let $r=r(f)\ge 1$ be a unique natural number such that $\Const(f^n)=0$ if $r\nmid n$ and that $\Const(f^{rn})>0$ for all but finitely many $n\ge 0$. Then the limit $\lim_{n\to \infty} (\Const(f^{rn}))^{1/(rn)}$ exists. 
\end{lemma} 
\begin{proof} 
The argument is identical to \cite[Lemma 3.13]{GaIr}, although the result there is stated under different assumptions. The existence of $r(f)$ follows from the fact that the set $\{n>0 : \Const(f^n)>0\}$ is nonempty and closed under addition. 
\end{proof} 

In particular, if a Fano manifold $X$ admits a weak Landau-Ginzburg model $f$ which is convenient and has only nonnegative coefficients, then the A-model conifold value of $X$ satisfies formula  \eqref{eq:suplim_replacedby_lim} for some $r$ divisible by the Fano index.

\begin{thm} 
\label{thm:nonnegative_Laurent}
Let $f$ be a convenient Laurent polynomial with nonnegative coefficients. Let $r = r(f)$ be the integer as in Lemma \ref{lem:limit_constantterms} and let $T_{\text{con}} = \min_{\bx \in (\bR_{>0})^m} f(\bx)$ be the conifold value. Then we have 
\[
\lim_{n\to\infty} (\Const(f^{rn}))^{1/(rn)} = T_{\text{con}}.   
\]
In particular, the convergence radius of the constant term series $\hG(t) := \sum_{n=0}^\infty \Const(f^n) t^n$ equals $T_{\text{con}}^{-1}$. Furthermore, $\hG(t)$ can be extended to a multi-valued analytic function on a punctured neighbourhood of $t=T_{\text{con}}^{-1}$ and has nontrivial monodromy at $T_{\text{con}}^{-1}$. 
\end{thm} 

\begin{cor} 
\label{cor:TAcon_Tcon} 
Suppose that a Fano manifold $X$ admits a weak Landau-Ginzburg model $f$ with nonnegative coefficients. 
Then the A-model conifold value $\TAcon$ of $X$ coincides with the conifold value $T_{\text{con}} = \min_{\bx\in (\bR_{>0})^m} f(\bx)$ of $f$. 
\end{cor} 

\begin{proof}[Proof of Theorem \ref{thm:nonnegative_Laurent}]
Let $m$ be the number of variables of $f(\bx)$ and we write $f(\bx) = \sum_i c_i \bx^{b_i}$ with $c_i>0$ and $b_i \in \bZ^m$ for all $i$. 
It suffices to prove the special case where $r(f)=1$. The general case then follows by applying this special case to the Laurent polynomial $f^r$. Hence, without loss of generality, we may assume that $r(f) = 1$. The convenience of $f$ implies that the vectors $b_i$ span $\bR^m$. We may also assume that the vectors $b_i$ generate $\bZ^m$ by changing the lattice.  

We use a random walk interpretation for $\Const(f^n)$ due to Galkin \cite{Galkin:Split}. We consider the random walk on the lattice $\bZ^m$, which starts at $0$ and at each step moves by the vector $b_i$ with probability $p_i := c_i \bx_{\text{con}}^{b_i}/T_{\text{con}}$. Note that we have $\sum_i p_i = 1$ by the definition of $T_{\text{con}}$. Moreover, since $\bx_{\text{con}}$ is a critical point, we have $\sum_i p_i b_i = 0$, i.e.~the random walk has mean zero. Note that the quantity 
\[
\frac{\Const(f^n)}{T_{\text{con}}^n} = \sum_{b_{i_1} + \cdots + b_{i_n} = 0} \frac{c_{i_1} \cdots c_{i_n}}{T_{\text{con}}^n} 
= \sum_{b_{i_1} + \cdots + b_{i_n} = 0} p_{i_1} \cdots p_{i_n} 
\]
is the $n$-step return probability. By the above assumptions, this random walk is \emph{irreducible}, (i.e.~it can reach any lattice point with positive probability) and \emph{aperiodic} (i.e.~the probability of returning to the origin in $n$ steps is positive for all sufficiently large $n$). Therefore, we can apply the Local Central Limit Theorem \cite[Theorem 2.3.5]{Lawler-Limic:random_walk} which holds for mean-zero, irreducible, aperiodic random walks with finite second and third moments. It follows that there exists $c>0$ such that 
\[
\left|\frac{\Const(f^n)}{T_{\text{con}}^n} - \frac{1}{(2\pi n)^{m/2} \sqrt{\Hess(f)(\bx_{\text{con}})}} \right| \le \frac{c}{n^{(m+1)/2}}
\]
holds for all $n$, where $\Hess(f)(\bx_{\text{con}})$ is the logarithmic Hessian $\det(\partial_{\log x_i} \partial_{\log x_j} f(\bx_{\text{con}}))$ of $f$ at $\bx_{\text{con}}$, which is denoted by $\det \Gamma$ in \cite[(2.2)]{Lawler-Limic:random_walk}. In particular, as $n\to \infty$, 
\[
\frac{\Const(f^n)}{T_{\text{con}}^n} = \frac{c'}{n^{m/2}} \left(1+ O\left(\frac{1}{\sqrt{n}}\right)\right)  
\]
for some constant $c'>0$. This implies $\lim_{n\to \infty} (\Const(f^n))^{1/n} = T_{\text{con}}$. 

Duistermaat and van der Kallen \cite{DvdK:constantterms} studied the constant term series $\hG(t)$ of a convenient Laurent polynomial $f$ with complex coefficients. They constructed a smooth compactification $M$ of $(\bC^\times)^m$ such that $f$ extends to a map $f\colon M \to \bP^1$ and that $\frac{dx_1\cdots dx_m}{x_1\cdots x_m}$ extends to a holomorphic $m$-form $\omega$ on $M\setminus f^{-1}(\infty)$ in \cite[Theorem 4]{DvdK:constantterms}. They then showed that $\hG(t)$ can be written as a period integral of the holomorphic $(m-1)$-form $\omega/df$ on $f^{-1}(1/t)$ 
\begin{equation} 
\label{eq:period_of_sigma} 
\hG(t) = \frac{1}{(2\pi \iu)^{m-1}}\int_{\sigma(1/t) \subset f^{-1}(1/t)} \frac{\omega}{t \cdot df} 
\end{equation} 
by ``residuating'' the integral \eqref{eq:hG_constanttermseries} along an $(m-1)$-cycle $\sigma(1/t)$. It follows that $\hG(t)$ can be analytically continued along any path avoiding the inverses of critical values of compactified $f$ \cite[Section 4]{DvdK:constantterms}. This is a phenomenon mirror to Proposition \ref{prop:Ghat}. 
Furthermore, they showed in \cite[Lemma 8]{DvdK:constantterms} (see also \cite[Theorem 3]{Nils} and \cite[Property (i) on Page 227]{DvdK:constantterms}) that the singularities of $\hG(t)$ at inverse critical values are milder than simple poles, i.e.~$\lim_{t\to t_c} (t-t_c) \hG(t)=0$ if $1/t_c$ is a critical value of $f$ and $t$ approaches $t_c$ within any fixed angular sector.  

We apply their results to the present case. Since $\hG(t)$ is a power series with positive coefficients, the Vivanti-Pringsheim theorem implies that it cannot be extended to a holomorphic function around $t= 1/T_{\text{con}}$. But the singularity of $\hG(t)$ at $t=1/T_{\text{con}}$ is milder than a simple pole. If $\hG(t)$ has trivial monodromy around $t=1/T_{\text{con}}$, then it extends to a holomorphic function around $t=1/T_{\text{con}}$, contradicting the Vivanti-Pringsheim theorem. Therefore, $\hG(t)$ must have \emph{nontrivial} monodromy around $t=1/T_{\text{con}}$. 
\end{proof} 

\begin{proof}[Proof of Corollary \ref{cor:TAcon_Tcon}]
As discussed above, we may assume that $f$ is either convenient or zero. 
If $f$ is convenient, the conclusion follows from Theorem \ref{thm:nonnegative_Laurent}. If $f=0$, then $\hG_X(t) = 1$ and $\TAcon = T_{\text{con}} = 0$.  
\end{proof} 

\begin{remark} 
(1) The proof of the Local Central Limit Theorem in \cite{Lawler-Limic:random_walk} is based on the estimate (Laplace method) of the integral $\Const(f^n) =\frac{1}{(2\pi\iu)^m} \int_{(S^1)^m_{\text{con}}} f(\bx)^n \omega$, where the integration cycle $(S^1)^m_{\text{con}}$ is the $(S^1)^m$-orbit of $\bx_{\text{con}}$. Alternatively, we can directly argue that a sufficiently high derivative $(\frac{d}{dt})^k \hG(t) =k! (2\pi\iu)^{-m} \int_{(S^1)^m_{\text{con}}} f(\bx)^k (1-tf(\bx))^{-(k+1)} \omega$ diverges as $t\nearrow T_{\text{con}}^{-1}$. 

(2) The cycle $\sigma(1/t) \subset f^{-1}(1/t)$ appearing in the proof can be chosen to be a torus $(S^1)^{m-1}$. In fact, for small enough $|t|>0$, we can choose it as an approximate fibre of the map $\Log_{|t|} \colon f^{-1}(1/t) \to \bR^m, \bx \mapsto (\log_{|t|} |x_1|,\dots,\log_{|t|} |x_m|)$ whose image converges to the tropical amoeba (the corner locus of $\Trop(1-tf)(\bx) = \min(\{b_i \cdot \bx +1\}_{i} \cup \{0\})$) as $t\to 0$. On the other hand, the vanishing cycle $V_{\text{con}}\subset f^{-1}(1/t)$ associated with $\bx_{\text{con}}$ is given by the positive real locus $(\bR_{>0})^m \cap f^{-1}(1/t)$ for $0<t<T_{\text{con}}^{-1}$. Thus, if $\bx_{\text{con}}$ is the only critical point of $f$ within $f^{-1}(T_{\text{con}})$ (including the compactified locus), the period integral \eqref{eq:period_of_sigma} over $\sigma(1/t)$ has a nontrivial monodromy at $t= T_{\text{con}}^{-1}$ due to Picard-Lefschetz theory, since $\sigma(1/t)$ and $V_{\text{con}}$ intersects at a single point transversely. This again provides an alternative proof. 

When $f$ is mirror to a Fano manifold $X$, then $f^{-1}(1/t)$ should be compactified to a Calabi-Yau manifold that is mirror to an anticanonical hypersurface $Y$ of $X$. The cycles $\sigma(1/t)$ and $V_{\text{con}}$ correspond to a fibre or section of a conjectural SYZ fibration on $f^{-1}(1/t)$ respectively. They should be mirror to $\cO_{\pt}$ and $\cO_Y$ respectively. The intersection number $V_{\text{con}} \cdot \sigma(1/t)$ should mirror to $\chi(\cO_Y,\cO_{\pt})=1$. By the Picard-Lefschetz argument, the monodromy around $t=T_{\text{con}}^{-1}$ acts as $\hG(t)\to \hG(t) \pm (2\pi\iu)^{-(m-1)}\varpi(t)$, where $\varpi(t)$ is the period $\int_{V_{\text{con}}} \frac{\omega}{t\cdot df}$ of $V_{\text{con}}$. After analytic continuation near $t=0$, we expect that $\varpi(t)$ has the asymptotics $\sim \int_Y t^{-c_1(X)} \hGamma_Y$ (as $t\to 0$), see \cite[Theorem 4.6]{Iri23}. 

(3) It follows by the triangle inequality that the cycle $(S^1)^m_{\text{con}} = (S^1)^m \cdot \bx_{\text{con}}$ maps\footnote{This also shows that $\hG(t)=(2\pi\iu)^{-m} \int_{(S^1)^m_{\text{con}}} (1-tf(\bx))^{-1} \omega$ converges for $|t|<T_{\text{con}}^{-1}$.} to the disc $\{|z|\le T_{\text{con}}\}$ by the map $f$. This implies that a Lefschetz thimble $\Gamma$ associated with a critical value $u$ with $|u|>T_{\text{con}}$ and a vanishing path avoiding the disc $\{|z|\le T_{\text{con}}\}$ (see Section \ref{subsec:Gamma-flat-sections_thimbles}) does not intersect with $(S^1)^m_{\text{con}}$. If $f$ is mirror to a Fano manifold $X$, then we anticipate that a Lagrangian torus isotopic to $(S^1)^m_{\text{con}}$ should be mirror to the skyscraper sheaf $\cO_{\pt}\in D^b(X)$ for some point $\pt\in X$. Therefore, by mirror symmetry, the exceptional object $E\in D^b(X)$ mirror to such a Lefschetz thimble $\Gamma$ should be numerically orthogonal to $\cO_{\pt}$ with respect to the Euler pairing. In particular, if $\hat{c}_1$ has a simple rightmost eigenvalue $T>T_{\text{con}} = \TAcon$ (as in our counter-examples $X_{2n}$ with $n\ge 2$), the principal asymptotic class $A_X$ should be orthogonal to $[\pt]$ with respect to the Poincar\'e pairing. This is because we expect that $A_X$ should be given by $\hGamma_X \Ch(E)$ for the exceptional object $E$ mirror to the Lefschetz thimble associated with the vanishing path $T+\bR_{\ge 0}$. In fact, the principal asymptotic class $A_{X_{2n}}$ in Theorem \ref{thm:A_for_Twrong} (with $n\ge 2$) satisfies $\langle A_{X_{2n}}, [\pt]\rangle^{X_{2n}} =0$. 
\end{remark} 

\begin{example} 
\label{exa:toric_weak_LG} 
The Landau-Ginzburg potential $f$ mirror to a toric Fano variety $X$ introduced in Section \ref{subsec:MS_Fano_toric} is a weak Landau-Ginzburg model, see e.g.~\cite[Section 4.4.4]{Iri1}. As can be seen from the discussion there, the coefficients $a_n$ of $\hG_X(t)=\sum_{n=0}^\infty a_n t^n$ are the nonnegative integers 
\[
a_n = \sum_{\substack{\bfd\in H_2(X,\bZ) \\  c_1 \cdot \bfd =n, D_i\cdot \bfd\ge 0 (\forall i)}} \frac{n!}{\prod_{i=1}^m (D_i\cdot \bfd)!} 
\]
where $\{D_1,\dots,D_m\}$ is the set of prime toric divisors of $X$. Since $\{\bfd : D_i\cdot \bfd \ge 0 (\forall i)\}$ is a full-dimensional cone in $H_2(X,\bR)$, it follows that $a_{rn}>0$ for sufficiently large $n$, with $r$ being the Fano index of $X$. Hence we have $\TAcon = \lim_{n\to\infty} (a_{rn})^{1/(rn)}$ by Lemma \ref{lem:limit_constantterms} and $\TAcon$ coincides with the conifold value $T_{\text{con}}$ of $f$ by Corollary \ref{cor:TAcon_Tcon}.  
\end{example}

\subsection{Proposal for a modified Gamma conjecture I}
\label{subsec:modified_Gamma_conjecture_I} 

In this section, we propose a modification of the Gamma conjecture I. There are strong and weak versions. 

\begin{conjecture}[Modified Gamma conjecture I: strong form]  
\label{conj:modified_Gamma-I_strong} 
Let $X$ be a Fano manifold. 
\begin{itemize} 
\item[$(1)$] There exist numbers $T\in \bC$ and $\epsilon>0$ such that the flat section associated with the Gamma class $\hGamma_X$ multiplied by $e^{T/z}$ 
\begin{align*}
e^{T/z} S(z)z^{-\mu}z^{c_1}\hGamma_X
\end{align*}
is of moderate growth as $z \to 0$ in the angular sector $|\arg z|<\pi/2 + \epsilon$. 
\item[$(2)$] The number $T$ coincides with the A-model conifold value $\TAcon$ in Definition \ref{defn:A-model_conifold_value}. 
\end{itemize} 
\end{conjecture} 

\begin{remark} \label{rmk: strong_form}
(1) All toric Fano manifolds satisfy Conjecture \ref{conj:modified_Gamma-I_strong}. This follows from Theorem \ref{thm-gammaclassandconifoldpoint} and the fact that the A-model conifold values $\TAcon$ of toric Fano manifolds coincide with the conifold values $T_{\text{con}}$ of their mirrors, see Example \ref{exa:toric_weak_LG}.  

(2) If $X$ satisfies Gamma conjecture I (assuming condition \eqref{eq:condition_star}) and if $\TAcon$ is a simple rightmost eigenvalue of $\hat{c}_1$, then $X$ satisfies Conjecture \ref{conj:modified_Gamma-I_strong}. 

(3) Note that it is a strong condition to require asymptotic behaviour on a sector of angle greater than $\pi$ centred at $\arg z =0$. As we shall see in the proof of Proposition \ref{prop:TAcon_eigenvalue}, this asymptotic condition on such a wide sector forces the flat section $S(z) z^{-\mu} z^{c_1} \hGamma_X$ to belong to a single exponential component of the sectorial decomposition $\Phi_I$ in \eqref{eq:analytic_lift}. Furthermore, the centering of the sector at $\arg z=0$ implies that this flat section does not undergo mutation at this angle, despite the potential existence of a real eigenvalue $u$ of $\hat{c}_1$ with $u > \TAcon$.

(4) Assuming the Sanda-Shamoto style conjecture \cite[Conjecture 2.13]{Iritani:Gamma_quantum} and adopting the notation there, the modified Gamma conjecture I (strong form) asserts that the subspaces $V_{\TAcon}^{\pm \epsilon}\subset K^0(X)$ contain the class $[\cO_X]$ for all sufficiently small $\epsilon>0$.   
\end{remark} 

Recently, Pomerleano-Seidel \cite{Pomerleano-Seidel:qconn} and Chen \cite{ZihongChen:exp} proved that the quantum connection of a Fano manifold is always of (unramified) exponential type. Using this significant result, we can establish the following two propositions. 

\begin{prop}
\label{prop:TAcon_eigenvalue} 
If Conjecture \ref{conj:modified_Gamma-I_strong}(1) holds, then the number $T$ therein is an eigenvalue of $\hat{c}_1$. Therefore, Conjectures \ref{conj:modified_Gamma-I_strong} (1) and (2) together imply that $\TAcon$ is an eigenvalue of $\hat{c}_1$. 
\end{prop} 
\begin{proof} 
The quantum connection $\nabla$ is said to be of \emph{exponential type}\footnote{Hertling and Sevenheck \cite{Hertling-Sevenheck:nilpotent} called it ``require no ramification''.} \cite{Malgrange:DE_book, Hertling-Sevenheck:nilpotent, KKP:Hodge} if it admits a formal decomposition 
\begin{equation} 
\label{eq:formal_decomposition} 
(H^\bullet(X) \otimes \bC[\![z]\!], \nabla) \cong \bigoplus_{u\in \Spec(\hat{c}_1)} (e^{u/z} \otimes R_u)\otimes_{\bC\{z\}} \bC[\![z]\!] 
\end{equation} 
where $R_u$ is a finite free $\bC\{z\}$-module equipped with a regular singular connection (with poles of order at most two at $z=0$) and $e^{u/z}$ denotes the rank one connection $(\bC\{z\}, d+ d(u/z) = d - u z^{-2} dz)$, see \cite[Lemma 8.2]{Hertling-Sevenheck:nilpotent}. 
Let $\cA$ be the sheaf of $\bC\{z\}$-modules over $S^1$ defined as follows \cite[Section 8]{Hertling-Sevenheck:nilpotent}. For an open interval $I\subset S^1$, we define $\cA(I)$ to be the set of holomorphic functions $f(z)$ on $\{0<|z|<\delta, z/|z| \in I\}$ (for some $\delta>0$ depending on $f$) admitting an asymptotic expansion 
\[
f(z) \sim a_0 + a_1 z + a_2 z^2 + \cdots 
\]
as $z\to 0$ in the sector $I$. Let $\cA_I$ denote the restriction of $\cA$ to $I$.  
By the Hukuhara-Turrittin theorem (see \cite[Lemma 8.3, Lemma 8.4]{Hertling-Sevenheck:nilpotent} in this context), the formal decomposition \eqref{eq:formal_decomposition} uniquely lifts to an analytic decomposition over an open sector $I\subset S^1$ of angle greater than $\pi$. 
\begin{equation} 
\label{eq:analytic_lift} 
\Phi_I \colon 
(H^\bullet(X) \otimes \cA_I, \nabla) \cong \bigoplus_{u\in \Spec(\hat{c}_1)} (e^{u/z} \otimes R_u) \otimes_{\bC\{z\}} \cA_I. 
\end{equation} 
By \cite[Lemma 8.3]{Hertling-Sevenheck:nilpotent}, we can choose $I$ to be an angular sector of the form $|\arg z|<\frac{\pi}{2} + \epsilon$ for some $\epsilon>0$ if $u_1-u_2\notin \bR_{>0}$ for all $u_1, u_2 \in \Spec(\hat{c}_1)$; otherwise we can choose $I$ to be an angular sector of the form $-\frac{\pi}{2} < \arg z < \frac{\pi}{2} + \epsilon$ for some $\epsilon>0$. In either case, we can choose $I=\{-\frac{\pi}{2}<\arg z < \frac{\pi}{2}+\epsilon\}$ for some $\epsilon>0$. 
Using the analytic decomposition, any flat section $s(z)$ of $\nabla$ can be written as  
\[
s(z) = \Phi_I^{-1} \left( \bigoplus_u e^{-u/z} \psi_u(z) \right) 
\]
over a region of the form $\{0<|z|<\delta, z/|z| \in I\}$, where $\psi_u(z) \in R_u\otimes \cA_I$ is a flat section of $R_u$ over the sector $I$. Let $s(z)$ be the flat section $S(z) z^{-\mu} z^{c_1} \hGamma_X$ and write it in the above form. Conjecture \ref{conj:modified_Gamma-I_strong}(1) then implies that 
\[
e^{(T-u)/z} \psi_u(z) 
\]
is of moderate growth as $z\to 0$ in the sector $-\frac{\pi}{2} <\arg z < \frac{\pi}{2} + \epsilon$ by choosing a smaller $\epsilon$ if necessary. Suppose that $\psi_u(z) \neq 0$ for some $u\neq T$. Then we can find an angle $\theta$ with $-\frac{\pi}{2} <\theta < \frac{\pi}{2} + \epsilon$ such that $\Re((T - u)/e^{\iu\theta})>0$. Since $R_u$ is regular singular, its flat section $\psi_u(z)$ satisfies an estimate of the form $\|\psi_u(z) \| \ge C |z|^N$ for some $C>0$, $N\geq 0$ and sufficiently small $|z|$, with respect to a local holomorphic frame of $R_u$ near $z=0$. This contradicts the moderate growth of $\|e^{(T-u)/z} \psi_u(z)\|$ along the ray $\bR_{>0} e^{\iu\theta}$. Therefore, we conclude that $\psi_u(z) = 0$ for all $u\neq T$. Hence we must have $T\in \Spec(\hat{c}_1)$ and   
\begin{equation} 
\label{eq:flat_section_T-component} 
S(z) z^{-\mu} z^{c_1} \hGamma_X = s(z) = \Phi_I^{-1} \left(e^{-T/z} \psi_T(z)\right). 
\end{equation} 
The last equation will be used in the proof of the next proposition. 
\end{proof} 

\begin{prop} 
\label{prop:modified_Gamma_I-asymptotic_expansion}
If Conjecture \ref{conj:modified_Gamma-I_strong}(1) holds and if the number $T$ therein is a simple eigenvalue of $\hat{c}_1$, then we have the asymptotic expansion 
\[
e^{T/z} S(z)z^{-\mu}z^{c_1}\hGamma_X \sim \alpha_0 + \alpha_1 z + \alpha_2 z^2 + \cdots
\]
as $z\to 0$ in an angular sector of the form $|\arg z|<\frac{\pi}{2}+\epsilon$ for some $\epsilon>0$, where $\alpha_0,\alpha_1,\dots $ lie in $H^\bullet(X)$ and $\alpha_0$ is a non-zero eigenvector of $\hat{c}_1$ with eigenvalue $T$. 
\end{prop} 

\begin{proof} 
We follow the notation in the proof of Proposition \ref{prop:TAcon_eigenvalue}. Comparing the order-two pole parts of the connection in the formal decomposition \eqref{eq:formal_decomposition}, we find that $\rank R_T =1$. Here we used the fact that the order-two pole part of the regular singular connection on $R_u$ is nilpotent (see \cite[Chapter II, Theorem 4.1]{Sabbah:isomonodromic}, \cite[Theorem 5.5.4]{Sibuya:DE_complex_domain}; in our context, this is also discussed by Chen \cite[Corollary 2.9, Equation (4.5)]{ZihongChen:exp}). The following argument is similar to the proof of \cite[Proposition 3.2.1(3)]{GGI}. The rank-one regular singular connection $R_T$ can be presented in the form $(\bC\{z\}, d-c \frac{dz}{z})$ for some $c\in\bC$. Under this presentation, the flat section $\psi_T(z)$ of $R_T$ appearing in \eqref{eq:flat_section_T-component} is of the form $C z^c$ for some $C\neq 0$. Therefore the flat section $s(z) = S(z) z^{-\mu} z^{c_1} \hGamma_X$ admits an asymptotic expansion of the form 
\begin{equation} 
\label{eq:expansion_s}
s(z) \sim e^{-T/z} z^c (\alpha_0 + \alpha_1 z + \alpha_2 z^2 + \cdots)  
\end{equation} 
along the angular sector $I$, where $\alpha_0,\alpha_1,\dots \in H^\bullet(X)$ and $\alpha_0 \neq 0$. Arguing as in the proof of Theorem \ref{thm-gammaclassandeigenvector}, we find that the flatness $\nabla s(z) =0$ implies 
\[
(T- \hat{c}_1) \alpha_0 =0, \quad (\mu+c) \alpha_0 = (\hat{c}_1 - T) \alpha_1. 
\]
Thus $\alpha_0$ is a non-zero eigenvector of $\hat{c}_1$ of eigenvalue $T$. From the second equation, we obtain 
\[
c \langle\alpha_0,\alpha_0\rangle^X = \langle \alpha_0, (\mu+c) \alpha_0 \rangle^X = \langle \alpha_0, (\hat{c}_1- T)\alpha_1\rangle^X = \langle (\hat{c}_1-T) \alpha_0, \alpha_1\rangle^X = 0 
\]
where we use the skew-symmetry of $\mu$ with respect to $\langle \cdot,\cdot \rangle^X$. 
On the other hand, we claim that $\langle \alpha_0, \alpha_0 \rangle^X \neq 0$. Since $T$ is a simple eigenvalue, we have $H^\bullet(X)= \bC \alpha_0 \oplus \Im(T-\hat{c}_1)$. Since $\alpha_0$ is orthogonal to $\Im(T-\hat{c}_1)$ for the Poincar\'e pairing, we must have $\langle \alpha_0, \alpha_0\rangle^X \neq 0$. It now follows that $c=0$. 

In the proof of Proposition \ref{prop:TAcon_eigenvalue}, it is also possible to choose $I$ to be a sector of the form $-\frac{\pi}{2} - \epsilon < \arg z < \frac{\pi}{2}$ for some $\epsilon>0$. Thus the asymptotic expansion \eqref{eq:expansion_s} with $c=0$ holds on a sector of the form $|\arg z|<\frac{\pi}{2} + \epsilon$. The conclusion follows. 
\end{proof} 

\begin{remark} 
Note that the conclusion of Proposition \ref{prop:modified_Gamma_I-asymptotic_expansion} holds for toric Fano manifolds (see Theorems \ref{thm-gammaclassandconifoldpoint}, \ref{thm-gammaclassandeigenvector}) despite the potential non-simplicity of the eigenvalue $\TAcon$.
\end{remark} 

\begin{defn}[Property $\mathcal{O}_A$]\label{defnPOA}
We say that a Fano manifold $X$ satisfies \emph{Property $\cO_A$} if $\TAcon$ is a simple rightmost eigenvalue of $\hat{c}_1$. 
\end{defn}

If $X$ satisfies Property $\cO_A$, then it satisfies condition \eqref{eq:condition_star}. We also note that 
\begin{prop}
\label{prop:sufficient_condition_PropertyOA} 
Assume that condition \eqref{eq:condition_star} is satisfied and $\langle A_X, [\pt]\rangle^X \neq 0$. Then Property $\cO_A$ holds under either of the following conditions. 
\begin{itemize} 
\item[(1)] Part (1) of Property $\cO$ (Definition \ref{defnrho}) holds. 
\item[(2)] We have $a_n\ge 0$ for all but finitely many $n$, where $a_n$ is as in \eqref{eq:Ghat}. 
\end{itemize} 
\end{prop} 
\begin{proof} 
Suppose that (1) holds. Proposition 3.7.6 in \cite{GGI} says that $\TAcon$ equals the spectral radius $\rho$ of $\hat{c}_1$ under these assumptions (Property $\cO$ was assumed there, but only Part (1) was used in the argument). Property $\cO$(1) implies that $\TAcon$ is a rightmost simple eigenvalue. 

Next, suppose that (2) holds. Let $\rho'\in \bR$ be the simple rightmost eigenvalue of $\hat{c}_1$. As discussed in the proof of \cite[Lemma 3.7.6]{GGI}, $\langle A_X, [\pt]\rangle^X \neq 0$ implies that the limit 
\begin{equation}
\label{eq:exponential_growth_GX}
\lim_{t\to +\infty} e^{-\rho't} t^{\frac{\dim X}{2}} G_X(t)
\end{equation} 
exists and is nonzero. (Note that while Property $\cO$ was assumed in the proof, only condition \eqref{eq:condition_star} was necessary in the argument). Since $a_0=1$ and $a_n \ge 0$ for all but finitely many $n$, we have $\lim_{t\to \infty} |t^{\frac{\dim X}{2}} G_X(t)|=\infty$ (unless $X$ is a point, in which case Property $\cO_A$ holds). This, together with the existence of the limit \eqref{eq:exponential_growth_GX}, forces that $\rho'>0$ (unless $X$ is a point). 
Recall the following identity from \cite[(3.7.5)]{GGI}: 
\[
\int_0^\infty t^{\frac{\dim X}{2}} G_X(t) e^{-t/\kappa} dt = \kappa^{\frac{\dim X}{2}} \sum_{n=0}^\infty \Gamma(1+n+\tfrac{\dim X}{2}) \frac{a_n}{n!} \kappa^n 
\]
which holds for sufficiently small $\kappa>0$. Note that the right-hand side has the same convergence radius as $\hG_X(t)$. 
The left-hand side is convergent if $\Re(1/\kappa)>\rho'$ and diverges if $\kappa$ approaches $1/\rho'$ from the left. Therefore, since the coefficients of the power series in the right-hand side is nonnegative (except for finitely many terms), Vivanti-Pringsheim theorem implies that the convergence radius of the right-hand side is exactly $1/\rho'$. Consequently, the convergence radius of $\hG_X(t)$ is also $1/\rho'$, and thus $\TAcon=\rho'$. 
\end{proof} 

We formulate a weak form of the modified Gamma conjecture I. 

\begin{conjecture}[Modified Gamma conjecture I: weak form] 
\label{conj:modified_Gamma-I_weak} 
Suppose that a Fano manifold $X$ satisfies Property $\cO_A$. Then $X$ satisfies Gamma conjecture I, i.e~the principal asymptotic class $A_X$ is given by the Gamma class $\hGamma_X$.
\end{conjecture} 

By the definition (Proposition/Definition \ref{prop-defn:AX}) of the principal asymptotic class, the strong form of the modified Gamma conjecture I (Conjecture \ref{conj:modified_Gamma-I_strong}) immediately implies the weak form (Conjecture \ref{conj:modified_Gamma-I_weak}). In particular, toric Fano manifolds also satisfy the weak form of the conjecture.

\begin{remark}
    We anticipate that $\TAcon$ is an eigenvalue of $\hat{c}_1$ and coincides with the conifold value $T_{\text{con}}$ determined by the mirror Landau-Ginzburg potential $f$, as discussed in Sections \ref{subsec:TAcon}-\ref{subsec:TAcon_weakLG}. In particular, this coincidence does hold  for all toric Fano manifolds. 
By numerical computations for mirrors of low-dimensional toric Fano manifolds, we have obtained \footnote{Table \ref{tabPO} was obtained by using Mathematica with all versions 13.0, 10.0 and 8.0 involved. Since the numerical computation may contain errors due to  the command ``NSolve", this table should be used for reference only.} Table \ref{tabPO}.

\begin{center}
\begin{table}[h]
    \caption{Eigenvalue properties of $\hat c_1$ for low-dimensional toric Fano manifolds}\label{tabPO}
    \begin{tabular}{ccccc}
  \toprule
  \multirow{2}{*}{Dimension} & Number of toric &  Property $\cO$(1) & Condition \eqref{eq:condition_star} & Property $\cO_A$  \\
  & Fano manifolds & fails & fails & fails \\
  \hline
  1 & 1 & 0 & 0 & 0 \\
  
  2 & 5 & 0 & 0 & 0 \\
  
  3 & 18  & 0 & 0 & 0 \\
  
   4 & 124  &  8 & 0 & 4 \\
  
  5 & 866 & 111 & 9 & 44 \\
  
  6 & 7622 & 1409 & 245 & 626 \\
  \bottomrule
\end{tabular}
\end{table}
\end{center}
As we can see from Table \ref{tabPO}, most of the toric Fano manifolds $X_\Sigma$ of dimension $\le 6$ satisfy Property $\mathcal{O} (1)$ and Property $\mathcal{O}_A$. Almost all of them satisfy condition \eqref{eq:condition_star}.
It will be desirable to give a combinatorical  characterization on those toric Fano manifolds that satisfy Property $\cO_A$.
 
\end{remark}
 
To end this section, we draw Figure \ref{fig:relationsnew}, illustrating the logical relationship among various versions of Gamma conjecture I and the relevant conditions. In particular, we remind that the strong form of Gamma conjecture I is a statement for Fano manifolds without any further conditions, while the weak form is a statement under the additional condition of Property $\cO_A$.

\begin{figure}[htbp]
\[
\xymatrix{
& &
*+[F]\txt{$X$ is smooth toric Fano}
\ar@{=>}[dll]
\ar@{=>}[d]|{\text{with Property } \cO_A}
& &
*+[F]\txt{convenient weak LG with positive coe.}
\ar@/^6pt/@{=}!<50pt,0pt>;!<65pt,-15pt>\ar@/^0pt/@{=}!<65pt,-15pt>;!<65pt,-85pt>\ar@/^4.8pt/@{=}!<65pt,-85pt>;!<52pt,-95pt>
\\
*+[F]\txt{modified $\hat{\Gamma}$-I:\\strong form}
\ar@{<}!<20pt,-5pt>;!<18pt,-5pt>
\ar@{=>}[d]
\ar@/_6pt/@{=}!<-25pt,-10pt>;!<-40pt,-25pt>\ar@/_0pt/@{=}!<-40pt,-25pt>;!<-40pt,-60pt>\ar@/_6pt/@{=}!<-40pt,-60pt>;!<-25pt,-75pt>\ar@/_0pt/@{=}!<-30pt,-75pt>;!<280pt,-75pt>\ar@/_6pt/@{=}!<280pt,-75pt>;!<295pt,-60pt>\ar@{=>}!<295pt,-60pt>;!<295pt,-58pt>
& &
*+[F]\txt{$X$ satisfies $\hat{\Gamma}$-I}
\ar@/_0pt/@{=}!<-40pt,-5pt>;!<-100pt,-5pt>
\ar@{}[d]|+
\ar@/_6pt/@{--}!<-25pt,12pt>;!<-40pt,-3pt>\ar@/_0pt/@{--}!<-40pt,0pt>;!<-40pt,-45pt>\ar@/_6pt/@{--}!<-40pt,-50pt>;!<-25pt,-65pt>\ar@/_0pt/@{--}!<-25pt,-65pt>;!<25pt,-65pt>\ar@/_6pt/@{--}!<25pt,-65pt>;!<40pt,-50pt>\ar@/_0pt/@{--}!<40pt,-50pt>;!<40pt,0pt>\ar@/_6pt/@{--}!<40pt,-3pt>;!<25pt,12pt>\ar@/_0pt/@{--}!<25pt,12pt>;!<-25pt,12pt>
& &
*+[F]\txt{Condition \eqref{eq:condition_star} }
\ar@{.}[ll]_(0.48){\text{\scriptsize underlying}}^(0.48){\text{\scriptsize assumption}}
\\
*+[F]\txt{modified $\hat{\Gamma}$-I:\\weak form}
& &
*+[F]\txt{Property $\cO_A$}
 \ar@{.}[ll]_(0.48){\text{\scriptsize underlying}}^(0.48){\text{\scriptsize assumption}}
 \ar@{=>}[rr]
 \ar@{=>}[urr]
& &
*+[F]\txt{$T_{\rm A,\mathrm{con}}$ is an eigenvalue}
\ar@{=>}!<55pt,0pt>;!<50pt,0pt>
\\
}
\]
\caption{A logical relationship among various conditions.}\label{fig:relationsnew}
\end{figure}
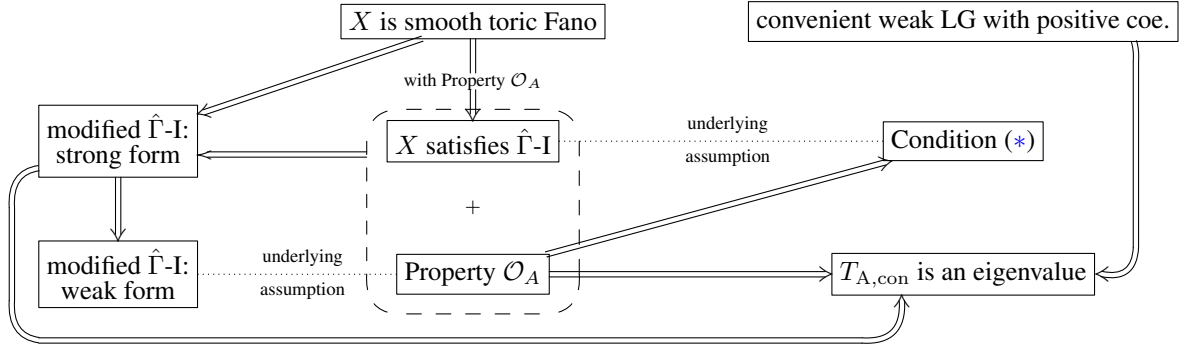

\section{Gamma conjecture I over the K\"ahler moduli space}
\label{sec:Gamma_Kaehler}

Recall that the small quantum product $\star$ is parametrized by $\bq$, which lies in the ``K\"ahler moduli space'' $H^2(X,\bZ)\otimes \bC^\times$. While both Conjecture $\cO$ and Gamma conjecture I focus on the quantum multiplication at $\bq=\mathbf{1}$, no theoretical justification (beyond empirical evidence) existed for this specific specialization. In this section, we explore these conjectures for general values of $\bq$.

We will also demonstrate that varying $\bq$ is closely related to birational geometry of the Fano manifold. When $\bq$ varies and approaches a boundary point of the K\"ahler moduli space, the spectrum of quantum cohomology assembles in a specific way, reflecting the geometry of the corresponding extremal contraction; such phenomena have been observed in \cite{Gonzalez-Woodward:tmmp, Iritani:discrepant, Lee-Lin-Wang:flip, Iritani-Koto:projective, Iritani:monoidal}. In the next Section \ref{sec:example_Kaehler}, we shall investigate this phenomenon in more details for the example $X_n=\bP_{\bP^n}(\cO\oplus \cO(n))$.


\subsection{Quantum cohomology and extremal contractions}
As mentioned above, the small quantum product $\star$ of a Fano manifold $X$ is parametrized by a point $\bq$ in the \emph{K\"ahler moduli space}
\[
\cM_X := \Hom(H_2(X,\bZ)/{\text{torsion}}, \bC^\times) \cong H^2(X,\bZ)\otimes \bC^\times
\]
where the monomial $\bq^\bfd$ in the quantum product $\star$ is identified with the function $\cM_X \to \bC^\times$ given by the evaluation at $\bfd\in H_2(X,\bZ)$.
We also introduce a linear coordinate $\tau \in H^2(X,\bC)$ on $\cM_X$ related to $\bq$ via
\[
\bq^\bfd = e^{\tau\cdot \bfd}
\]
for $\bfd\in H_2(X,\bZ)$. The space $H^2(X,\bC)=H^{1,1}(X)$ of $\tau$ is identified with the universal covering of $\cM_X$.
Let $\NEN(X)\subset H_2(X,\bZ)/{\text{torsion}}$ denote the monoid generated by effective curves. This monoid partially compactifies $\cM_X$:
\[
\cM_X \hookrightarrow \ovcM_X :=\Hom(\NEN(X),\bC^\times)
\]
where the coordinate ring of $\ovcM_X$ is the monoid ring $\bC[\NEN(X)]$. The small quantum product $\star$ extends to a family of supercommutative product structures over $\ovcM_X$: we call $\ovcM_X$ the \emph{partially compactified K\"ahler moduli space}.

Let $\ovNE(X)\subset H_2(X,\bR)$ be the closure of the cone generated by $\NEN(X)$; it is called the Mori cone. An \emph{extremal ray} $R=\bR_{\ge 0} \bfd$ is a one-dimensional face of the cone $\ovNE(X)$ generated by a class $\bfd$ with $c_1(X)\cdot \bfd>0$. By the Contraction Theorem \cite[Theorem 3.1.2]{KMM:MMP}, an extremal ray $R$ associates a contraction $f\colon X\to Y$ to a normal projective variety $Y$ such that a curve in $X$ contracts to a point in $Y$ if and only if its homology class lies in $R$. We call $f$ an \emph{extremal contraction}. The extremal ray $R$ also defines a ring homomorphism
\[
\bC[\NEN(X)] \to \bC[R_\bN], \quad \bq^\bfd \mapsto
\begin{cases}
\bq^\bfd & \text{if $\bfd \in R$} \\
0 & \text{otherwise}
\end{cases}
\]
where $R_\bN = R\cap \NEN(X)$. This defines a line
\[
\cC_f :=\Spec \bC[R_\bN] \subset \ovcM_X
\]
in the boundary of the partially compactified K\"ahler moduli space. We call $\cC_f$ the \emph{$f$-boundary} of $\ovcM_X$.

We want to consider the limit where the parameter $\bq\in \cM_X$ approaches $\cC_f$. Given an ample class $\omega_Y$ on the base $Y$ of the contraction, we can consider the flow $s\mapsto \bq_s$ on $\cM_X$ given by
\[
\bq_s= e^{-s f^*\omega_Y} \bq  \qquad \text{(or more precisely, $\bq_s^\bfd = e^{-s f^*\omega_Y \cdot \bfd} \bq^\bfd$).}
\]
As $s\to \infty$, the flow $\bq_s$ converges to a point in the $f$-boundary $\cC_f$. In terms of the linear coordinate $\tau\in H^2(X,\bC)$, the same flow is given by
\[
s\mapsto \tau_s = \tau - s f^* \omega_Y.
\]
Under the $s\to \infty$ limit, the small quantum product of $X$ at $\bq_s$ approaches the \emph{$f$-exceptional quantum product} $\star_{\text{exc}}$ defined by
\[
\alpha\star_{\text{exc}} \beta =\sum_i \sum_{\bfd \in R_\bN} \langle \alpha, \beta, \phi^i\rangle_\bfd \bq^\bfd \phi_i.
\]
This defines a family of supercommutative rings over $\cC_f$. A similar product has been defined by Ruan \cite{Ruan:crepant} in connection with his famous crepant resolution conjecture.
The $f$-exceptional quantum product, which is easier to compute, leads to a decomposition of the original quantum cohomology. When $\bq$ is sufficiently close to a boundary point $\overline{\bq} \in \cC_f=\bC$, the quantum cohomology ring $QH(X)|_\bq=(H^*(X),\star_\bq)$ decomposes into subrings indexed by maximal ideals $\frakm$ of $(H^{\text{ev}}(X),\star_{\text{exc},\overline{\bq}})$ (i.e. points of the $0$-dimensional scheme $\Spec (H^{\text{ev}}(X), \star_{\text{exc},\overline{\bq}})$) for the non-zero complex number $\overline{\bq}$:
\begin{equation}
\label{eq:decomp_QH}
QH(X)\bigr|_{\bq} \cong \bigoplus_{\frakm} A_\frakm.
\end{equation}
\begin{question}
Can we identify each factor $A_\frakm$ with quantum cohomology of another space in a natural way?
\end{question}
Examples (see below) show that $A_\frakm$ can arise as the (big) quantum cohomology of the base space $Y$, the flip $X^\dagger \to Y$ of $f$, or a space related to the exceptional set of $f$.

\begin{remark}
The $f$-boundary $\cC_f$ and $f$-exceptional quantum product $\star_{\text{exc}}$ can be defined beyond the Fano case. For the degree reason, $\star_{\text{exc}}$ is always convergent (polynomial) even if the original quantum product $\star$ is not known to be convergent. In general, the decomposition \eqref{eq:decomp_QH} makes sense for $\bq$ in the formal neighbourhood of $\overline{\bq}\in \cC_f$.
\end{remark}

\begin{example}[{\cite[Theorem 5.5]{Gonzalez-Woodward:tmmp}, \cite[Theorem 1.1]{Iritani:discrepant}, \cite[Theorem 0.1]{Lee-Lin-Wang:flip}}]
Let $X$ be a smooth toric Deligne-Mumford (DM) stack given as a GIT quotient of a vector space and let $X\dasharrow X^\dagger$ be a flip (or a discrepant transformation) arising from variation of GIT quotients. Then $QH(X)$ contains $QH(X^\dagger)$ as a direct summand, possibly after deforming $QH(X^\dagger)$ to big quantum cohomology.
\end{example}

\begin{example}[{\cite[Theorem 5.1]{Iritani-Koto:projective}}]
Let $X=\bP(V)$ be the projectivization of a vector bundle $V$ of rank $r$ over a smooth base $Y$. The $\bP^{r-1}$-fibration $X \to Y$ gives rise to a decomposition
$QH(X) \cong \bigoplus_{i=1}^r QH(Y)_{\sigma_i}$, where $QH(Y)_{\sigma_i}$ stands for the big quantum cohomology with parameter $\sigma_i \in H^*(Y)$. This can be viewed as a quantum cohomology analogue of the Leray-Hirsch theorem.
\end{example}

\begin{example}[{\cite[Theorem 1.1]{Iritani:monoidal}}]
\label{exa:blowup}
Let $\widetilde{X}$ be the blowup of a smooth projective variety $X$ along a smooth centre $Z$ and let $r$ be the codimension of $Z$ in $X$. The blowdown morphism $\widetilde{X} \to X$ gives rise to a decomposition $QH(\widetilde{X}) \cong QH(X)_\tau \oplus \bigoplus_{i=1}^{r-1} QH(Z)_{\sigma_i}$ where $\tau \in H^*(X)$, $\sigma_i\in H^*(Z)$ stand for the big quantum cohomology parameters.
\end{example}

\subsection{Principal asymptotic class for general $\bq$}
\label{subsec:A_general_q}
We discuss the principal asymptotic class for general $\bq\in \cM_X=H^2(X,\bZ)\otimes \bC^\times$ for a Fano manifold $X$. The discussion here is parallel to Section \ref{subsubsec:Gamma-I}.

\begin{defn}\label{defnsimplerm} Let $\bq\in \cM_X$. We say that $(c_1(X)\star_\bq)$ has a \emph{simple rightmost eigenvalue $u\in \bC$} if
\begin{itemize}
\item[(1)]
$u$ is a simple root of the characteristic polynomial of $(c_1(X)\star_\bq)$ and
\item[(2)]
any other eigenvalues $u'$ of $(c_1(X)\star_\bq)$ satisfy $\Re(u') < \Re(u)$.
\end{itemize}
\end{defn}
This condition relaxes Property $\cO$ in several ways: we do not require $u$ to be positive real or have the largest norm among eigenvalues; also we do not require the condition involving the Fano index $i_X$. As remarked in \cite[Remark 3.1.9]{GGI}, it enables us to define the principal asymptotic class.

\begin{remark}
(1) When $\bq$ is real, eigenvalues of $(c_1(X)\star_\bq)$ are invariant under complex conjugation and a simple rightmost eigenvalue $u$  (if any) is necessarily real. In particular when $\mathbf{q}=\mathbf{1}$, the notion of a simple rightmost eigenvalue coincides with condition \eqref{eq:condition_star}. 

(2) For every $\bq\in \cM_X$, the set of eigenvalues of the small quantum product $(c_1(X)\star_\bq)$ is invariant under multiplication by $i_X$th roots of unity. This is because the complex Euler flow $\bC \ni s\mapsto \bq_s = e^{s c_1(X)} \bq$ dilates the spectrum of $(c_1(X)\star_\bq)$ by $e^s$ and $\bq_s = \bq$ for $s = 2\pi \iu/i_X$. This implies that the argument of a non-zero rightmost eigenvalue lies in the interval $[-\pi/i_X, \pi/i_X].$
\end{remark}

We consider the quantum connection
\begin{equation}
\label{eq:quantum_connection_q}
\nabla_{z\partial_z}\bigr|_\bq = z\partial_z - \frac{1}{z}( c_1(X)\star_\bq) + \mu
\end{equation}
on the trivial $H^*(X)$-bundle over $\bP^1$. This has a fundamental solution of the form (due to Dubrovin; see \cite[Remark 2.3.2]{GGI} and references therein)
\[
S(\tau,z) z^{-\mu} z^{c_1(X)}
\]
where $\tau\in H^2(X,\bC)$ is a lift of $\bq$ (i.e.~$\bq = e^\tau$) and $S(\tau,z)$ is a $GL(H^*(X))$-valued holomorphic function of $(\tau,z) \in H^2(X,\bC)\times (\bP^1\setminus \{0\})$. The map $\alpha \mapsto S(\tau,z) z^{-\mu} z^{c_1(X)} \alpha$ identifies $H^*(X)$ with the space of multi-valued flat sections for $\nabla_{z\partial_z}|_\bq$.
This fundamental solution is characterized by the condition that
\begin{itemize}
\item it is also flat with respect to the quantum connection in the $\tau$-direction and;
\item $S(0,z)$ coincides with $S(z)$ in Section \ref{subsubsec:Gamma-I}.
\end{itemize}
More precisely, in terms of descendant Gromov-Witten invariants, $S(\tau,z)$ is given by
\[
S(\tau,z) \alpha= e^{-\tau/z}\alpha +\sum_i \sum_{\bfd \neq 0} \left\langle  \frac{e^{-\tau/z}\alpha}{-z-\psi},\phi^i\right\rangle_{0,2,\bfd} e^{\tau\cdot \bfd} \phi_i.
\]
The following result follows by the same argument as in \cite[Proposition 3.3.1]{GGI}.
\begin{prop-defn}\label{pdasymp} Suppose that $(c_1(X)\star_\bq)$ has a simple rightmost eigenvalue $u$. Then the space
\[
\{ \alpha \in H^*(X) : \text{$\|e^{u/z} S(\tau,z) z^{-\mu} z^{c_1(X)} \alpha\|=O(z^{-m})$ for some $m\in \bN$ as $z\to +0$}\}
\]
is one-dimensional, where $\tau\in H^2(X,\bC)$ is a lift of $\bq$. The \emph{principal asymptotic class} $A_X(\tau)$ at $\tau$ is a generator of this vector space defined up to constant. Moreover the limit
\[
\lim_{z\to+0} e^{u/z} S(\tau,z) z^{-\mu} z^{c_1(X)} A_X(\tau)
\]
exists and is a  $u$-eigenvector of $(c_1(X)\star_\bq)$.
\end{prop-defn}

\begin{remark}
\label{rem:AX_choice_of_lift}
The class $A_X(\tau)$ depends on the choice of a lift $\tau$ of $\bq$. We have $A_X(\tau+2\pi \iu\xi) = e^{2\pi \iu \xi} A_X(\tau)$ for $\xi \in H^2(X,\bZ)$ since $S(\tau+2\pi \iu \xi,z) = S(\tau,z) e^{-2\pi \iu \xi/z}$.
\end{remark}

The principal asymptotic class also arises as a limit of the $J$-function.
The restriction of the $J$-function \eqref{eq:J-function} to $\tau\in H^2(X)$ is given by
\begin{align*}
J_X(\tau,z) & = S(\tau,z)^{-1} 1= e^{\tau/z} \left(1 +\sum_i \sum_{\bfd \neq 0}
\left\langle \frac{\phi^i}{z(z-\psi)} \right\rangle_{0,1,\bfd} e^{\tau\cdot \bfd} \phi_i \right).
\end{align*}
\begin{prop}
\label{prop:limit_formula_general_q}
Suppose that $(c_1(X)\star_\bq)$ has a simple rightmost eigenvalue $u$ for $\bq=e^\tau\in \cM_X$. Then we have the asymptotic expansion
\[
J_X(\tau+c_1(X)\log t, 1) \sim C t^{-\frac{\dim X}{2}} e^{ut} (A_X(\tau) + \alpha_1 t^{-1} + \alpha_2 t^{-2} + \cdots)
\]
as $t\to +\infty$ on the positive real line, where $C\neq 0$ is a constant and $\alpha_i \in H^*(X)$. In particular, we have $\lim_{t\to +\infty} [J(\tau+ c_1(X) \log t,1)] = [A_X(\tau)]$ in $\bP(H^*(X))$.
\end{prop}
\begin{proof}
This follows by the same argument as the $\tau=0$ case in \cite[Proposition 3.8]{GaIr}, which is based on \cite[Propositions 3.6.2, 3.2.1]{GGI}. The argument there shows that
\[
z^{-c_1(X)} z^\mu J_X(\tau,z)
\sim C e^{u/z} (A_X(\tau) + \alpha_1 z + \alpha_2 z^2+ \cdots)
\]
as $z\to +0$. The conclusion follows from the identity
\[
t^{\frac{\dim X}{2}} J_X(\tau+c_1(X)\log t,1) = z^{-c_1(X)} z^\mu J_X(\tau,z)
\]
with $t= z^{-1}$.
\end{proof}

\begin{cor}
\label{cor:A_easy_properties}
The principal asymptotic class $A_X(\tau)$, when it is defined, depends only on the class $[\tau]$ in $H^2(X,\bC)/\bR c_1(X)$ up to a constant multiple. If $\tau\in H^2(X,\bR)$, $A_X(\tau)$ is a complex multiple of a real cohomology class.
\end{cor}

Gamma conjecture I states that $[A_X(0)] = [\hGamma_X]$ in $\bP(H^*(X))$. In Theorem \ref{thmGC1nothold}, we show that $[A_{X_n}(0)] \neq [\hGamma_{X_n}]$ for even $n\ge 4$. This motivates the following question:
\begin{question}\label{quesAX}
Does there exist a point $\bq = e^\tau \in \cM_X$ such that $(c_1(X)\star_\bq)$ has a simple rightmost eigenvalue and that $[A_X(\tau)] = [\hGamma_X]$? Can we find such a $\tau$ within $H^2(X,\bR)$? Can we characterize the region in $H^2(X)$ where $[A_X(\tau)] = [\hGamma_X]$ holds?
\end{question}

Even if $[A_X(\tau)] \neq [\hGamma_X]$, $A_X(\tau)$ should be a special class: we can conjecture that $A_X(\tau)$ comes from an exceptional object in the derived category. In light of Dubrovin's conjecture \cite[Conjecture 4.2.2]{Du1998} and Gamma conjecture II \cite{GGI}, we ask the following:

\begin{question}
\label{question:AX_exceptional}
When the principal asymptotic class $A_X(\tau)$ is defined,
does there exist a $K$-class $V\in K^0(X)$ of topological vector bundles such that $[A_X(\tau)] = [\hGamma_X \Ch(V)]$?  Here $\Ch(V) = \sum_{k\ge 0} (2\pi \iu)^k \ch_k(V)$ is the $(2\pi\iu)$-modified Chern character. Does $V$ satisfy $\chi(V,V) =1$? Does it come from an exceptional object in the derived category of coherent sheaves?
\end{question}

If the answer to this question is affirmative at a real $\tau\in H^2(X,\bR)$, then Corollary \ref{cor:A_easy_properties} implies that $\ch(V)$ lies either in $\bigoplus_{k\ge 0} H^{4k}(X)$ or in $\bigoplus_{k\ge 0} H^{4k+2}(X)$ (i.e.~$\Ch(V)$ is real or purely imaginary).


\begin{remark}
Here we have focused on the \emph{principal} asymptotic class. However, our scope can be broadened to consider (1) \emph{higher} asymptotic classes associated with a simple eigenvalue of $(c_1(X)\star_\bq)$ (see Gamma conjecture II of \cite{GGI}), or even (2)  asymptotic \emph{subspaces} associated with a not necessarily simple eigenvalue of $(c_1(X)\star_\bq)$ (coined as ``Gamma conjecture III'' by the authors of \cite{GGI}, see \cite{SaSh, Iritani:Gamma_quantum}).
\end{remark}

\subsection{Semisimple case}
In this section, we briefly discuss the case where the (big) quantum cohomology is generically semisimple. Note that toric Fano manifolds have generically semisimple (small and big) quantum cohomology.

When the big quantum cohomology is convergent at $\tau \in H^*(X)$, we can similarly define the principal asymptotic class $A_X(\tau)$. For this we need to replace $(c_1(X)\star_\bq)$ in the quantum connection \eqref{eq:quantum_connection_q} with the Euler multiplication $(E\star_\tau)$ (we have $E= c_1(X)$ for $\tau \in H^2(X)$) and use the fundamental solution $S(\tau,z)$ extended to $\tau\in H^*(X)$.

Suppose that the big quantum cohomology of $X$ is generically semisimple. Around a semisimple point $\tau_0\in H^*(X)$, the eigenvalues $u_1,\dots, u_s$ of $(E\star_\tau)$ form a coordinate system in a neighbourhood of $\tau=\tau_0$ \cite[Lecture 3]{Dubrovin:Painleve}.
We can move the eigenvalues $u_1,\dots,u_s$ as long as we can solve the corresponding Riemann-Hilbert problem. In fact, the big quantum cohomology Frobenius manifold can be extended to an open dense subset\footnote{the complement of an analytic hypersurface} of the universal covering $(\bC^s \setminus \text{(big diagonal)})\sptilde$ \cite[Theorem 4.7]{Dubrovin:Painleve}, where  $\bC^s \setminus \text{(big diagonal)} $ is   the configuration space of $s$ mutually distinct points $u_1,\dots,u_s$.
On the semisimple locus, the condition that $(E\star_\tau)$ has a simple rightmost eigenvalue fails along  real-codimension one walls and the big quantum cohomology parameter space is divided into chambers. The principal asymptotic class $A_X(\tau)$ stays constant in a chamber, and the dominant asymptotic classes alternate when $\tau$ crosses a wall. Note that this is not a mutation we discuss in \S\ref{subsec:identification_A} below.

When the big quantum product is semisimple for generic $\tau\in H^*(X,\bC)$, the same holds also for generic \emph{real} parameters $\tau \in H^*(X,\bR)$; the Euler multiplication $(E\star_\tau)$ has a mutually distinct simple eigenvalues away from a real-codimension one hypersurface in $H^*(X,\bR)$. On the real locus $H^*(X,\bR)$, we observe two types of open chambers. At a generic $\tau\in H^*(X,\bR)$, we have either
\begin{itemize}
\item[(i)] $(E\star_\tau)$ has a simple rightmost eigenvalue (which is automatically real), or
\item[(ii)] $(E\star_\tau)$ has a complex-conjugate pair of two simple eigenvalues $u,\overline{u}$ such that all other eigenvalues $u'$ satisfy $\Re(u') < \Re(u)$.
\end{itemize}
The locus where both conditions fail is of real-codimension 1 or greater. In fact, if, for some $\tau\in H^*(X,\bR)$, $(E\star_\tau)$ has only simple eigenvalues $u_1,\dots,u_s$ and these conditions fail, then either (a) there are three eigenvalues $u_1,u_2,u_3$ such that $u_1 = \overline{u}_1$, $u_3 = \overline{u}_2$ and $\Re(u_1) = \Re(u_2) = \Re(u_3)$ or (b) there are four eigenvalues $u_1,u_2,u_3,u_4$ such that $u_2=\overline{u}_1$, $u_4= \overline{u}_3$ and $\Re(u_1) = \Re(u_2) = \Re(u_3) = \Re(u_4)$. Such a point $\tau$ belongs to the real-codimension one locus given by the intersection of $H^*(X,\bR)$ with the complex hypersurface (a) $u_1=(u_2+u_3)/2$ or (b) $u_1+u_2 = u_3+ u_4$.
When $\tau\in H^*(X,\bR)$ lies in a chamber of type-(ii) with $u= x+ \iu y$, we expect an oscillating asymptotics
\begin{align*}
z^{-c_1} z^\mu J_X(\tau,z)
& \sim e^{x/z} (\Re(e^{iy/z} A) + O(z))  \\
& = e^{x/z} (\cos(y/z) A_1 - \sin(y/z) A_2 + O(z))
\end{align*}
as $z\to +0$, for some $A=A_1 + \iu A_2 \in H^*(X,\bC)$. If $\tau \in H^2(X,\bR)$, the left hand side equals $t^{\dim X/2} J_X(\tau+c_1(X) \log t, 1)$ with $t= 1/z$. The class $A$ here can be viewed as a substitute of the principal asymptotic class in a type-(ii) chamber.

\begin{example}
Let $f\colon \widetilde{X}\to X$ be a blowup of $X$ at a point. If $\dim X \ge 4$ and a positive real parameter $\bq=e^\tau\in \cM_{\widetilde{X}}$ is close to the $f$-boundary, then the point $\tau \in H^2(\widetilde{X},\bR)$ lies in a type-(ii) chamber. See Remark \ref{rem:eigenvalues_blowup}.
\end{example}

\section{Gamma conjecture I for $X_n$ over the K\"ahler moduli space}
\label{sec:example_Kaehler}

We examine the counterexamples $X_n=\bP_{\bP^n}(\cO\oplus \cO(n))$ to Conjecture $\cO$/Gamma conjecture I  from a birational geometry perspective.
The space $X_n$ has two extremal contractions: the divisorial contraction $\varphi \colon X_n \to \bP(1^{n+1},n)$ to the weighted projective space and the $\bP^1$-fibration $\pi \colon X_n \to \bP^n$.
Each contraction corresponds to an extremal ray of the Mori cone and associates a line in the boundary of the K\"ahler moduli space. We observe that the principal asymptotic class equals the Gamma class $\hGamma_{X_n}$ when $\bq$ is close to the ``$\pi$-boundary''. However, when $\bq$ is close to the ``$\varphi$-boundary'' and $n$ is even, it deviates from $\hGamma_{X_n}$. We also determine the principal asymptotic class in the latter region (that includes $\bq=\mathbf{1}$) using monodromy and mutation argument.

\subsection{Quantum cohomology of $X_n$ over the K\"ahler moduli space}
Let us start with the geometric construction of $X_n$.  It is a toric variety defined by the weight matrix
\begin{equation}
\label{eq:weight_matrix}
\begin{pmatrix}
1 & 1 & \cdots & 1 & 0 & -n \\
0 & 0 & \cdots & 0 & 1 & 1
\end{pmatrix}
\end{equation}
of size $2\times (n+3)$. This matrix defines a $(\bC^\times)^2$-action on $\bC^{n+3}$ and $X_n$ arises as a GIT quotient of this action. Columns of the matrix represent the classes of toric divisors $D_1,\dots,D_{n+1}, D_{n+2}, D_{n+3}$ of $X_n$ in $H^2(X_n,\bZ) \cong \Hom((\bC^\times)^2,\bC^\times)$, where $D_i$ is the zero-set $\{z_i=0\}$ of the $i$th homogeneous coordinate in $X_n= \bC^{n+3}/\!/(\bC^\times)^2$. They form the divisor diagram in Figure \ref{fig:divisor_diagram}.
We have two GIT chambers (for stability conditions):  the right chamber (the first quadrant) gives rise to $X_n$ and the left chamber (spanned by $D_{n+2}, D_{n+3}$) gives rise to the weighted projective space
\[
Y_n := \bP(1^{n+1},n) = \bP(\underbrace{1,\dots,1}_{\text{$n+1$ times}}, n)
\]
as GIT quotients. The birational map $X_n \dasharrow Y_n$ is an example of toric discrepant transformations studied in \cite{Iritani:discrepant}.
\begin{figure}[htbp]
\centering
\begin{tikzpicture}[x=3pt, y=3pt]

\fill (-50,0) circle [radius = 0.3];
\fill (-40,0) circle [radius = 0.3];
\fill (-30,0) circle [radius = 0.3];
\fill (-20,0) circle [radius = 0.3];
\fill (-10,0) circle [radius = 0.3];
\fill (0,0) circle [radius=0.3];
\fill (10,0) circle [radius =0.3];
\fill (20,0) circle [radius = 0.3];
\fill (30,0) circle [radius = 0.3];

\fill (-50,10) circle [radius = 0.3];
\fill (-40,10) circle [radius = 0.3];
\fill (-30,10) circle [radius = 0.3];
\fill (-20,10) circle [radius = 0.3];
\fill (-10,10) circle [radius = 0.3];
\fill (0,10) circle [radius=0.3];
\fill (10,10) circle [radius =0.3];
\fill (20,10) circle [radius = 0.3];
\fill (30,10) circle [radius = 0.3];

\fill (-50,20) circle [radius = 0.3];
\fill (-40,20) circle [radius = 0.3];
\fill (-30,20) circle [radius = 0.3];
\fill (-20,20) circle [radius = 0.3];
\fill (-10,20) circle [radius = 0.3];
\fill (0,20) circle [radius=0.3];
\fill (10,20) circle [radius =0.3];
\fill (20,20) circle [radius = 0.3];
\fill (30,20) circle [radius = 0.3];

\fill (-50,30) circle [radius = 0.3];
\fill (-40,30) circle [radius = 0.3];
\fill (-30,30) circle [radius = 0.3];
\fill (-20,30) circle [radius = 0.3];
\fill (-10,30) circle [radius = 0.3];
\fill (0,30) circle [radius=0.3];
\fill (10,30) circle [radius =0.3];
\fill (20,30) circle [radius = 0.3];
\fill (30,30) circle [radius = 0.3];

\draw[->, very thick] (0,10) -- (0,20);
\draw [->, very thick] (0,10) -- (10,10);
\draw [->, very thick] (0,10) -- (-40,20);

\filldraw (0,10) circle [radius =0.5];
\draw (0,7) node {$0$};
\draw (10,7) node {$D_{1}$};
\draw (0,23) node {$D_{n+2}$};
\draw (-40,23) node {$D_{n+3}$};

\fill [opacity=0.3, gray] (30,10) -- (0,10) -- (0,30) -- (30,30);
\fill [opacity=0.3, yellow] (0,10) -- (-50,22.5) -- (-50,30) -- (0,30);

\draw (15,25) node {$X_n$-chamber};
\draw (-20,25) node {$Y_n$-chamber};

\end{tikzpicture}
\caption{The divisor diagram for $X_n$ (for $n=4$) and the two GIT chambers.}
\label{fig:divisor_diagram}
\end{figure}
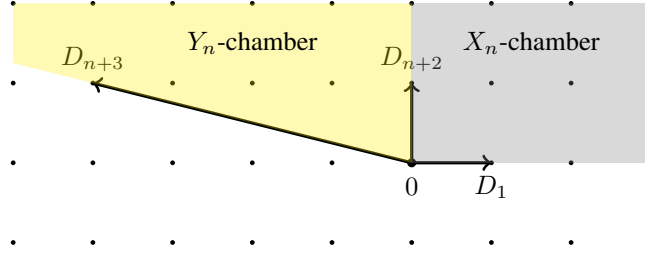

The $X_n$-chamber (spanned by $D_1,D_{n+2}$) is identified with the ample cone of $X_n$. Its dual cone is the Mori cone $\ovNE(X_n)$ that has two extremal rays. Let $q_1$, $q_2$ denote the quantum variables corresponding to the primitive generators of the extremal rays, which are dual to $(D_1,D_{n+2})$. These extremal rays respectively associate the contractions
\[
\varphi \colon X_n \to \overline{Y}_n, \quad \pi \colon X_n \to \bP^n
\]
where $\varphi$ is the divisorial contraction (contracting $D_{n+3}$) to the coarse moduli space $\overline{Y}_{n}$ of $Y_n$ and $\pi$ is the $\bP^1$-fibration. The map $\varphi$ is given by
\[
\varphi([z_1,\dots,z_{n+1},z_{n+2},z_{n+3}]) = [ z_1 z_{n+3}^{1/n},\dots,  z_{n+1} z_{n+3}^{1/n}, z_{n+2}].
\]
The variable $q_1$ represents the class of a line in the exceptional divisor $D_{n+3}\cong \bP^{n}$ and $q_2$ represents the class of a fiber of $\pi \colon X_n \to \bP^n$.

The K\"ahler moduli space $\cM_{X_n}$ of $X_n$ is $(\bC^\times)^2$ with coordinates $\bq=(q_1,q_2)$. The partially compactified K\"ahler moduli space $\ovcM_{X_n}$ is $\bC^2$.
The $\varphi$- and $\pi$-boundaries of $\ovcM_{X_n}$ are given respectively by
\[
\cC_\varphi = \text{$q_1$-axis} = \{q_2=0\}, \qquad \cC_\pi= \text{$q_2$-axis} = \{q_1=0\}.
\]

In Sections \ref{sec:preliminaries} and \ref{sec:conjecture_O_Fano_toric}, we studied the quantum cohomology of $X_n$ using mirror symmetry at the specialization $\mathbf{q}=(1,1)$. Let us consider general $\mathbf{q}$ now. The Landau-Ginzburg mirror is a family of functions parametrized by $(q_1,q_2)\in \cM_{X_n}=(\bC^\times)^2$ and is given by
\[
f_{q_1,q_2} = x_1+\cdots + x_{n+1} + y_1 + y_2
\]
where $(x_1,\dots,x_{n+1},y_1,y_2)$ is required to satisfy the constraints
\begin{equation*}
x_1\cdots x_{n+1} y_2^{-n} =q_1, \qquad y_1y_2 = q_2.
\end{equation*}
These constraints can be read off from rows of the weight matrix \eqref{eq:weight_matrix}. Choosing $(x_1,\dots,x_n,y_2)$ as independent coordinates, we can rewrite $f_{q_1,q_2}$ as the Laurent polynomial
\begin{equation}
\label{eq:LG_polynomial}
f_{q_1,q_2} = x_1 + \cdots + x_n + q_1 \frac{y_2^n}{x_1\cdots x_n} + y_2+ \frac{q_2}{y_2}.
\end{equation}
Setting $q_1=q_2=1$ gives the mirror Landau-Ginzburg model in formula \eqref{eqnd}.
We have a ring isomorphism
\[
QH(X_n)\bigr|_{q_1,q_2} \cong \Jac(f_{q_1,q_2}) \cong \bC[x_1,y_2]/(x_1^{n+1}-q_1 y_2^n, y_2(y_2+nx_1)-q_2)
\]
that sends $c_1(X_n)$ to the class of $f_{q_1,q_2}$. This ring isomorphism extends to $\ovcM_{X_n}=\bC^2$ if we define the extension of the Jacobi ring over $\ovcM_{X_n}$ as follows (see \cite{CCIT:toric_stacks_MS, Iritani:discrepant}): let $A\subset \bC[q_1,q_2,x_1^{\pm1},\dots,x_n^{\pm1}, y_2^{\pm1}]$ be the $\bC[q_1,q_2]$-subalgebra generated by all  monomials appearing in \eqref{eq:LG_polynomial} and let $I\subset A$ be the ideal generated by $x_i \partial f_{q_1,q_2}/\partial x_i$ ($i=1,\dots,n$) and $y_2 \partial f_{q_1,q_2}/\partial y_2$; then the Jacobi ring over $\ovcM_{X_n}$ is defined as $A/I$, which is finite flat as a $\bC[q_1,q_2]$-module.
Understanding the extension in this way, we calculate its restriction to boundary points in $\cC_\varphi = \text{($q_1$-axis)}$ and $\cC_\pi =\text{($q_2$-axis)}$ as follows:
\begin{align*}
\Jac(f_{q_1,0}) &\cong \bC[x_1,y_2]/(x_1^{n+1}-q_1 y_2^n, y_2(y_2+n  x_1)), \\
\Jac(f_{0,q_2}) & \cong \bC[x_1, y_2]/(x_1^{n+1}, y_2(y_2+nx_1)-q_2).
\end{align*}
They are isomorphic to the $\varphi$- and $\pi$-exceptional quantum cohomology. The spectrum of $\Jac(f_{q_1,0})$ for $q_1\neq 0$ consists of two points (one `fat' point of multiplicity $2n+1$ and one reduced point) and the spectrum of $\Jac(f_{0,q_2})$ for $q_2\neq 0$ consists of two `fat' points of multiplicity $n+1$. In the next section we observe how $\Spec(\Jac(f_{q_1,q_2}))$ --- which consists of $2n+2$ points for generic $q_1,q_2$ --- clusters into these spectra, in terms of critical values.

As discussed in \cite{Iritani:discrepant}, the family of rings $\Jac(f_{q_1,q_2})$ also contains the big quantum cohomology of $Y_n$. To see this, we add another boundary $\{q_1=\infty\}$ to $\ovcM_{X_n}$. Introduce new parameters $\frq_1 = q_1^{-1/n}$ and $\frq_2 = q_1^{1/n} q_2$. Under the change of the variables $y_2=q_2/w$, the Landau-Ginzburg mirror can be rewritten as
\[
f_{q_1,q_2} = \frf_{\frq_1,\frq_2} :=x_1+\cdots+x_n+ \frac{\frq_2^n}{x_1\cdots x_n w^n} + w + \frac{\frq_1 \frq_2}{w}.
\]
Along $\{\frq_1 = 0\}=\{q_1=\infty\}$, this restricts to the Landau-Ginzburg potential $\frf_{0,\frq_2}$ mirror to the small (orbifold) quantum cohomology of $Y_n$, i.e.~
\[
QH(Y_n)\bigr|_{\frq_2} \cong \Jac(\frf_{0,\frq_2}).
\]
By gluing the two families $\{f_{q_1,q_2}\}$, $\{\frf_{\frq_1,\frq_2}\}$ of Landau-Ginzburg potentials, we obtain a global mirror family over the toric orbifold $\ovcM_{X_n} \cup \{\frq_1=0\}$ that is associated with the fan as in Figure \ref{fig:divisor_diagram}. See Figure \ref{fig:global_Kaehler_moduli}.
Note that the Jacobi ring $\Jac(\frf_{0,\frq_2})$ is of rank $2n+1$ as one of the critical points of $f_{q_1,q_2}$ goes to infinity in this limit. Near the boundary locus $\{\frq_1=0\}$, therefore, the Jacobi ring $\Jac(f_{q_1,q_2})$ contains a $(2n+1)$-dimensional subring as a direct summand. We can identify this subring with the big quantum cohomology deformation of $QH(Y_n)$, where $\frq_1$ corresponds approximately to the coordinate dual to the twisted sector $\frac{1}{n}(1,\dots,1)$ (supported at the origin of $[\bC^{n+1}/\mu_n] \subset Y_n$ with $\mu_n$ denoting the cyclic group of order $n$).

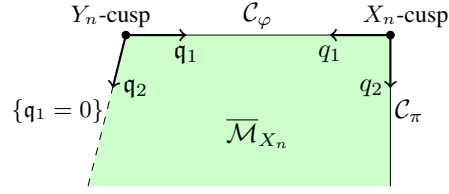
\begin{figure}[htbp]
\centering
\begin{tikzpicture}
\draw[dashed] (0,0) -- (0.5,2);
\draw (0.5,2) -- (4,2) -- (4,0);
\fill[opacity=0.2,green]  (0,0) -- (0.5,2) -- (4,2) -- (4,0);
\filldraw (0.5,2) circle [radius=0.05];
\filldraw (4,2) circle [radius =0.05];
\draw (0.3,2.25) node {\small $Y_n$-cusp};
\draw (4.2,2.25) node {\small $X_n$-cusp};
\draw[->, thick] (0.5,2) -- (0.325,1.3);
\draw[->, thick] (0.5,2) -- (1.3,2);
\draw[->, thick] (4,2) -- (3.2,2);
\draw[->, thick] (4,2) -- (4,1.3);
\draw (0.63,1.3) node {$\frq_2$};
\draw (1.3,1.75) node {$\frq_1$};
\draw (3.2,1.75) node {$q_1$};
\draw (3.75,1.3) node {$q_2$};
\draw (2.25,2.25) node {$\cC_\varphi$};
\draw (4.25,1) node {$\cC_\pi$};
\draw (-0.4,1) node {\small $\{\frq_1=0\}$};

\draw (2.25, 0.7) node {$\ovcM_{X_n}$};

\end{tikzpicture}
\caption{The partially compactified K\"ahler moduli space $\ovcM_{X_n}$ together with the (dashed) infinity line $\{\frq_1=0\}=\{q_1=\infty\}$ which is identified with $\ovcM_{Y_n}$. }
\label{fig:global_Kaehler_moduli}
\end{figure}

\subsection{Tropical study of critical points}
\label{subsec:tropical_crit}
We study the limiting behaviour of the quantum cohomology ring of $X_n$ (or equivalently, the Jacobi ring $\Jac(f_{q_1,q_2})$) near the boundaries $\cC_\pi$ and $\cC_\varphi$.
More specifically, we study the behaviour of critical values of the Landau-Ginzburg potential $f_{q_1,q_2}$ above.
We start by noting that there is a conformal symmetry (given by the Euler vector field):
\begin{equation}
\label{eq:conformal_symmetry}
x_1 \to e^s x_1,\quad \dots \quad x_n \to e^s x_n, \quad y_2 \to e^s y_2, \quad q_1 \to e^s q_1, \quad q_2 \to e^{2s} q_2
\end{equation}
under which the Landau-Ginzburg potential \eqref{eq:LG_polynomial} is scaled by $e^s$. Hence it suffices to consider a family where $q_1$ is fixed and only $q_2$ changes (the essential parameter is $q_2/q_1^2$).
Let $0<\sfT \ll 1$ be a small tropical parameter and consider the following family:
\[
f_{q_1,q_2 \sfT^\lambda}(x_1,\dots,x_n,y_2) = x_1 + \cdots + x_n + q_1 \frac{y_2^n}{x_1\cdots x_n} + y_2+ \frac{q_2 \sfT^\lambda}{y_2}
\]
where $\lambda \in \bR$. This Laurent polynomial represents a point close to the $\pi$-boundary when $\lambda<0$ and a point close to the $\varphi$-boundary when $\lambda>0$. Following the method in Section \ref{subsec:Example_Xn}, we find that critical points of $f_{q_1,q_2 \sfT^\lambda}$ are given as
\begin{equation*}
x_1 = \cdots = x_n = q_1^{\frac{1}{n+1}} t^{n}, \quad y_2 = t^{n+1}
\end{equation*}
for $t$ satisfying the algebraic equation:
\begin{equation}
\label{eq:t_equation}
t^{2n+2} + n q_1^{\frac{1}{n+1}} t^{2n+1} - q_2 \sfT^\lambda =0.
\end{equation}
The corresponding critical value $u$ of $f_{q_1,q_2 \sfT^\lambda}$ is
\begin{equation}
\label{eq:u_equation}
u = 2 t^{n+1} + (2n+1) q_1^{\frac{1}{n+1}} t^n.
\end{equation}
We solve for $t$ and $u$ as Puiseux series in $\sfT$. The lowest order exponent $e$ of $\sfT$ in $t$ (so that $t = a \sfT^e + \text{(higher order terms)}$ with $a\neq 0$) is such that the minimum
\[
\min\{ (2n+2) e, (2n+1) e, \lambda \}
\]
is attained by more than one term. We find that such an $e$ is given as
\[
e = \begin{cases}
0 \ \text{ or } \ \lambda/(2n+1) & \text{for $\lambda \ge 0$;}  \\
\lambda/(2n+2) & \text{for $\lambda\le 0$.}
\end{cases}
\]
This shows a bifurcation of critical points at $\lambda=0$: see Figure \ref{fig:tropical_bifurcation}.

\begin{figure}[htbp]
\centering
\begin{tikzpicture}
\draw[->] (-3,0) -- (3,0);
\draw[->] (0,-2) -- (0,2);
\draw[very thick] (0,0) -- (2.7,0);
\draw[very thick] (0,0) -- (2.7,1.2);
\draw[very thick] (0,0) -- (-2.7,-1.35);

\draw (-3,-1.35) node {$\frac{\lambda}{2}$};
\draw (3.3,1.2) node {$\frac{n}{2n+1} \lambda$};

\draw (3.25,0) node {$\lambda$};
\draw (0,2.15) node {exponent of $\sfT$};

\fill[opacity =0.2, gray] (0,2) -- (3,2) -- (3,-2) -- (0,-2);
\fill[opacity =0.2, yellow] (0,2) -- (-3,2) -- (-3,-2) -- (0,-2);

\draw (-1.5, -1.8) node {near $\pi$-boundary};
\draw (1.5,-1.8) node {near $\varphi$-boundary};

\end{tikzpicture}
\caption{Bifurcation of the leading exponent (valuation) of the critical values $u$ (for $n=4$).}
\label{fig:tropical_bifurcation}
\end{figure}
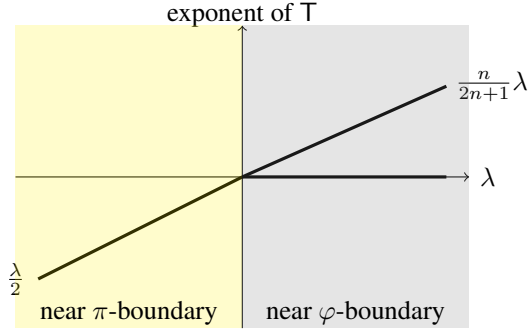

The Puiseux expansions of the critical values $u$ with respect to $\sfT$ start as follows: for $\lambda>0$, we have
\begin{align}
\nonumber
u & \sim
\begin{cases}
(2n+1) n^{-\frac{n}{2n+1}} (q_1 q_2^n)^{\frac{1}{2n+1}} \sfT^{\frac{n}{2n+1} \lambda} & \text{($2n+1$ solutions)} \\
(-n)^n q_1 & \text{(one solution)}
\end{cases}
\intertext{and for $\lambda<0$, we have }
\label{eq:crit_values_pi-boundary}
u & \sim
2 q_2^{\frac{1}{2}} \sfT^{\frac{\lambda}{2}} + (n+1) (q_1 q_2^{\frac{n}{2}})^{\frac{1}{n+1}}\sfT^{\frac{n}{2n+2} \lambda}  \quad \ \text{($2n+2$ solutions)}.
\end{align}
They give two pictures for the critical values (or the eigenvalues of $(c_1(X_n)\star_\bq)$) near the boundaries $\cC_\pi$, $\cC_\varphi$. See Figures \ref{fig:critical_values_even} and \ref{fig:critical_values_odd} for these pictures when $n$ is even/odd.

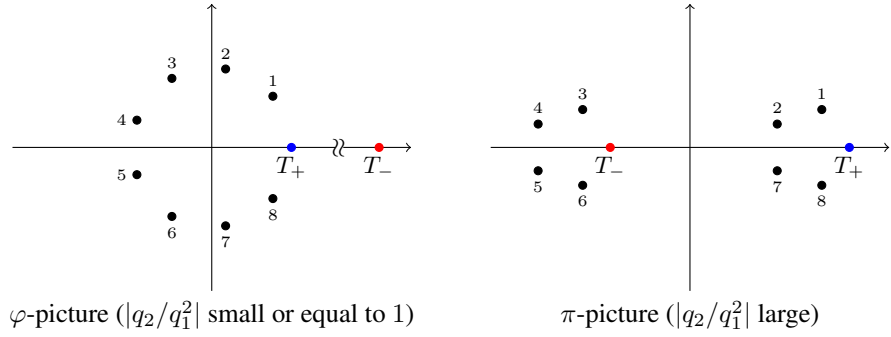
\begin{figure}[htbp]
\centering
\begin{tikzpicture}[x=30pt,y=30pt]
\draw (-4,0) -- (0.05,0);
\draw[->] (0.15,0) -- (1,0);
\draw[->] (-1.5,-1.8) -- (-1.5,1.8);
\draw[->] (2,0) -- (7,0);
\draw[->] (4.5,-1.8) -- (4.5,1.8);
\draw (-1.5,-2.1) node {$\varphi$-picture ($|q_2/q_1^2|$ small or equal to $1$)};
\draw (4.5,-2.1) node {$\pi$-picture ($|q_2/q_1^2|$ large)};

\filldraw[color=blue] (-0.5,0) circle [radius=0.05];
\filldraw (-0.733,0.642) circle [radius =0.05];
\filldraw (-1.326,0.984) circle [radius=0.05];
\filldraw (-2,0.866) circle [radius =0.05];
\filldraw (-2.439,0.342) circle [radius =0.05];
\filldraw (-0.733,-0.642) circle [radius =0.05];
\filldraw (-1.326,-0.984) circle [radius=0.05];
\filldraw (-2,-0.866) circle [radius =0.05];
\filldraw (-2.439,-0.342) circle [radius =0.05];

\draw (-0.733,0.84) node {\tiny $1$};
\draw (-1.326,1.18) node {\tiny $2$};
\draw (-2,1.066) node {\tiny $3$};
\draw (-2.63,0.342) node {\tiny $4$};
\draw (-2.63,-0.342) node {\tiny $5$};
\draw (-2,-1.066) node {\tiny $6$};
\draw (-1.326,-1.18) node {\tiny $7$};
\draw (-0.733,-0.84) node {\tiny $8$};

\draw (0.1,0) node {\rotatebox{90}{$\approx$}};
\filldraw[color=red] (0.6,0) circle [radius =0.05];

\draw (0.6,-0.25) node {\small $T_-$};
\draw (-0.5,-0.25) node {\small $T_+$};

\filldraw[color=red] (3.5,0) circle [radius =0.05];
\filldraw (3.154,0.475) circle [radius =0.05];
\filldraw (2.595,0.293) circle [radius =0.05];
\filldraw (3.154,-0.475) circle [radius =0.05];
\filldraw (2.595,-0.293) circle [radius =0.05];

\draw (3.154,0.675) node {\tiny $3$};
\draw (2.595, 0.493) node {\tiny $4$};
\draw (2.595, -0.493) node {\tiny $5$};
\draw (3.154, -0.675) node {\tiny $6$};

\filldraw[color=blue] (6.5,0) circle [radius =0.05];
\filldraw (6.154,0.475) circle [radius =0.05];
\filldraw (5.595,0.293) circle [radius =0.05];
\filldraw (6.154,-0.475) circle [radius =0.05];
\filldraw (5.595,-0.293) circle [radius =0.05];

\draw (6.154,0.675) node {\tiny $1$};
\draw (5.595,0.493) node {\tiny $2$};
\draw (5.595,-0.493) node {\tiny $7$};
\draw (6.154,-0.675) node {\tiny $8$};

\draw (6.5,-0.25) node {\small $T_+$};
\draw (3.5,-0.25) node {\small $T_-$};

\end{tikzpicture}
\caption{Critical values of $f_{q_1,q_2}$ for $n=4$. We let $q_1,q_2$ be positive real numbers and vary $q_2/q_1^2$; the critical values marked by the same number correspond to each other.}
\label{fig:critical_values_even}
\end{figure}

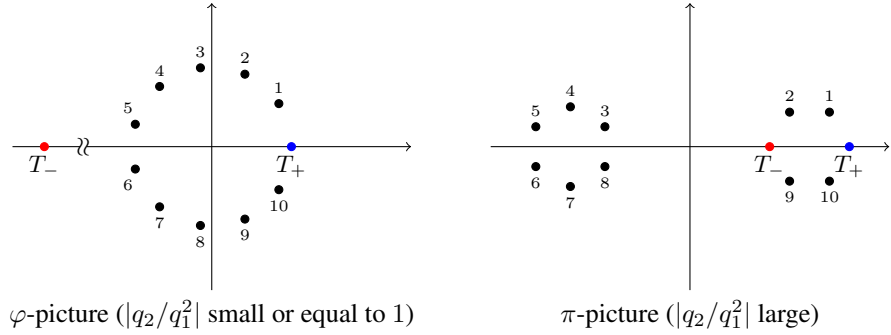
\begin{figure}[htbp]
\centering
\begin{tikzpicture}[x=30pt,y=30pt]
\draw (-4,0) -- (-3.15,0);
\draw[->] (-3.07,0) -- (1,0);
\draw[->] (-1.5,-1.8) -- (-1.5,1.8);
\draw[->] (2,0) -- (7,0);
\draw[->] (4.5,-1.8) -- (4.5,1.8);
\draw (-1.5,-2.1) node {$\varphi$-picture ($|q_2/q_1^2|$ small or equal to $1$)};
\draw (4.5,-2.1) node {$\pi$-picture ($|q_2/q_1^2|$ large)};

\filldraw[color=blue] (-0.5,0) circle [radius=0.05];
\filldraw (-0.658, 0.54) circle [radius =0.05];
\filldraw (-1.084,0.909) circle [radius=0.05];
\filldraw (-1.642,0.989) circle [radius=0.05];
\filldraw (-2.154,0.755) circle [radius=0.05];
\filldraw (-2.459,0.281) circle [radius=0.05];
\filldraw (-0.658, -0.54) circle [radius =0.05];
\filldraw (-1.084,-0.909) circle [radius=0.05];
\filldraw (-1.642,-0.989) circle [radius=0.05];
\filldraw (-2.154,-0.755) circle [radius=0.05];
\filldraw (-2.459,-0.281) circle [radius=0.05];


\draw (-0.658,0.74) node {\tiny $1$};
\draw (-1.084,1.109) node {\tiny $2$};
\draw (-1.642,1.189) node {\tiny $3$};
\draw (-2.154, 0.955) node {\tiny $4$};
\draw (-2.549,0.481) node {\tiny $5$};
\draw (-0.658,-0.74) node {\tiny $10$};
\draw (-1.084,-1.109) node {\tiny $9$};
\draw (-1.642,-1.189) node {\tiny $8$};
\draw (-2.154, -0.955) node {\tiny $7$};
\draw (-2.549,-0.481) node {\tiny $6$};

\draw (-3.1,0) node {\rotatebox{90}{$\approx$}};
\filldraw[color=red] (-3.6,0) circle [radius =0.05];

\draw (-3.6,-0.25) node {\small $T_-$};
\draw (-0.5,-0.25) node {\small $T_+$};

\filldraw (3.433,0.25) circle [radius =0.05];
\filldraw (3,0.5) circle [radius =0.05];
\filldraw (2.566,0.25) circle [radius =0.05];
\filldraw (2.566,-0.25) circle [radius =0.05];
\filldraw (3.433,-0.25) circle [radius =0.05];
\filldraw (3,-0.5) circle [radius =0.05];

\draw (3.433,0.45) node {\tiny $3$};
\draw (3,0.7) node {\tiny $4$};
\draw (2.566,0.45) node {\tiny $5$};
\draw (3.433,-0.45) node {\tiny $8$};
\draw (3,-0.7) node {\tiny $7$};
\draw (2.566,-0.45) node {\tiny $6$};

\filldraw[color=blue] (6.5,0) circle [radius =0.05];
\filldraw (6.25,0.433) circle [radius =0.05];
\filldraw (5.75,0.433) circle [radius =0.05];
\filldraw[color=red] (5.5,0) circle [radius =0.05];
\filldraw (6.25,-0.433) circle [radius =0.05];
\filldraw (5.75,-0.433) circle [radius =0.05];

\draw (6.25,0.633) node {\tiny $1$};
\draw (5.75,0.633) node {\tiny $2$};
\draw (6.25,-0.633) node {\tiny $10$};
\draw (5.75,-0.633) node {\tiny $9$};

\draw (6.5,-0.25) node {\small $T_+$};
\draw (5.5,-0.25) node {\small $T_-$};

\end{tikzpicture}
\caption{Two pictures of critical values of $f_{q_1,q_2}$ for $n=5$.}
\label{fig:critical_values_odd}
\end{figure}

Suppose that $q_1,q_2$ are positive real. The same argument as in Section \ref{subsec:Example_Xn} shows that $f_{q_1,q_2}$ has precisely two \emph{real} critical points, and one of them is the conifold point (whose coordinates are all positive real). Let $T_-, T_+$ denote the critical values of the two real critical points, where $T_+=T_{\text{con}}$ corresponds to the conifold point. These values move on the real line and their behaviour (as $q_2/q_1^2$ varies) depends on the parity of $n$; when $n$ is even, they must collide for some value of  $q_2/q_1^2$ (see Figures \ref{fig:critical_values_even}, \ref{fig:critical_values_odd}). We observe that the other critical values never cross the real line as far as $q_1,q_2$ are positive real, i.e.,~

\begin{lemma}
\label{lem:real_critical_values}
When $q_1$ and $q_2$ are positive real, $T_\pm$ are the only real critical values of $f_{q_1,q_2}$.
\end{lemma}
\begin{proof}
By the conformal symmetry \eqref{eq:conformal_symmetry}, it suffices to consider the case where $q_1=1$.
As discussed in Section \ref{subsec:Example_Xn}, the equation \eqref{eq:t_equation} has precisely two real solutions $t_\pm$ and they give rise to the two real critical points. We need to show that non-real solutions of \eqref{eq:t_equation} give  rise to a non-real critical value $u$ given by \eqref{eq:u_equation}.
Suppose on the contrary that we have a non-real solution $t$ of \eqref{eq:t_equation} (with $\lambda=0$) such that the corresponding $u$ is real. We may assume that $t$ is in the upper-half plane and set $t = r e^{i \theta}$ with $r>0$, $\theta \in (0,\pi)$. We have the equalities
\begin{gather*}
2 t^{n+1} + (2 n+1) t^n  = u =  2 \overline{t}^{n+1} + (2n+1) \overline{t}^n, \\
t^{2n+2} + n t^{2n+1}  = q_2 = \overline{t}^{2n+2} +n \overline{t}^{2n+1}
\end{gather*}
that imply that
\begin{equation}
\label{eq:sin_contiguity}
\sin((n+1) \theta) = -\frac{2n+1}{2r} \sin (n\theta), \quad
\sin((2n+2)\theta) = -\frac{n}{r} \sin ((2n+1) \theta).
\end{equation}
Eliminating $r$, we obtain:
\[
\sin((n+1) \theta) \sin((2n+1)\theta) = \frac{2n+1}{2n} \sin (n\theta) \sin ((2n+2) \theta).
\]
If $\sin((n+1)\theta)=0$, $\sin (n\theta) =0$ by the equation \eqref{eq:sin_contiguity}; $\sin((n+1)\theta) = \sin \theta \cos (n\theta) + \cos \theta \sin (n \theta)$ implies that $\cos(n\theta)=0$, which contradicts $\sin(n\theta)=0$. Hence $\sin ((n+1) \theta) \neq 0$ and we get
\[
\sin((2n+1) \theta) = \frac{2n+1}{n} \sin (n \theta) \cos ((n+1)\theta).
\]
We claim that this cannot hold for $\theta \in (0,\pi)$.
For general $\theta$, we have
\[
\sin((2n+1) \theta) - \frac{2n+1}{n} \sin (n \theta) \cos ((n+1)\theta)
= \frac{(2n+1) \sin \theta - \sin((2n+1)\theta)}{2n}
\]
by a product-to-sum formula. It suffices to show that $g_m(\theta) := \sin (m\theta)/\sin \theta$ has the absolute value less than $m$ for $\theta \in (0,\pi)$ and $m=2,3,4,\dots$. We can easily show this by induction on $m$ using $g_{m+1}(\theta) = \cos(m \theta) + g_m(\theta) \cos \theta$.
\end{proof}

\begin{remark}
\label{rem:eigenvalues_blowup}
Recall that $\varphi$ is a blowup at a unique singular point of $\bP(1^{n+1},n)$.
For a blowup $f\colon \widetilde{X}\to X$ of a smooth variety $X$ along a smooth subvariety $Z$ of codimension $r$, the spectrum of $(c_1(\widetilde{X})\star_q)$ near the $f$-boundary has a similar picture: there is a `convergent' group (corresponding to $QH(X)$) surrounded by $(r-1)$ many `divergent' satellite groups (each of which corresponds to $QH(Z)$), see Example \ref{exa:blowup}, \cite[Figures 10, 16]{Iritani:discrepant} and \cite[Figure 5]{Iritani:monoidal}.
However, in the smooth blowup case (unlike $\varphi \colon X_n \to \overline{Y}_n$ with even $n$), the satellite groups are centered around $C (-1)^{\frac{1}{r-1}}$ for a $C>0$ and do not contain a positive real eigenvalue for a positive real $\bq$.
\end{remark}

\begin{remark}
We studied a configuration of critical values of the mirror Landau-Ginzburg potential $f_{q_1,q_2}$, or equivalently, eigenvalues of $(c_1(X_n)\star_\bq)$ as $\bq$ approaches the boundary of $\cM_{X_n}$. For a general Fano manifold $X$, the spectrum of $(c_1(X)\star_\bq)$ up to overall $\bR_{>0}$-scaling is parametrized by a point in the reduced K\"ahler moduli space:
\[
\cM_X/\text{(real Euler flow)} \cong (H^2(X,\bR)/\bR c_1(X)) \times H^2(X,\iu \bR)/H^2(X,2\pi \iu \bZ).
\]
We can compactify it by adding a sphere $S = S(H^2(X,\bR)/\bR c_1(X))$ at infinity on the real part; the compactified space is homeomorphic to $D^{r-1} \times (S^1)^r$ for $r = \dim H^2(X)$. It is interesting to understand the limiting behaviour of the spectrum near the boundary: $S$ is decomposed into cells corresponding to extremal rays and the spectrum near each cell is approximated by that of the corresponding exceptional quantum cohomology. In the example at hand, $H^2(X_n,\bR)/\bR c_1(X_n) \cong \bR$ has two boundary points corresponding to the extremal contractions $\pi$ and $\varphi$.
\end{remark}

\subsection{$\hGamma$-integral flat sections and Lefschetz thimbles} 
\label{subsec:Gamma-flat-sections_thimbles}
Before identifying the principal asymptotic class, we take a brief detour into toric mirror symmetry. We recall the correspondence between $\hGamma$-integral flat sections for quantum connection of $X_n$ and Lefschetz thimbles of the Landau-Ginzburg mirror $f_{q_1,q_2}$, as demonstrated in \cite{Iri1}. Although our attention will be focused on the target spaces $X_n$, similar results hold true for weak-Fano toric DM stacks and their mirrors.

To simplify notation (and to emphasize that the results in this section hold for more general toric varieties), we write $X$ for the toric variety $X_n=\bP_{\bP^n}(\cO\oplus \cO(n))$ and $f_\bq = f_{q_1,q_2}$ for the Laurent polynomial mirror \eqref{eq:LG_polynomial} parametrized by $\bq\in \cM_X$. Let $N$ be the complex dimension of $X$. Then $f_\bq$ is a function on $(\bC^\times)^N$. We write $x_1,\dots,x_N$ for the variables of $f_\bq$ (which were denoted by $x_1,\dots,x_n,y_2$ in \eqref{eq:LG_polynomial}). The mirror of the quantum connection is given by the \emph{Brieskorn lattice}
\[
\Bri(f_{\bq}) := H^{N}\left(\Omega_{(\bC^\times)^N}^\bullet [z], z d + df_{\bq}\wedge \right)
= \frac{\Omega^N_{(\bC^\times)^N}[z]}{(zd + df_\bq \wedge) \Omega^{N-1}_{(\bC^\times)^N}[z]}
\]
equipped with the connection $\nabla^{\text{B}}_{z^2\partial_z}\colon \Bri(f_\bq) \to \Bri(f_\bq)$
\[
\nabla^{\text{B}}_{z^2 \partial_z} \left[\phi \frac{d\bx}{\bx}\right]  = \left[ \left( z^2 \parfrac{\phi}{z} -  f_\bq \phi - \frac{N}{2} z \phi \right) \frac{d\bx}{\bx} \right]
\]
where $\phi \in \bC[x_1^{\pm1},\dots,x_N^{\pm1},z]$ and $\frac{d\bx}{\bx} := \frac{dx_1\cdots dx_N}{x_1\cdots x_N}$. We have a mirror isomorphism of $\bC[z]$-modules with meromorphic connections:
\begin{equation}
\label{eq:mirror_isomorphism}
\Mir \colon (\Bri(f_\bq), \nabla^{\text{B}}_{z^2 \partial_z}) \cong (H^*(X)[z], \nabla^{\text{A}}_{z^2\partial_z}|_\bq)
\end{equation}
sending $\left[\frac{d\bx}{\bx}\right]$ to the identity class $1$, where $\nabla^{\text{A}}|_\bq$ stands for the quantum connection \eqref{eq:quantum_connection_q} at $\bq$. 
This isomorphism also preserves a connection in the $\bq$-direction.

The function $f_\bq \colon (\bC^\times)^N \to \bC$ defines a locally trivial fibration of $C^\infty$ manifolds away from critical values of $f_\bq$; the symplectic connection associated with the complete K\"ahler metric
 $\frac{\iu}{2} \sum_{i=1}^N d \log x_i \wedge d \log \overline{x_i}$ on $(\bC^\times)^N$ defines a well-defined parallel translate between fibres $f_\bq^{-1}(u)$ along any path on $\bC$ avoiding critical values.
Let $p \in (\bC^\times)^N$ be a non-degenerate critical point of $f_\bq$ and let $\gamma \colon [0,\infty) \to \bC$ be an injective path starting from the critical value $\gamma(0)=f_\bq(p)$ and avoiding any other critical values. We assume that $\gamma(t)$ is linear near $t=0$ and $t=\infty$. When $\gamma'(t) \in \bR_{>0} e^{\iu\theta}$ near $t=\infty$, we call such a path $\gamma$ an (infinite) \emph{vanishing path in the direction $e^{\iu\theta}$}.
To the critical point $p$ and the vanishing path $\gamma$, we associate a non-compact cycle $\Gamma\cong \bR^N$
\[
\Gamma = \Gamma_{p,\gamma} = \bigcup_{t\in [0,\infty)} C_t
\]
where $C_t\subset f_\bq^{-1}(\gamma(t))$ is the \emph{vanishing cycle} consisting of points that converge to $p$ by parallel translate along $\gamma(t)$ as $t\to +0$ (and $C_0 = \{p\}$).
When $\gamma$ is a vanishing path in the direction $e^{\iu\theta}$, we call $\Gamma_{p,\gamma}$ an (infinite) \emph{Lefschetz thimble in the direction $e^{\iu\theta}$}.
When the image of $\gamma$ is the straight ray $f_\bq(p) + \bR_{\ge 0} e^{\iu\theta}$, $\Gamma$ is the stable manifold of $p$ with respect to the downward gradient flow of the function $\Re(e^{-\iu\theta} f_\bq)$.
The infinite Lefschetz thimble $\Gamma$ gives rise to a solution of $\Bri(f_\bq)$
\begin{equation}
\label{eq:oscillatory_integral}
\Bri(f_\bq) \ni \left[\Omega \right] \longmapsto (-2\pi z)^{-N/2}\int_\Gamma e^{f_\bq(\bx)/z} \Omega
\end{equation}
that intertwines $\nabla^{\text{B}}_{z^2 \partial_z}$ with $z^2 \partial_z$. This oscillatory integral converges when $\Re(z e^{-\iu\theta})<0$ where $e^{\iu \theta}$ is the direction of $\Gamma$.
A Lefschetz thimble in the direction $e^{\iu\theta}$ gives an element of the relative homology group
\[
\Lef_{\bq,e^{\iu\theta}} := H_N((\bC^\times)^N, \{\bx : \Re(f_\bq(\bx) e^{-\iu\theta}) \ge M\};\bZ)
\]
where $M\gg 1$ is a sufficiently large number (the right-hand side does not depend on a sufficiently large $M$).
These groups $\Lef_{\bq,e^{\iu\theta}}$ form a local system over $\cM_X \times S^1$.
For a generic $\bq$, $f_\bq$ has only non-degenerate critical points $p_1,\dots,p_s$. In this case, we set $u_i = f_\bq(p_i)$ and take mutually non-intersecting vanishing paths $\gamma_i \colon [0,\infty) \to \bC$ starting from $u_i$ and extending in the direction of $e^{\iu \theta}$ (we assume $\gamma_i = \gamma_j$ when $u_i = u_j$). Such a system of paths $\gamma_1,\dots,\gamma_s$ is called a \emph{distinguished basis of vanishing paths} when they are ordered in such a way that
\[
\Im(\gamma_1(t) e^{-\iu\theta})\ge \Im(\gamma_2(t) e^{-\iu\theta})\ge \cdots \ge \Im(\gamma_s(t) e^{-\iu\theta}) \quad \text{for $t\gg 1$.}
\]
The corresponding Lefschetz thimbles $\Gamma_{p_1,\gamma_1},\dots,\Gamma_{p_s,\gamma_s}$ form a basis of $\Lef_{\bq,e^{\iu\theta}}$.

\begin{example}
\label{exa:positive_real_thimble}
Suppose that the parameter $\bq$ is positive real, i.e.~
\[
\bq\in \cM_{X,\bR} :=\Hom(H_2(X,\bZ),\bR_{>0}) \subset \cM_X.
\]
The Lefschetz thimble associated with the conifold point $\bx_{\text{con}} \in (\bR_{>0})^N$, a unique point where $f_\bq|_{(\bR_{>0})^N}$ attains a minimum, and the vanishing path $f_\bq(\bx_{\text{con}}) + \bR_{\ge 0}$ is given by $\Gamma_\bR = (\bR_{>0})^N$. We call it the \emph{positive-real Lefschetz thimble}. Note that this is well-defined (i.e.~does not contain other critical points in the closure) even when $f_\bq(\bx_{\text{con}})+\bR_{\ge 0}$ contains other critical values.
\end{example}

The main result of \cite{Iri1} concerns a relationship between oscillatory integrals  \eqref{eq:oscillatory_integral} associated with Lefschetz thimbles $\Gamma$ and $\hGamma$-integral flat sections for the quantum connection. Let $K^0(X)$ denote the $K$-group of topological $\bC$-vector bundles on $X$. The \emph{$\hGamma$-integral flat section} $\frs_V$ associated with $V\in K^0(X)$ is defined to be
(\cite[Definition 2.9]{Iri1})
\[
\frs_V(\tau,z) := (2\pi)^{-N/2} S(\tau,z) z^{-\mu} z^{c_1(X)} \left( \hGamma_X \Ch(V) \right)
\]
where $S(\tau,z) z^{-\mu} z^{c_1(X)}$ is the fundamental solution in the previous section \S\ref{subsec:A_general_q} and $\Ch(V)=\sum_{k\ge 0} (2\pi \iu)^k \ch_k(V)$ is the modified Chern character. The map $V \mapsto \frs_V(\tau,z)$ intertwines the Euler pairing with the Poincar\'e pairing $\<\cdot,\cdot\>^X$ as follows \cite[Proposition 2.10]{Iri1}:
\[
\chi(V_1,V_2) = \left\<\frs_{V_1}(\tau,e^{-\iu\pi}z), \frs_{V_2}(\tau,z)\right\>^X
\]
where $\chi(V_1,V_2)$ is given by $\sum_{k=0}^N (-1)^k \dim \Ext^k(V_1,V_2)$ if  $V_1,V_2$ are holomorphic vector bundles; otherwise, it is the $K$-theoretic push-forward of $V_1^*\otimes V_2$ to a point.
The lattice $\Lef_{\bq,1}$ on the mirror side also has a pairing. The perfect intersection pairing $\# \colon \Lef_{\bq,-1}\times \Lef_{\bq,1} \to \bZ$ induces a pairing
\[
\#(e^{-\pi \iu} \Gamma_1 \cdot \Gamma_2) = \# (\Gamma_1 \cdot e^{\pi \iu} \Gamma_2)
\]
of $\Gamma_1, \Gamma_2 \in \Lef_{\bq,1}$, where $e^{\pm \pi \iu} \Gamma \in \Lef_{\bq,-1}$ is the parallel translate of $\Gamma \in \Lef_{\bq,1}$ along the semicircle $[0,\pi]\ni \theta \mapsto (\bq,e^{\pm\iu\theta}) \in \cM_X\times S^1$.

\begin{thm}[{\cite[\S 4.3]{Iri1}}]
\label{thm:mirror_isom_lattice}
For a positive real parameter $\bq =e^\tau \in \cM_{X,\bR}$ with $\tau \in H^2(X,\bR)$, we have an isomorphism of lattices
\begin{equation}
\label{eq:mirror_isom_lattice}
\Lef_{\bq,1} \cong K^0(X), \quad \Gamma \mapsto V(\Gamma)
\end{equation}
which is locally constant in $\bq\in \cM_{X,\bR}$ such that for $z>0$,
\begin{equation}
\label{eq:oscillatory_integral_flat_section}
(2\pi z)^{-N/2} \int_{\Gamma} e^{-f_\bq(\bx)/z} \Omega_{-z}
= \left\<\iota^* (\Mir[\Omega]), \frs_{V(\Gamma)}(\tau,z) \right\>^X
\end{equation}
for $[\Omega] = [\Omega_z] \in \Bri(f_\bq)$, where $\iota^* \colon H^*(X)[z] \to H^*(X)[z]$ is the map sending $f(z)$ to $f(-z)$.
Moreover, we have
\begin{align}
\label{eq:Euler_intersection_pairings}
\begin{split}
 V(\Gamma_\bR) & = \cO_X \\
\chi(V(\Gamma_1), V(\Gamma_2)) &= (-1)^{N(N-1)/2}\#(e^{-\pi \iu}\Gamma_1 \cdot \Gamma_2)
\end{split}
\end{align}
for $\Gamma_1,\Gamma_2 \in \Lef_{\bq,1}$, where $\Gamma_\bR$ is the positive-real Lefschetz thimble in Example \ref{exa:positive_real_thimble}.
\end{thm}

The correspondence \eqref{eq:mirror_isom_lattice} between Lefschetz thimbles and $K$-classes of vector bundles is expected to be induced by homological mirror symmetry: an equivalence between the Fukaya-Seidel category of $f_\bq$ and the derived category of coherent sheaves on $X$. The recent works \cite{Fang:centralcharges, Fang-Zhou:GammaII} confirmed this expectation for a version of the Fukaya-Seidel category for toric Fano manifolds $X$, thereby showing the Gamma conjecture II for them. Note that a Lefschetz thimble is an exceptional object in the Fukaya-Seidel category; moreover, a distinguished system of vanishing paths gives rise to a full exceptional collection of Lefschetz thimbles.

\begin{remark}
\label{rem:Kouchnirenko}
We collect several references and facts omitted in the above expositions.

(1) The Brieskorn lattice for a tame function on an affine variety has been studied in \cite{Sabbah:hypergeometric, Douai-Sabbah:I}.
The Brieskorn lattice of a convenient and non-degenerate Laurent polynomial $f$ (in the sense of Kouchnirenko \cite[Definitions 1.5, 1.19]{Kouchnirenko:Newton}) is free of rank equal to $\dim(\Jac(f))$.
The Landau-Ginzburg potential $f_\bq$ mirror to a toric Fano manifold is non-degenerate for all $\bq\in \cM_X$.


(2) The mirror isomorphism \eqref{eq:mirror_isomorphism} follows from Givental's mirror theorem \cite{Gi2} for weak-Fano toric manifolds (generalized to toric DM stacks in \cite{CCIT:toric_stacks_MS}). A detailed comparison of connections (together with pairings) is given in \cite[Proposition 4.8]{Iri1}, \cite[Proposition 4.10]{Reichelt-Sevenheck:logarithmic}, \cite[Proposition 6.16]{Iritani:discrepant} for weak-Fano\footnote{In the weak-Fano case, in general, the mirror isomorphism holds for $\bq$ in a small neighbourhood of the origin $\bq=0$ in $\ovcM_X$; we also need a non-trivial mirror map between the parameter $\bq$ of the Landau-Ginzburg potentials $f_\bq$ and the K\"ahler parameter $\bq\in \cM_X$. For toric Fano manifolds, the mirror map is the identity map and the mirror isomorphism holds globally. } toric DM stacks/manifolds.

(3) The symplectic connection defines a well-defined $C^\infty$ parallel translate between regular fibres $f_\bq^{-1}(u)$ when $f_\bq$ is a convenient and non-degenerate Laurent polynomial. This follows from the proof of \cite[Corollary 7.21]{Iritani:discrepant}.

(4) That the map $\Gamma \mapsto V(\Gamma)$ preserves the pairing follows from the fact that the mirror isomorphism preserves the pairing \cite[\S 3.3.1-3.3.2, Proposition 4.8]{Iri1}. The sign factor $(-1)^{N(N-1)/2}$ was missing in \cite{Iri1} and corrected in \cite[footnote (16)]{Iri2}.
\end{remark}

Since $H^2(X,\bC)$ is a universal covering of $\cM_X$, we can uniquely extend the isomorphism \eqref{eq:mirror_isom_lattice} to any $\tau\in H^2(X,\bC)$
\begin{equation}
\label{eq:mirror_isom_lattice_extended}
\Lef_{e^\tau,1} \cong K^0(X), \quad \Gamma \mapsto V(\Gamma)
\end{equation}
so that it is locally constant in $\tau$.

\begin{thm}
\label{thm:Lefschetz_A}
Suppose that $(c_1(X)\star_\bq)$ has a simple rightmost eigenvalue $u$ for some $\bq =e^\tau \in \cM_X$. Then $u$ equals $f_\bq(p)$ for a unique non-degenerate critical point $p$ of $f_\bq$ and $\gamma\colon [0,\infty) \to \bC$, $\gamma(t) = u+ t$ gives a vanishing path. The principal asymptotic class at $\tau$ is given by
\[
A_X(\tau) = \hGamma_X \Ch(V(\Gamma_{p,\gamma}))
\]
where $V(\Gamma_{p,\gamma})$ is given by the isomorphism \eqref{eq:mirror_isom_lattice_extended} extended to $\tau\in H^2(X,\bC)$. In particular, if $\tau\in H^2(X,\bR)$ and $p$ is the conifold point (such that $u=f_\bq(p)$ is a simple rightmost eigenvalue), we have $A_X(\tau) = \hGamma_X$.
\end{thm}
\begin{proof}
The statement about $p$ and $\gamma$ follows from the isomorphism $QH(X)|_\bq \cong \Jac(f_\bq)$ sending $c_1(X)$ to $[f_\bq]$.
By analytic continuation in $\tau$, we obtain from \eqref{eq:oscillatory_integral_flat_section} that
\[
\left\<\iota^* (\Mir[\Omega]), S(\tau,z) z^{-\mu} z^{c_1(X)} \hGamma_X \Ch(V(\Gamma_{p,\gamma})) \right\>^X = z^{-N/2} \int_{\Gamma_{p,\gamma}} e^{-f_\bq(\bx)/z} \Omega_{-z}
\]
holds for all $\Omega \in \Omega^N_{(\bC^\times)^N}[z]$ where $\bq= e^\tau$.
The right-hand side has the asymptotic expansion of the form $e^{-u/z}(a_0+a_1 z + \cdots)$ as $z\to +0$ and therefore the flat section
\[
e^{u/z} S(\tau,z) z^{-\mu} z^{c_1(X)} \hGamma_X \Ch(V(\Gamma_{p,\gamma}))
\]
is of moderate growth as $z\to +0$ (see the proof of Theorem \ref{thm-gammaclassandconifoldpoint}).
The conclusion follows from the definition of the principal asymptotic class.
\end{proof}

\subsection{Identification of the principal asymptotic class via mutation}
\label{subsec:identification_A}
Finally we determine the principal asymptotic class of $X_n$ in certain regions of the K\"ahler moduli space.

The exponential map defines an isomorphism $H^2(X,\bR) \xrightarrow{\cong} \cM_{X,\bR}$ onto the positive real locus. We write $\bq\mapsto \log \bq\in H^2(X,\bR)$ for the inverse map defined for a positive real $\bq \in \cM_{X,\bR}$.

Recall from \S\ref{subsec:tropical_crit} that the critical value $T_+=T_{\text{con}}$ associated with the conifold point gives a simple rightmost eigenvalue of $(c_1(X)\star_\bq)$ when $q_2/q_1^2\gg 1$ i.e.~near the $\pi$-boundary. When $n$ is odd, the same holds also when $0<q_2/q_1^2 \ll 1$, i.e.~near the $\varphi$-boundary. Theorem \ref{thm:Lefschetz_A} immediately implies:
\begin{thm}\label{XnasympGamma}
Suppose that $q_1,q_2$ are positive real. We have $A_{X_n}(\log \bq)= \hGamma_{X_n}$ when $q_2/q_1^2$ is sufficiently large and $n\ge 1$ is arbitrary, or when $q_2/q_1^2$ is sufficiently small and $n\ge 1$ is odd.
\end{thm}

This proposition is a consequence of the fact that the positive real Lefschetz thimble $\Gamma_\bR$ corresponds to the structure sheaf (Theorem \ref{thm:mirror_isom_lattice}), i.e.~
\begin{equation}
\label{eq:Gamma_R_O}
(2\pi z)^{-N/2} \int_{\Gamma_\bR} e^{-f_\bq(\bx)/z} \Omega_{-z}
= \left\<\iota^* (\Mir[\Omega]), \frs_{\cO}(\log \bq,z) \right\>^{X_n}
\end{equation}
for $\bq \in \cM_{X_n,\bR}$ and $z>0$.
Using monodromy transformations, we can determine $V(\Gamma)$ for other Lefschetz thimbles in the region $q_2/q_1^2 \gg 1$.
Suppose now that $n$ is even and write $n=2m$. We write $H:=D_{n+2}$; $H$ is the pull-back of the ample class $\cO(n)$ on $\overline{Y}_n=\bP(1^{n+1},n)$.
From \cite[Proposition 2.10(ii)]{Iri1}, the monodromy transformations for $\hGamma$-integral flat sections along the paths $[0,2\pi] \ni \theta \mapsto ( e^{\iu\theta}q_1, q_2)$ and $[0,2\pi] \ni \theta \mapsto (q_1,e^{\iu\theta}q_2)$ in $\cM_{X_n}$ correspond respectively to tensoring with $\cO(-D_1)$ and $\cO(-H)$ on the $K$-group, i.e.,
\begin{align}
\label{eq:flatsection_monodromy}
\begin{split}
\frs_V(\log q_1 + 2\pi \iu,\log q_2,z) & = \frs_{V\otimes \cO(-D_1)}(\log q_1,\log q_2,z), \\
\frs_V(\log q_1, \log q_2 + 2\pi \iu, z) & = \frs_{V\otimes \cO(-H)}(\log q_1,\log q_2,z).
\end{split}
\end{align}
This follows from the definition of $\frs_V$ and 
\begin{align*} 
S(\tau + 2\pi\iu \xi,z) z^{-\mu} z^{c_1(X_n)} & = S(\tau,z) e^{-2\pi\iu \xi/z} z^{-\mu} z^{c_1(X_n)} \\
& = S(\tau,z) z^{-\mu} z^{c_1(X_n)} e^{-2\pi \iu \xi}
\end{align*} 
for $\xi \in H^2(X_n,\bZ)$. 
Let $\Gamma_k\in \Lef_{\bq,1}$, $k=1,\dots,m$ be the monodromy transforms of $\Gamma_\bR \in \Lef_{\bq,1}$ along the path $[0,2\pi]\ni \theta \mapsto (e^{\iu k \theta} q_1, q_2)$; the Puiseux expansion \eqref{eq:crit_values_pi-boundary} shows that they correspond to the vanishing paths starting from the $k$th node depicted in Figure \ref{fig:vanishing_paths}.
By applying monodromy transformations to \eqref{eq:Gamma_R_O} and using \eqref{eq:flatsection_monodromy}, we have
\[
(2\pi z)^{-N/2} \int_{\Gamma_k} e^{-f_\bq(x)/z} \Omega_{-z}
= \left\<\iota^* (\Mir[\Omega]), \frs_{\cO(-k D_1)}(\log \bq,z) \right\>^{X_n}
\]
for all $\Omega\in \Omega^N_{(\bC^\times)^N}[z]$; it follows that $V(\Gamma_k) = \cO(-kD_1)$ under the isomorphism \eqref{eq:mirror_isom_lattice}.
Similarly, we consider the monodromy transform of $\Gamma_\bR$ along the path $[0,2\pi] \ni \theta \mapsto (e^{-\iu m \theta} q_1, e^{\iu \theta} q_2)$; it is a Lefschetz thimble $\tGamma_-$ associated with the `greatly bent' vanishing path starting from $T_-$ as in Figure \ref{fig:vanishing_paths}. We find that $V(\tGamma_-) = \cO(-H + m D_1)$.

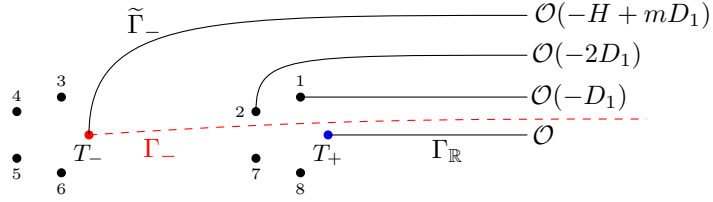
\begin{figure}[htbp]
\centering
\begin{tikzpicture}[x=30pt,y=30pt]
\draw (0,0) node {\phantom{A}};





\filldraw[color=red] (3.5,0) circle [radius =0.05];
\filldraw (3.154,0.475) circle [radius =0.05];
\filldraw (2.595,0.293) circle [radius =0.05];
\filldraw (3.154,-0.475) circle [radius =0.05];
\filldraw (2.595,-0.293) circle [radius =0.05];

\draw (3.154,0.675) node {\tiny $3$};
\draw (2.595, 0.493) node {\tiny $4$};
\draw (2.595, -0.493) node {\tiny $5$};
\draw (3.154, -0.675) node {\tiny $6$};

\filldraw[color=blue] (6.5,0) circle [radius =0.05];
\filldraw (6.154,0.475) circle [radius =0.05];
\filldraw (5.595,0.293) circle [radius =0.05];
\filldraw (6.154,-0.475) circle [radius =0.05];
\filldraw (5.595,-0.293) circle [radius =0.05];

\draw (6.154,0.675) node {\tiny $1$};
\draw (5.395,0.293) node {\tiny $2$};
\draw (5.595,-0.493) node {\tiny $7$};
\draw (6.154,-0.675) node {\tiny $8$};

\draw (6.5,-0.25) node {\small $T_+$};
\draw (3.5,-0.25) node {\small $T_-$};

\draw (6.5,0) -- (9,0);
\draw (6.154,0.475) -- (9,0.475);
\draw (5.595,0.293) .. controls (5.595,1) and (6,1) .. (9,1);
\draw (3.5,0) .. controls (3.5,1.5) and (5,1.5) .. (9,1.5);
\draw[color=red,dashed] (3.5,0) .. controls (6.5,0.2) and (8,0.2) .. (10.5,0.2);

\draw (4.2,1.4) node {$\tGamma_-$};
\draw[color=red] (4.4,-0.2) node {$\Gamma_-$};
\draw (8,-0.2) node {$\Gamma_\bR$};
\draw (9.2,0) node {$\cO$};
\draw (9.65,0.475) node {$\cO(-D_1)$};
\draw (9.75,1) node {$\cO(-2D_1)$};
\draw (10.2,1.5) node {$\cO(-H+mD_1)$};

\end{tikzpicture}

\caption{Vanishing paths and corresponding coherent sheaves (when $q_2/q_1^2$ large and $n=4$):  thimbles of solid paths are all given by monodromy transformations from the positive-real Lefschetz thimble $\Gamma_\bR$.}
\label{fig:vanishing_paths}
\end{figure}

We want to find the class of the Lefschetz thimble $\Gamma_-$ corresponding to the  dashed vanishing path from $T_-$ (straight ray in the direction $e^{\iu\varepsilon}$ with $0<\varepsilon \ll 1$) in Figure \ref{fig:vanishing_paths} using mutation (or the Picard-Lefschetz transformation).
Let us first recall the rudiments of mutation for Lefschetz thimbles or (complexes of) coherent sheaves.
Figure \ref{fig:right_mutation} shows an example of the right mutation $R_\Delta \Gamma$ of a Lefschetz thimble $\Gamma$ with respect to another one $\Delta$. By the Picard-Lefschetz monodromy formula (see e.g.~\cite{Arnold-Gusein-Zade-VarchenkoII}), we find that the relative homology class of $\Gamma$ equals the sum of $k \Delta$ and $R_\Delta \Gamma$, where $k$ is given by the intersection number of vanishing cycles $C_\Gamma$, $C_\Delta$ associated with $\Gamma, \Delta$:
\[
k =(-1)^{(N-1)(N-2)/2} \#(C_\Gamma \cdot C_\Delta) =(-1)^{N(N-1)/2} \#(e^{-\pi \iu} \Gamma \cdot \Delta).
\]
Therefore we have
\[
[R_\Delta \Gamma] = [\Gamma] - (-1)^{N(N-1)/2} \#(e^{-\pi\iu} \Gamma \cdot \Delta) [\Delta].
\]
When we write $e, f, R_e f\in K^0(X)$ for the $K$-classes corresponding to $\Delta, \Gamma, R_\Delta \Gamma$ under the isomorphism \eqref{eq:mirror_isom_lattice},  we have
\[
R_e f = f - \chi(f,e) e
\]
by \eqref{eq:Euler_intersection_pairings}. We define the right mutation $R_e f$ of $f\in K^0(X)$ with respect to $e\in K^0(X)$ by this formula. This operation can be lifted to the derived category. The right mutation of an object $\cF\in D^b(X)$ with respect to $\cE \in D^b(X)$ is defined to be
\[
R_\cE \cF = \Cone(\cF \to \Hom^\bullet(\cF,\cE)^\vee \otimes \cE) [-1]
\]
where $\Hom^\bullet(\cF,\cE)^\vee \otimes \cE = \bigoplus_k \Hom^k(\cF,\cE)^\vee \otimes \cE[k]$ and $\cF \to \Hom^\bullet(\cF,\cE)^\vee \otimes \cE$ is the co-evaluation map.
This operation preserves exceptional collections, i.e.~if $(\cE_1,\dots,\cE_i, \cE_{i+1}, \dots, \cE_n)$ is an exceptional collection in $D^b(X)$, then so is $(\cE_1,\dots,\cE_{i+1}, R_{\cE_{i+1}}\cE_i, \dots, \cE_n)$. We refer the reader to \cite{Bondal-Polishchuk} for the details.

\begin{figure}[htbp]
\centering
\begin{tikzpicture}
\filldraw (0,0) circle[radius=0.03];
\filldraw (-0.7,-0.7) circle[radius=0.03];

\draw[color=blue, dashed] (0,0)--(2,0);
\draw (-0.7,-0.7) .. controls (-0.7,0.7) and (0.7,0.7) .. (2,0.7);
\draw[color=blue] (2.2,0) node {$\Delta$};
\draw (2.2,0.7) node {$\Gamma$};

\draw (2.9,-0.1) node {$\rightsquigarrow$};

\filldraw (4,0) circle[radius =0.03];
\filldraw (3.3,-0.7) circle[radius=0.03];
\draw (3.3,-0.7) -- (4,-0.2);
\draw (4,-0.2) arc[radius=0.2, start angle = 270, end angle = -70];
\draw (4.07,-0.19) -- (6.3,-0.19);
\draw (6.5,-0.19) node {$\Gamma$};

\draw (7.1,-0.1) node {$\rightsquigarrow$};

\filldraw (8.2,0) circle[radius=0.03];
\filldraw (7.5,-0.7) circle[radius=0.03];

\draw (8.2,0)--(10.3,0);
\draw (7.5,-0.7) -- (10.3,-0.7);
\draw (10.7,0) node {$k \Delta$};
\draw (10.7,-0.7) node {$R_{\Delta}\Gamma$};

\end{tikzpicture}
\caption{Right mutation of Lefschetz thimbles.}
\label{fig:right_mutation}
\end{figure}
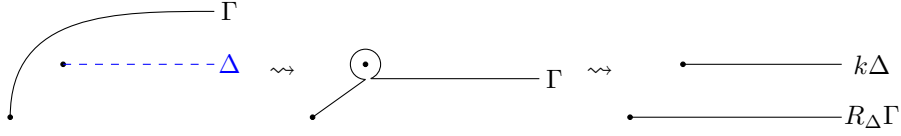

\begin{lemma}
Suppose that $n = 2m$ is even. In the derived category $D^b(X_n)$ of coherent sheaves, we have
\[
R_{\cO(-D_1)} R_{\cO(-2D_1)} \cdots R_{\cO(-mD_1)} \cO(-H + mD_1) =  \cO_E(-m)[-1]
\]
where $E := D_{n+3}\cong \bP^n$ is the exceptional divisor of $\varphi \colon X_n\to \overline{Y}_n$.
Moreover, this object is orthogonal\footnote{Two objects $A,B\in D^b(X)$ are said to be orthogonal if  $\Hom^\bullet(A,B) = \Hom^\bullet(B,A) = 0$.} to $\cO(-jD_1)$ when $|j| \le m-1$.
\end{lemma}
\begin{proof}
We begin by computing $R_{\cO(-mD_1)} \cO(-H + mD_1)$. We have
\[
\Hom^\bullet(\cO(-H+mD_1), \cO(-mD_1)) = H^\bullet(X_n,\cO(-2m D_1+H)) = H^\bullet(X_n,\cO(E))
\]
by $E = D_{n+3} = D_{n+2} - n D_1$. Using the long exact sequence associated with
\begin{equation}
\label{eq:short_exact_E}
0 \to \cO \to \cO(E) \to \cO_E(E) \cong \cO_{\bP^n}(-n) \to 0
\end{equation}
and the Kodaira vanishing for $\cO_{\bP^n}(-n) = \cO_{\bP^n}(1) \otimes K_{\bP^n}$, we get $H^\bullet(X_n, \cO(E)) = H^\bullet(X_n, \cO) = \bC$.
Hence
\begin{align*}
R_{\cO(-mD_1)} \cO(-H+mD_1) & = \Cone(\cO(-H+mD_1) \to \cO(-mD_1))[-1] \\
& \cong \cO(-mD_1) \otimes \Cone(\cO(-E) \to \cO)[-1] \\
& \cong \cO(-mD_1) \otimes \cO_E[-1] \cong \cO_E(-m)[-1].
\end{align*}
We denote this object by $\cE$. We claim that $\cE$ is orthogonal to $\cO(-jD_1)$ when $|j| \le m-1$.
Applying $\Hom^\bullet(-,\cO(-jD_1))$ to the triangle
\[
\cO(-E-mD_1) \to \cO(-mD_1) \to \cE[1],
\]
we obtain the long exact sequence
\begin{equation}
\label{eq:Hom_E_linebundles}
\Hom^\bullet(\cE[1],\cO(-jD_1)) \to H^\bullet(X_n, \cO((m-j) D_1)) \to H^\bullet(X_n, \cO(E+(m-j) D_1)) \to \cdots.
\end{equation}
Using the long exact sequence associated with \eqref{eq:short_exact_E} tensored by $\cO((m-j) D_1)$ and the Kodaira vanishing for $\cO_E(E+(m-j) D_1) \cong \cO_{\bP^n}(-(m+j)) \cong \cO_{\bP^n}(m-j+1) \otimes K_{\bP^n}$, we find that the map
\[
H^\bullet(X_n, \cO((m-j)D_1)) \to H^\bullet(X_n, \cO((m-j)D_1+ E))
\]
induced by the natural inclusion $\cO \hookrightarrow \cO(E)$ is an isomorphism. Therefore $\Hom^\bullet(\cE[1], \cO(-jD_1)) = 0$ by \eqref{eq:Hom_E_linebundles}.
Also we have
\begin{align*}
\Hom^\bullet(\cO(-jD_1),\cE)
& \cong H^{\bullet-1}(\bP^n, \cO_{\bP^n}(-(m-j))) =0
\end{align*}
again by the Kodaira vanishing. The claim follows. It implies that
\[
R_{\cO(-D_1)} \cdots R_{\cO(-(m-1)D_1)} \cE = \cE
\]
and the conclusion follows.
\end{proof}

Because the mutation of Lefschetz thimbles corresponds to the mutation of $K$-classes,
the above lemma shows that $V(\Gamma_-) = -[\cO_E(-m)]$ when $q_2/q_1^2$ is large.
Furthermore, since this class is orthogonal to $\cO$, the classes of the Lefschetz thimbles associated with the vanishing paths $T_- + e^{\iu\epsilon} \bR_{\ge 0}$ and $T_- + e^{-\iu\epsilon} \bR_{\ge 0}$ are the same for a sufficiently small $\epsilon>0$.

Now we vary $\bq$ within the positive real locus and pass from the $\pi$-picture ($q_2/q_1^2$ large) to the $\varphi$-picture ($q_2/q_1^2$ small).  Due to orthogonality, there is no interaction (mutation) between the thimbles $\Gamma_-$ and $\Gamma_\bR$ associated with $T_-$ and $T_+=T_{\text{con}}$ respectively, even though $T_\pm$ collide. Lemma \ref{lem:real_critical_values} guarantees no interactions with other critical points either. Therefore we have $V(\Gamma_-) = -[\cO_E(-m)]$ for all $q_1,q_2>0$ and conclude the following by Theorem \ref{thm:Lefschetz_A}:
\begin{thm}
\label{thm:A_for_Twrong}
Let $n\ge 2$ be even. Suppose that $q_1,q_2$ are positive real and that $T_-$ is a simple rightmost eigenvalue of $(c_1(X_n)\star_\bq)$; this happens when $q_2/q_1^2$ is sufficiently small or when $n\ge 4$ and $q_2/q_1^2 =1$.
Then the principal asymptotic class is given by
\[
A_{X_n}(\log \bq) = \hGamma_{X_n} \Ch(\cO_E(-m)).
\]
\end{thm}

\begin{remark}
The principal asymptotic class in Theorem \ref{thm:A_for_Twrong} is a purely imaginary class, confirming the latter half of Corollary \ref{cor:A_easy_properties}. The limit
\[
 \lim_{t\to\infty} \frac{J_X(c_1 \log t, 1)}{\<J_X(c_1\log t,1),[\pt]\>^X}
\]
appearing in the original Gamma conjecture I \cite[Corollary 3.6.9]{GGI} does not exist in this case, as $A_{X_n}$ is supported on the exceptional divisor and $\<A_{X_n},[\pt]\>^{X_n}$ vanishes. Instead, we need to consider the limit in the projective space $\bP(H^*(X))$ as in Proposition \ref{prop:limit_formula_general_q}.
\end{remark}

Using monodromy in the $\pi$-picture and passing to the $\varphi$-picture, we can determine the $K$-classes corresponding to a distinguished basis of vanishing paths. See Figure \ref{fig:distinguished_basis} for the $n=4$ case.

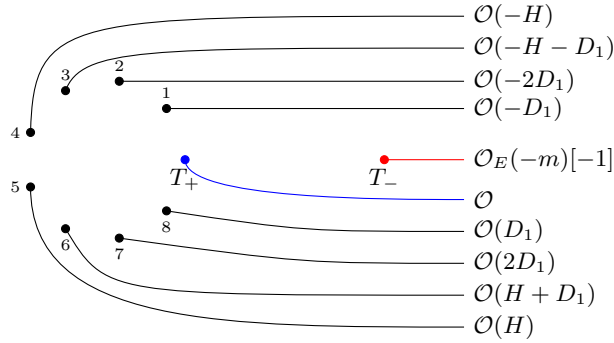
\begin{figure}[htbp]
 \centering

\begin{tikzpicture}[x=30pt,y=30pt]
\draw (-5,0) node {\phantom{a}};

\filldraw[color=blue] (-0.5,0) circle [radius=0.05];
\filldraw (-0.733,0.642) circle [radius =0.05];
\filldraw (-1.326,0.984) circle [radius=0.05];
\filldraw (-2,0.866) circle [radius =0.05];
\filldraw (-2.439,0.342) circle [radius =0.05];
\filldraw (-0.733,-0.642) circle [radius =0.05];
\filldraw (-1.326,-0.984) circle [radius=0.05];
\filldraw (-2,-0.866) circle [radius =0.05];
\filldraw (-2.439,-0.342) circle [radius =0.05];

\draw (-0.733,0.84) node {\tiny $1$};
\draw (-1.326,1.18) node {\tiny $2$};
\draw (-2,1.066) node {\tiny $3$};
\draw (-2.63,0.342) node {\tiny $4$};
\draw (-2.63,-0.342) node {\tiny $5$};
\draw (-2,-1.066) node {\tiny $6$};
\draw (-1.326,-1.18) node {\tiny $7$};
\draw (-0.733,-0.84) node {\tiny $8$};

\filldraw[color=red] (2,0) circle [radius =0.05];

\draw (2,-0.25) node {\small $T_-$};
\draw (-0.5,-0.25) node {\small $T_+$};

\draw[color=red] (2,0) -- (3,0);
\draw[color=blue] (-0.5,0) .. controls (-0.5,-0.5) and (2,-0.5) .. (3,-0.5);
\draw (-0.733,0.642) -- (3,0.642);
\draw (-1.326,0.984) -- (3,0.984);
\draw (-2,0.866) .. controls (-1.8,1.4) and (-1,1.4) .. (3,1.4);
\draw (-2.439,0.342) .. controls (-2.4,1.8) and (-2,1.8) .. (3,1.8);
\draw (-0.733,-0.642) .. controls (1,-0.9) ..  (3,-0.9);
\draw (-1.326,-0.984) .. controls (1,-1.3).. (3,-1.3);
\draw (-2,-0.866) .. controls (-1.5,-1.7) .. (3,-1.7);
\draw (-2.439,-0.342) .. controls (-2.439,-2.1) and (0,-2.1) .. (3,-2.1);

\draw (3.0,1.8) node[right] {\small $\cO(-H)$};
\draw (3.0,1.4) node[right] {\small $\cO(-H-D_1)$};
\draw (3,0.984) node[right] {\small $\cO(-2D_1)$};
\draw (3,0.642) node[right] {\small $\cO(-D_1)$};
\draw (3,0) node[right] {\small $\cO_E(-m)[-1]$};
\draw (3,-0.5) node[right] {\small $\cO$};
\draw (3,-0.9) node[right] {\small $\cO(D_1)$};
\draw (3,-1.3) node[right] {\small $\cO(2D_1)$};
\draw (3,-1.7) node[right] {\small $\cO(H+D_1)$};
\draw (3,-2.1) node[right] {\small $\cO(H)$};

\end{tikzpicture}
\caption{A distinguished basis of vanishing paths and the corresponding objects in $D^b(X_n)$, when $q_2/q_1^2$ is small or equal to $1$ and $n=4$.}
\label{fig:distinguished_basis}
\end{figure}

\appendix
\section{Conjecture $\mathcal{O}$ for $\mathbb{P}_{\mathbb{P}^n}(\mathcal{O}\oplus \mathcal{O}(n-1))$.}
In the appendix, we investigate Conjecture $\mathcal{O}$ for  $X_\Delta=\mathbb{P}_{\mathbb{P}^n}(\mathcal{O}\oplus \mathcal{O}(n-1))$,
which is again a  toric Fano manifold of Picard number two and of    Fano index two.

When $n=1$, $X_\Delta=\mathbb{P}^1\times \mathbb{P}^1$. When $n=2$, $X=Bl_{\mathrm{pt}}\mathbb{P}^3$ is the blow-up of $\mathbb{P}^3$ of a point, and  has been studied in \cite{Yang}. Now we assume $n\geq 3$. The arguments are similar to and slightly more involved than that for $X_n$. We include the details for two reasons: (a) Property $\mathcal{O}$ for $X_\Delta$ of higher Fano index can also be addressed with this method; (b) the eigenvalues of $\hat c_1$ of maximum modulo could all belong to $\mathbb{C}\setminus \mathbb{R}$.

The mirror superpotential $f$ of $X_\Delta$ is given by
\begin{equation} f(\textbf{x})=x_1+x_2+\cdots+x_n+x_{n+1}+\frac{(x_{n+1})^{n-1}}{x_1x_2\cdots x_{n}}+\frac{1}{x_{n+1}}.
\end{equation}
By solving $\frac{\partial f}{\partial x_i}=0$,  we have $x:=x_1=\cdots=x_n$, $y:=x_{n+1}$, $x,y\in \mathbb{C}^\times$,
  $$y^{n-1}=x^{n+1},\quad
        1+(n-1)\frac{y^{n-2}}{x^n}-\frac{1}{y^2}=0.$$
\subsubsection*{Case: $n$ is odd, with
$m:=\frac{n-1}{2}\geq 1$} We have $(x^{m+1}+y^m)(x^{m+1}-y^m)=0$, resulting in the parameterization $(x, y)=\pm(t^m,t^{m+1})$ together with the constraint
\[h_1(t):=t^{2m+2}+(2m)t^{2m+1}=1.\]
At   every root $t$ of $h_1(t)-1$, we have $g_1(t)=\Tilde{g}_1(t)=\check{g}_1(t)$, where
\begin{equation*}
 g_1(t):=(2m+1)t^m+t^{m+1} +t^{m}+\frac{1}{t^{m+1}} ,\quad  \Tilde{g}_1(t):=2t^m+\frac{2}{t^{m+1}},\quad \check{g}_1(t):=2t^{m+1}+(4m+2)t^{m}.
\end{equation*}
Moreover, all the $2n+2=2(2m+2)$ critical values of $f$ (with multiplicity counted) are precisely given by $\pm g_1(t)$ at these roots. That is, it suffices to  consider the following optimization problem in nonlinear programming.

\begin{prob}\label{NLP2}
 $\textnormal{Maximize}\quad   \big| g_1(t)\big|   \quad \textnormal{subject to}\quad h_1(t)=1,\quad t\in \mathbb{C}^\times.$
  \end{prob}
 With similar arguments to $h(t)$ for $X_n$, we see that every root of $h_1(t)-1$ is of multiplicity one, and $h_1(t)-1$ has exactly two simple real roots $a_{1\pm}$ with
\[a_{1+}\in (0,1), \qquad a_{1-}\in(-2m-1,-2m).\]
Noting the critical points of $f$ are of the form $\pm(t^m,\cdots,t^m,t^{m+1})$, for $f$ we have
\begin{prop}
    $\mathbf{x}_{\rm con}=(a_{1+}^m,\cdots,a_{1+}^m,a_{1+}^{m+1})$, hence $T_{\rm con}=g_1(a_{1+})$.
\end{prop}

\begin{prop}\label{ann2}
   All the roots of $h_1(t)-1$ but $a_{1-}$ are in the annulus $N_{[a_{1+},1]}$. Moreover, we have $\Tilde{g}_1(a_{1+})=\max_{t\in N_{[a_{1+},1]}}|\Tilde{g}_1(t)|.$

\end{prop}
\begin{proof}
By comparing with $(2m)t^{2m+1}$ and using Rouch\'e's theorem,   we conclude that $h_1(t)-1$ has $2m+1$ zeros in $D_1$. By the maximum modulus principle,   $\mathring D_{a_{1+}}$ contains no zeros of $h_1(t)-1$. The first statement follows by noting $a_{1-}\in (-2m-1,-2m)$.  The argument for the second statement is the same as that for Lemma \ref{maxonann1}.
\end{proof}

The nonlinear programming Problem \ref{NLP2} can be answered   by the following.
\begin{prop} \label{estimate2} Let $n$ be odd with $m=\frac{n-1}{2}\geq 1$. Take $\alpha\in \mathbb{C}\setminus\{a_{1+}, a_{1-}\}$ with $h_1(\alpha)=1$. We have
 \begin{enumerate}
 \item $|g_1(\alpha)|<g_1(a_{1+})<4m+4$.
    \item $2(2m)^m-1<(-1)^mg_1(a_{1-})<2(2m)^m$.
    \item If $m\ge 2$,  then $|g_1(a_{1-})|>g_1(a_{1+})$; if  $m=1$,   $g_1(a_{1+})>|g_1(a_{1-})|$.
\end{enumerate}
\end{prop}

\begin{proof}
Statement (1) follows by  the same arguments as in Proposition \ref{estimate1}. \par

   Viewing $h_1, \check{g}_1$ as   real functions in   $[-2m-1,-2m]$, we have  \[0<(-1)^m\check g_1'(x)=(-1)^m((2m+2)x+4m^2+2m)x^{m-1}<-((2m+2)x+4m^2+2m)x^{2m}=-h_1'(x).\]
  Hence,  $(-1)^mg_1(a_{1-})=(-1)^m\check g_1(a_{1-})<(-1)^m\check g_1(-2m)= 2(2m)^m$ and  $(-1)^mg_1(a_{1-})>(-1)^m\check g_1(-2m)-1=2(2m)^m-1$, by noting
  \[(-1)^m(\check g_1(-2m)-\check g_1(a_{1-}))=\int_{a_{1-}}^{-2m}(-1)^m\check g_1'(x) dx<\int_{a_{1-}}^{-2m}-h_1'(x) dx=h_1(a_{1-})-h_1(-2m)=1.\]
  Thus statement (2) holds. \par

   For $m\geq 2$,  $|g_1(a_{1-})|>2(2m)^m-1>4m+4>g_1(a_{1-})$.  For $m=1$, we have $a_{1+}>{7\over 10}$ and
   $$g_1(a_{1+})=\check g_1(a_{1+})>\check g_1({7\over 10})=2\cdot ({7\over 10})^2+6\cdot {7\over 10}
   >2\cdot2^1>|g_1(a_{1-})|>2\cdot2^1-1=3.$$
   Therefore statement (3) holds.
\end{proof}

\subsubsection*{Case: $n$ is even} $n-1$ and $n+1$ are coprime, so    $k_1(n-1)+k_2(n+1)=1$ for some integers $k_1, k_2$. We can take the parameterization $(x, y) =(t^{n-1}, t^{n+1})$ (by setting $t=y^{k_2}x^{k_1}$), resulting in the constraint
\begin{equation*}
    h_0(t):=t^{2n+2}+(n-1)t^{2n}=1.
\end{equation*} \par
Every root $\alpha$ of $h_0(t)-1$ is of multiplicity one, and   $g_0(\alpha)=\Tilde{g}_0(\alpha)=\check{g}_0(\alpha)$, where
\begin{equation*}
    g_0(t):=n t^{n-1}+t^{n+1}+t^{n-1}+\frac{1}{t^{n+1}}, \quad \Tilde{g}_0(t):=2t^{n-1}+\frac{2}{t^{n+1}},\quad \check{g}_0(t):=2t^{n+1}+(2n)t^{n-1}.
\end{equation*}
Moreover, all the $(2n+2)$   critical values of $f$, with multiplicity counted, are given by $g_0(t)$ at these   roots.
 It suffices to  consider the following optimization problem.

\begin{prob}\label{NLP3}$\textnormal{Maximize}\quad   \big| g_0(t)\big|   \quad \textnormal{subject to}\quad h_0(t)=1,\quad t\in \mathbb{C}^\times.$
  \end{prob}
With similar arguments to $h_1(t)-1$, we see that $u^{n+1}+(n-1)u^n-1$ consists of two real roots, which are both of multiplicity one and belong to $(0, 1)$ and $(-n+1, -n+{3\over 2})$ respectively.  Hence,
$h_0(t)-1$  has  four roots $\pm a_{0+}$  and $\pm a_{\rm I}$, with
  $$a_{0+}\in (0,1),\qquad  a_{\rm I}\in \iu\mathbb{R}\quad\mbox{and}\quad \mathrm{Im}(a_{\rm I}) \in (\sqrt{n-\frac{3}{2}},\sqrt{n-1}). $$
Moreover, $\pm a_{0+}$ are the only  real roots of $h_0(t)-1$.
Noting the critical points of $f$ are of the form $(t^{n-1},\cdots,t^{n-1},t^{n+1})$, we have
\begin{prop}
    $\mathbf{x}_{\rm con}=(a_{0+}^{n-1},\cdots,a_{0+}^{n-1} ,a_{0+}^{n+1})$, hence $T_{\rm con}=g_0(a_{0+})$.
\end{prop}
The following  is  parallel to Proposition \ref{ann2}, obtained  the same arguments.
\begin{prop}
     All the roots of $h_0(t)-1$ but $\pm a_{\rm I}$ are in the annulus $N_{[a_{0+},1]}$. Moreover,  we have
     $\Tilde{g}_0(a_{0+})=\max_{t\in N_{[a_{0+},1]}}|\Tilde{g}_0(t)|.$
\end{prop}

\begin{prop} \label{estimate3} Let $n\geq 4$ be even. Take any $\alpha\in \mathbb{C}\setminus\{\pm a_{0 +},\pm a_{ \rm I}\}$ with $h_0(\alpha)=1$. We have
 \begin{enumerate}
 \item $|g_0(\alpha)|<g_0(a_{0+})=-g_0(-a_{0+})<2n+2$.
    \item $2(n-1)^{\frac{n-1}{2}}<\iu^{n+1}g(a_{\rm I})=\iu^{n-1}g(-a_{\rm I})<2(n-1)^{\frac{n-1}{2}}+1$.
\end{enumerate}
\end{prop}
\begin{proof}
    It is evident that $g_0(-t)=-g_0(t)$, thus statement (1) follows by  the same arguments as in Proposition \ref{estimate2} (1).\par
     Consider $h_0(t), \check{g}_0(t)$ as    functions on the imaginary axis $\iu\mathbb{R}$ with $\mathrm{Im}(t)\in [\sqrt{n-\frac{3}{2}},\sqrt{n-1}]$. Denote $v:=\mathrm{Im}(t)$. Taking derivative with respect to $v$, we  have
     \begin{equation*}
     \begin{aligned}
          0&>\iu^{n+1}\check g_0'(v\iu)=\iu^{2n}(2n(n-1)-(2n+2)v^2)v^{n-2}\\
     &>\iu^{2n}(2n(n-1)-(2n+2)v^2)v^{2n-1}=h_0'(v \iu).
     \end{aligned}
     \end{equation*}

  Hence,  $\iu^{n+1}g_0(a_{\rm I})=\iu^{n+1}\check{g}_0(a_{\rm I})>\iu^{n+1}\check g_0(\sqrt{-n+1})= 2(n-1)^{\frac{n-1}{2}}$    and  $\iu^{n+1}g_0(a_{\rm I})<\iu^{n+1}\check g_0(\sqrt{-n+1})+1=2(n-1)^{\frac{n-1}{2}}+1$, by noting
  \begin{equation*}
      \begin{aligned}
        &  \iu^{n+1} (\check g_0(a_{\rm I})-\check g_0(\sqrt{-n+1}))\\
  =&-\int^{\sqrt{n-1}}_{\mathrm{Im}(a_{\rm I})} \iu^{n+1} \check g_0'(v\iu) dv\\
  <&-\int^{\sqrt{n-1}}_{\mathrm{Im}(a_{\rm I})} h_0'(v\iu) dv
  =h_0(a_{\rm I})-h_0(\sqrt{-n+1})=1.
      \end{aligned}
  \end{equation*}
 Thus statement (2) holds.
\end{proof}

\begin{thm}\label{thmconjO2}
  Conjecture $\mathcal{O}$ holds for  $\mathbb{P}_{\mathbb{P}^n}(\mathcal{O}\oplus \mathcal{O}(n-1))$ if and only if either $n$ is odd or $n=2$.
\end{thm}
\begin{proof}
    If $n=1$, then  Conjecture $\mathcal{O}$ holds by \cite{HKLY}.  If $n=2$, then $\mathbb{P}_{\mathbb{P}^n}(\mathcal{O}\oplus \mathcal{O}(1))=Bl_{\rm pt}{\mathbb{P}^3}$, and Conjecture $\mathcal{O}$ holds by \cite{Yang}.
 Assume $n\geq 3$ now. By Proposition \ref{eigncrit} and Propositions \ref{estimate2}, \ref{estimate3}, Property $\mathcal{O}$ (1) holds if and only if $n$ is odd.  So does Property $\mathcal{O}$ (2) by noting that  $\mathbb{P}_{\mathbb{P}^n}(\mathcal{O}\oplus \mathcal{O}(n-1))$ is  {Fano index two}.
\end{proof}

\begin{remark} \label{rmkmoreeg}
  Our analysis for the   superpotential works for more examples.  Toric Fano manifolds of Picard number two are classified in \cite{Klei}, and are the following projectivization of splitting bundles over projective spaces,    $\mathbb{P}_{\mathbb{P}^{s}}\big(\mathcal{O}\oplus\bigoplus_{i=1}^r\mathcal{O}(a_i)\big),$
where $r, s\geq 1, 0\leq a_1\leq \cdots \leq a_r$, and $\sum_{i=1}^ra_i\leq s$.
 Our method can   put the analysis for the special case $Bl_{\mathbb{P}^{r-1}}\mathbb{P}^n=\mathbb{P}_{\mathbb{P}^{n-r}}(\mathcal{O}(1)\oplus\bigoplus_{i=1}^r\mathcal{O})$ in \cite{Yang} uniformly into the framework of nonlinear programming, and works for instance for the toric Fano 4-fold with ID 71 in \cite{Obro2}, which  has Picard number five.
\end{remark}


\begin{thebibliography}{99}

\bibitem[Abou]{Abou} M. Abouzaid,\,{\it Homogeneous coordinate rings and mirror symmetry for toric varieties}, Geom. Topol. 10
(2006), 1097--1157.

\bibitem[AGIS]{AGIS}M. Abouzaid, S. Ganatra, H. Iritani, and N. Sheridan,\,{\it
The gamma and Strominger-Yau-Zaslow conjectures: a tropical approach to periods},
Geom. Topol. 24 (2020), no. 5, 2547--2602.

\bibitem[AGV]{Arnold-Gusein-Zade-VarchenkoII}
V.~I.~Arnold, S.~M.~Gusein-Zade, A.~N.~Varchenko,
\emph{Singularities of differentiable maps, Vol. II, Monodromy and asymptotics of integrals},
Monographs in Mathematics, Vol.83,
Birkh\"{a}user Boston, Inc., Boston, MA, 1988.

\bibitem[Au]{Auroux} D. Auroux,\,{\it Mirror symmetry and $T$-duality in the complement of an
anticanonical divisor}, J. G\"okova Geom. Topol. GGT 1 (2007), 51--91.

\bibitem[Baty]{Baty}V.V Batyrev,\,{\it Quantum cohomology rings of toric manifolds},
Ast\'erisque(1993), no. 218, 9--34.
 \bibitem[BFSS]{BFSS}L. Bones, G. Fowler, L. Schneider  and R. Shifler,\,{\it Conjecture $\mathcal{O}$ holds for some horospherical varieties of Picard rank $1$}, Involve 13 (2020), no. 4, 551--558.

\bibitem[BoPo]{Bondal-Polishchuk}
A.~Bondal and A.~Polishchuk,
\emph{Homological properties of associative algebras: the method of helices},
Russian Academy of Sciences. Izvestiya Mathematics, 1994, 42:2, 219--260.






 \bibitem[Chan20]{Chan20} K. Chan, \,{\it SYZ mirror symmetry for toric varieties},
 Adv. Lect. Math. (ALM), 47
 International Press, Somerville, MA, 2020, 1--32.

\bibitem[ChLe]{ChLe} K. Chan, N.C. Leung,\, {\it Mirror symmetry for toric Fano manifolds via SYZ transformations},
Adv. Math. 223 (2010), no. 3, 797--839.

\bibitem[Chen]{ZihongChen:exp}
Zihong Chen, \emph{On the exponential type conjecture}, to appear J. Amer. Math. Soc.; preprint at arXiv:2409.03922

\bibitem[Cheo]{Che} D. Cheong,\,{\it Quantum multiplication operators for Lagrangian and orthogonal Grassmannians}, J. Algebraic Combin., 45 (2017), no.4, 1153--1171.

 \bibitem[ChLi]{ChLi}D. Cheong and C. Li,\,{\it On the conjecture $\mathcal{O}$ of  {GGI} for {$G/P$}}, Adv. Math.  {306} (2017), 704--721.


\bibitem[ChOh]{ChOh} C.-H. Cho,   Y.-G. Oh,\,{\it Floer cohomology and disc instantons of Lagrangian torus fibers in toric Fano
manifolds}, Asian J. Math. 10 (2006), no. 4, 773--814

\bibitem[Cho25]{Chow} C. H. Chow,  \,{\it 
Gamma conjecture I for flag varieties},  arXiv: math.AG/2501.13221. 

\bibitem[CoKa]{Coates-Kasprzyk:database} 
T.~Coates, A.~M.~Kasprzyk, 
\emph{Databases of quantum periods for Fano manifolds}, 
Sci.~Data 9, 163 (2022).

\bibitem[CCGGK]{CCGGK:mirrorsymmetry}
T.~Coates, A.~Corti, S.~Galkin, V.~Golyshev, A.~Kasprzyk, 
\emph{Mirror symmetry and Fano manifolds}, 
European Congress of Mathematics (Krak\'ow, 2-7 July, 2012), November 2013 (824 pages), 
pp.285--300

\bibitem[CCGGK16]{CCGGK:quantum period}
T.~Coates, A.~Corti, S.~Galkin, V.~Golyshev, A.~Kasprzyk, \emph{Quantum periods for 3-dimensional Fano manifolds},
Geom. Topol. 20 (2016), no. 1, 103-256.

\bibitem[CCIT]{CCIT:toric_stacks_MS}
T.~Coates, A.~Corti, H.~Iritani, H-H.~Tseng,
\emph{Hodge-theoretic mirror symmetry for toric stacks},
J.~Differential Geom. 114 (2020),
no.1, pp.41-115.

\bibitem[CoxKa]{CoKa}D.A. Cox and S. Katz,\,{Mirror symmetry and algebraic geometry}, Mathematical Surveys and Monographs, 68. American Mathematical Society, Providence, RI, 1999.

\bibitem[CLS]{CLS}
D.~Cox, J.~Little, and H.~Schenck,\, {Toric varieties}, Graduate Studies in Mathematics, 124. American Mathematical Society, Providence, RI, 2011.



\bibitem[DuSa]{Douai-Sabbah:I}
A.~Douai, C.~Sabbah,
\emph{Gauss-Manin systems, Brieskorn lattices and Frobenius structures.  I.}
Ann.\ Inst.\ Fourier (Grenoble), 53 (2003), pp.1055--1116.

\bibitem[Du1]{Du1998}B. Dubrovin, \emph{Geometry and analytic theory of Frobenius manifolds}, In {\em Proceedings of the International Congress of Mathematicians,Vol. II (Berlin, 1998)}, 315--326.

\bibitem[Dub98]{Dubrovin:Painleve}
B.~Dubrovin,
\emph{Painlev\'e transcendents in two-dimensional topological field theory},
in The Painlev\'e Property, CRM Ser.\ Math.\ Phys.,
Springer, New York, 1999, 287-412.


\bibitem[DvdK]{DvdK:constantterms} 
J.~J.~Duistermaat and W.~van~der~Kallen, 
\emph{Constant terms in powers of a Laurent polynomial}, 
Indag.~Math.~(N.S.) {\bf 9} (1998), no.~2, 221--231. 



\bibitem[JDEv14]{Evans:qcoh_twistor} 
J.~D.~Evans, 
\emph{Quantum cohomology of twistor spaces and their Lagrangian submanifolds}, 
J.~Differential Geom.~96.3 (2014), pp.353--397. 

\bibitem[LCEv10]{LCEv10} L. C. Evans,\,{Partial differential equations (2nd ed.)},  American Mathematical Society,
 Providence, RI, 2010.


\bibitem[Fan20]{Fang:centralcharges}
B.~Fang,
\emph{Central charges of $T$-dual branes for toric varieties},
Trans.\ Amer.\ Math.\ Soc.\ 373(2020), no.6, pp.3829--3851.


\bibitem[FLTZ]{FLTZ} B. Fang, C.-C. M. Liu, D. Treumann and E. Zaslow, \,{\it
T-duality and homological mirror symmetry for toric varieties}, Adv. Math.229(2012), no.3, 1875--1911.

\bibitem[FWZ]{FWZ} B. Fang, J. Wang and Y. Zhou, \,{\it Mirror symmetric Gamma conjecture for del Pezzo surfaces}, preprint at arXiv: math.AG/2309.02154.

\bibitem[FaZh]{Fang-Zhou:GammaII}
B.~Fang, P.~Zhou,
\emph{Gamma II for toric varieties from integrals on T-dual branes and homological mirror symmetry},
2019, arXiv:1903.05300.

\bibitem[Fed]{Fed} M. V. Fedoryuk, \,{\it Asymptotic methods in analysis}, (Russian) Current problems of mathematics. Fundamental directions, Vol. 13 (Russian), 93-210.
Itogi Nauki i Tekhniki[Progress in Science and Technology]
Akad. Nauk SSSR, Vsesoyuz. Inst. Nauchn. i Tekhn. Inform., Moscow, 1986.

\bibitem[FiPa]{Fine-Panov}
J. Fine and D. Panov. \emph{Hyperbolic geometry and non-K\"ahler manifolds with trivial canonical
bundle}, Geom.~Topol.~14.3 (2010), 1723--1763.


 \bibitem[FOOO]{FOOO} K. Fukaya, Y.-G. Oh, H. Ohta and K. Ono,\,{\it Lagrangian Floer theory on compact toric manifolds. I},
Duke Math. J. 151 (2010), no. 1, 23--174.

\bibitem[Ga11]{Galkin:Split}
S.~Galkin, {\it Split notes (on non-commutative mirror symmetry)}, based on a lecture at Homological Mirror Symmetry and Category Theory workshop in Split, July 2011, IPMU 12-0112, \url{https://research.ipmu.jp/ipmu/sysimg/ipmu/884.pdf} 

\bibitem[Ga14]{Gal}S. Galkin, {\it The conifold point}, preprint at arXiv: math.AG/1404.7388.

\bibitem[GaGo]{GaGo}S. Galkin   and V.   Golyshev,\,{\it  Quantum cohomology of Grassmannians, and cyclotomic fields}, Russian Math. Surveys 61 (2006), no. 1, 171--173.

\bibitem[GGI14]{GGI-arXiv1} S. Galkin, V. Golyshev and H. Iritani, \,{\it Gamma classes and quantum cohomology of Fano manifolds: Gamma conjectures}, preprint at arXiv:1404.6407v1 (2014).

\bibitem[GGI]{GGI}S. Galkin, V. Golyshev and H. Iritani,\,{\it Gamma classes and quantum cohomology of Fano manifolds: Gamma conjectures}, Duke Math. J. 165 (2016), no. 11, 2005--2077.

\bibitem[GaIr]{GaIr}S. Galkin and  H. Iritani, \, {\it Gamma conjecture via mirror symmetry}, Primitive Forms and Related Subjects, Advanced Studies in Pure Mathematics Vol. 83, 55--115 (2019).

\bibitem[GaMi]{GaMi}S. Galkin and G. Mikhalkin,\,{\it Singular symplectic spaces and holomorphic membranes},
Eur. J. Math. 8 (2022), no. 3, 932-951.

\bibitem[GKR]{GKR}
J.~Gilman, I.~Kra and R.~Rodriguez,\, { Complex analysis:
in the spirit of Lipman Bers},
Graduate Texts in Mathematics, 245. Springer New York, New York, 2012.





\bibitem[Giv95]{Gi1}A. Givental, {\em Homological geometry and mirror symmetry},  Proceedings of the International Congress of Mathematicians, Z\"urich, 1994, Birkh\"auser, Basel, 1995, vol 1, 472--480.
\bibitem[Giv96]{Giv96} A. Givental, {\em Equivariant Gromov-Witten invariants}, Internat. Math. Res. Notices(1996), no. 13, 613--663.


\bibitem[Giv98]{Gi2}A. Givental, {\em A mirror theorem for toric complete intersections}, Topological field theory, primitive forms and related topics, (Kyoto, 1996),  141--175, Progr. Math., 160, Birkh\"auser, Boston, MA, 1998.

\bibitem[GoZa]{GoZa}V. Golyshev and D. Zagier,\,{\it Proof of the gamma conjecture for Fano 3-folds with a Picard lattice of rank one},   Izv. Math. 80 (2016), no. 1, 24--49.

\bibitem[GoWo]{Gonzalez-Woodward:tmmp}
E.~Gonzalez and C.~Woodward,
\emph{Quantum cohomology and toric minimal model programs},
Adv.~Math., Volume 353 (2019) 591--646.


\bibitem[HeSe]{Hertling-Sevenheck:nilpotent}
C.~Hertling and C.~Sevenheck,
\emph{Nilpotent orbits of a generalization of Hodge structures},
J.~Reine Angew.~Math.\ 609 (2007), 23--80.

\bibitem[HoVa]{HoVa}K. Hori and C. Vafa, {\em Mirror symmetry}, arXiv:hep-th/0002222v3.
\bibitem[Hoso]{Hoso} S. Hosono,\,{\it Central charges, symplectic forms, and hypergeometric series in local mirror symmetry}, Mirror symmetry. V, 405--439,
AMS/IP Stud. Adv. Math., 38, Amer. Math. Soc., Providence, RI, 2006.
 \bibitem[HKTY]{HKTY} S.  Hosono, A. Klemm, S. Theisen and S.-T. Yau,\,{\it Mirror symmetry, mirror map and applications to complete intersection Calabi-Yau spaces}, Nuclear Phys. B 433 (1995), no. 3, 501--552.
 \bibitem[HKLS]{HKLS} J. Hu, H. Ke, C. Li and L. Song,\,{\it   On the quantum cohomology of blow-ups of four-dimensional quadrics}, Acta. Math. Sin.-English Ser. 40, 313--328 (2024).

 \bibitem[HKLSu]{HKLSu} J. Hu, H. Ke, C. Li and Z. Su,\,{\it   On Galkin's lower bound conjecture},    Asian J. Math. 28 (2024), no. 4, 609--616.
 \bibitem[HKLSu26]{HKLS26} J. Hu, H. Ke, C. Li and Z. Su,\,{\it   Gamma conjecture II via global Gamma-I}, arXiv: math.AG/2606.07418.

\bibitem[HKLY]{HKLY} J. Hu, H. Ke, C. Li and T. Yang, \,{\it Gamma conjecture I for del Pezzo surfaces},  Adv. Math. 386 (2021), Paper No. 107797, 40 pp.

\bibitem[Hug24]{Hug24} K. Hugtenburg, \,{\it On the quantum differential equations for a family of non-Kahler monotone symplectic manifolds}, preprint at	arXiv:2402.10867.



\bibitem[Iri09]{Iri1}H. Iritani,\,{\it An integral structure in quantum cohomology and mirror symmetry for toric orbifolds}, Adv. Math. 222 (2009), no. 3, 1016--1079.

\bibitem[Iri11]{Iri2}H. Iritani,\,{\it Quantum cohomology and periods}, Ann. Inst. Fourier (Grenoble) 61(2011), no. 7, 2909--2958.

\bibitem[Iri20a]{Iri20}H. Iritani,\,{\it Quantum D-modules of toric varieties and oscillatory integrals}, Adv. Lect. Math. (ALM), 47, International Press, Somerville, MA, 2020, 131--147.

\bibitem[Iri20b]{Iritani:discrepant}
H.~Iritani,
\emph{Global mirrors and discrepant transformations for toric Deligne-Mumford stacks},
SIGMA Symmetry Integrability Geom.~Methods Appl.~16 (2020), Paper No. 032, 111 pages.

\bibitem[Iri23a]{Iritani:Gamma_quantum}
H.~Iritani,
\emph{Gamma classes and quantum cohomology},
ICM -- International Congress of Mathematicians. Vol. IV. Sections 5-8, 2552-2575.
EMS Press, Berlin, 2023.

\bibitem[Iri23b]{Iritani:monoidal}
H.~Iritani,
\emph{Quantum cohomology of blowups}, to appear in Kyoto J. Math.;  
arXiv:2307.13555.

\bibitem[Iri23c]{Iri23} H. Iritani, \,{\it Mirror symmetric Gamma conjecture for Fano and Calabi-Yau manifolds}, preprint at arXiv: math.AG/2307.15940.

\bibitem[IrKo]{Iritani-Koto:projective}
H.~Iritani, Y.~Koto
\emph{Quantum cohomology of projective bundles}, to appear in Geom. Topol.;
arXiv:2307.03696.

\bibitem[KMM]{KMM:MMP}
Y.~Kawamata, K.~Matsuda, K.~Matsuki,
\emph{Introduction to the minimal model program},
Advanced Studies in Pure Mathematics 10, 1987,
Algebraic Geometry, Sendai, 1985,
pp.283--360.

\bibitem[KKP]{KKP:Hodge}
L.~Katzarkov, M.~Kontsevich, and T.~Pantev,
\emph{Hodge theoretic aspects of mirror symmetry},
In: From Hodge theory to integrability and TQFT $tt^*$-geometry, pp.87--174,
Proc.~Sympos.~Pure Math., 78,
Amer.~Math.~Soc., Providence, RI, 2008.

 \bibitem[Ke]{Ke}H.-Z. Ke,\,{\it On Conjecture $\mathcal{O}$ for projective complete intersections}, Int. Math. Res. Not. IMRN 2024, no. 5, 3947-3974.

 \bibitem[Klei]{Klei} P. Kleinschmidt,\,{\it A classification of toric varieties with few generators}, Aequationes Math.35(1988), no.2-3, 254--266.


\bibitem[Kou76]{Kouchnirenko:Newton}
A.~G.~Kouchnirenko,
\emph{Poly\`{e}dres de Newton et nombres de Milnor},
Invent.\ Math.\ 32 (1976), 1--31.

\bibitem[LaLi]{Lawler-Limic:random_walk}
G.~F.~Lawler and V.~Limic,
\emph{Random walk: a modern introduction}, Cambridge Studies in Advanced Mathematics, 123, Cambridge Univ. Press, Cambridge, 2010.

\bibitem[LLW]{Lee-Lin-Wang:flip}
Y.~P.~Lee, H.~W.~Lin, C.~L.~Wang,
\emph{Quantum flips {I}},
Local model, Integrability, quantization, and geometry II.
Quantum theories and algebraic geometry, Proc.~Sympos.~Pure Math., vol.103.2,
Amer.~Math.~Soc.,
Providence, RI, 2021, pp.303--352.

\bibitem[LMS]{LMS} C. Li, L. C. Mihalcea and R. Shifler,\,{\it Conjecture $\mathcal{O}$ holds for the odd symplectic Grassmannian}, Bull. Lond. Math. Soc. 51 (2019), no. 4, 705--714.
\bibitem[Libg]{Libg}A. Libgober,\,{\it Chern classes and the periods of mirrors}, Math. Res. Lett. 6 (1999), no. 2, 141--149.

\bibitem[Lu]{Lu}R. Lu,\,{\it  The $\widehat \Gamma$-genus and a regularization of an $\mathbb{S}^1$-equivariant Euler class}, J. Phys. A 41 (2008), no. 42, 425204 (13 pp).

\bibitem[Mal]{Malgrange:DE_book}
B.~Malgrange,
\emph{\'{E}quations diff\'{e}rentielles \`{a} coefficients polynomiaux},
Progress in Math., vol.96, Birkh\"{a}user, Basel, Boston, 1991.




\bibitem[Nil]{Nils} N. Nilsson,\,{\it
Monodromy and asymptotic properties of certain multiple integrals},
Ark. Mat. 18 (1980), no. 2, 181--198.


 \bibitem[Obro]{Obro2} M.  {\O}bro, \,{\it Smooth toric Fano varieties of dimension less than 7}, http://www.grdb.co.uk/forms/toricsmooth.

\bibitem[OsTy]{OsTy} Y. Ostrover, I. Tyomkin, \,{\it On the quantum homology algebra of toric Fano manifolds}, Selecta Math. (N.S.) 15(2009), no.1, 121--149.


\bibitem[PoSe]{Pomerleano-Seidel:qconn}
D.~Pomerleano, P.~Seidel,
\emph{The quantum connection, Fourier-Laplace transform, and families of A-infinity-categories}, Invent. math. 246, 405--577 (2026).

\bibitem[Prz]{Przyjalkowski:weak}
V.~Przyjalkowski, \emph{Hori-Vafa mirror models for complete intersections in weighted projective spaces and weak Landau-Ginzburg models},
Cent.~Eur.~J.~Math.~9 (2011), no.5, 972--977.

\bibitem[ReSe]{Reichelt-Sevenheck:logarithmic}
T.~Reichelt, C.~Sevenheck,
\emph{Logarithmic Frobenius manifolds, hypergeometric systems and quantum $D$-modules},
J.\ Algebraic Geom.\ 24 (2015), 201--281.

\bibitem[Rua]{Ruan:crepant}
Y.~Ruan,
\emph{The cohomology ring of crepant resolutions of orbifolds},
Gromov-{W}itten theory of spin curves and orbifolds,
Contemp.~Math., 403, pp.117--126,
Amer.~Math.~Soc., Providence, RI, 2006.


\bibitem[Sab99]{Sabbah:hypergeometric}
C.~Sabbah, \emph{Hypergeometric periods for a tame polynomial},
Portugalia Mathematicae, 63 (2006), no.2, pp.173--226.
A short version without proofs in: C.~R.~Acad.~Sci.~Paris Sr.~I Math.~328 (1999),
no.7, pp.603--608.

\bibitem[Sab07]{Sabbah:isomonodromic}
C.~Sabbah,
\emph{Isomonodromic deformations and Frobenius manifolds, An Introduction},
French ed., Universitext, Springer-Verlag London, Ltd., London; EDP Sciences,
Les Ulis, 2007.
\bibitem[San21]{Sanda:toric_degeneration}
F.~Sanda, \emph{Mirror symmetry of Fano manifolds via toric degenerations}, Talk given at ``Online Workshop on Mirror symmetry and Related Topics, Kyoto 2021'' held at Kyoto University, December 2021.

\bibitem[SaSh]{SaSh}F. Sanda, Y. Shamoto,\,{\it An analogue of Dubrovin's conjecture}, Ann. Inst. Fourier (Grenoble) 70 (2020), no. 2, 621--682.


\bibitem[She]{She}N. Sheridan,\,{\it On the Fukaya category of a Fano hypersurface in projective space}, Publ. Math. Inst. Hautes \'Etudes Sci. 124(2016), 165--317.

\bibitem[Sib]{Sibuya:DE_complex_domain}
Y. Sibuya, {\it Linear differential equations in the complex domain: problems of analytic continuation}, translated from the Japanese by the author,
Translations of Mathematical Monographs, 82, Amer. Math. Soc., Providence, RI, 1990.

\bibitem[Ton]{Ton}D. Tonkonog,\,{\it String topology with gravitational descendants, and periods of LandauGinzburg potentials}, arXiv: 1801.06921.


\bibitem[Wang]{Wang} J.~Wang, \,{\it The Gamma conjecture for tropical curves in local mirror symmetry}, preprint at arXiv: math.AG/2011.01729.

\bibitem[Was]{Wasow} W.~Wasow,
\emph{Asymptotic expansions for ordinary differential equations},
Pure and Applied Mathematics, Vol.XIV Interscience Publishers John Wiley \& Sons, Inc.,
New York-London-Sydney 1965.

\bibitem[Wit]{With} C. Withrow,\,{\it The moment graph for Bott-Samelson varieties and applications to quantum
cohomology}, preprint at arxiv: math.AG/1808.09302.

\bibitem[Yang]{Yang} Z. Yang,\,{\it Gamma conjecture I for blowing up $\mathbb{P}^n$ along $\mathbb{P}^r$},  preprint at arXiv: math.AG/2202.04234.



\end{thebibliography}
\end{document}